\documentclass[11pt]{amsart}

\usepackage[margin=1.02in]{geometry}
\usepackage{amsmath,amssymb,amsthm,mathtools}
\usepackage{enumitem}
\usepackage{microtype,xurl}
\usepackage[hidelinks]{hyperref}
\usepackage[nameinlink,noabbrev]{cleveref}

\numberwithin{equation}{section}

\hypersetup{
  pdftitle={Finite-Time Singularities of the KÃ¤hler--Ricci Flow on Fano Bundles},
  pdfauthor={Wangjian Jian and Jian Song},
  pdfsubject={Spectral splitting at the limiting time, Type-I curvature, sharp fibre diameter, and cylindrical tangent spaces at limiting points},
  pdfkeywords={Kahler--Ricci flow, spectral splitting, Type-I curvature, fibre diameter, collapsing fibration, tangent space at a limiting point, Fano bundle, projective bundle}
}

\newtheorem{theorem}{Theorem}[section]
\newtheorem{maintheorem}[theorem]{Theorem}
\newtheorem{proposition}{Proposition}[section]
\newtheorem{lemma}{Lemma}[section]
\newtheorem{corollary}{Corollary}[section]

\newtheorem{conjecture}{Conjecture}[section]
\theoremstyle{definition}
\newtheorem{definition}{Definition}[section]
\theoremstyle{remark}
\newtheorem{remark}{Remark}[section]

\crefname{theorem}{Theorem}{Theorems}
\Crefname{theorem}{Theorem}{Theorems}
\crefname{maintheorem}{Theorem}{Theorems}
\Crefname{maintheorem}{Theorem}{Theorems}
\crefname{conjecture}{Conjecture}{Conjectures}
\Crefname{conjecture}{Conjecture}{Conjectures}
\crefname{proposition}{Proposition}{Propositions}
\Crefname{proposition}{Proposition}{Propositions}
\crefname{lemma}{Lemma}{Lemmas}
\Crefname{lemma}{Lemma}{Lemmas}
\crefname{corollary}{Corollary}{Corollaries}
\Crefname{corollary}{Corollary}{Corollaries}
\crefname{claim}{Claim}{Claims}
\Crefname{claim}{Claim}{Claims}
\crefname{definition}{Definition}{Definitions}
\Crefname{definition}{Definition}{Definitions}
\crefname{remark}{Remark}{Remarks}
\Crefname{remark}{Remark}{Remarks}

\newcommand{\C}{\mathbb C}
\newcommand{\R}{\mathbb R}
\newcommand{\PP}{\mathbb P}
\newcommand{\Reg}{\operatorname{Reg}}
\newcommand{\Sing}{\operatorname{Sing}}
\newcommand{\Ric}{\operatorname{Ric}}
\newcommand{\Rc}{\operatorname{Rc}}
\newcommand{\Rm}{\operatorname{Rm}}
\newcommand{\Vol}{\operatorname{Vol}}
\newcommand{\Var}{\operatorname{Var}}
\newcommand{\inj}{\operatorname{inj}}
\newcommand{\tr}{\operatorname{tr}}
\newcommand{\supp}{\operatorname{supp}}
\newcommand{\dist}{\operatorname{dist}}
\newcommand{\Area}{\operatorname{Area}}
\newcommand{\Span}{\operatorname{span}}
\newcommand{\Dom}{\operatorname{Dom}}
\newcommand{\cR}{\mathcal R}
\newcommand{\cE}{\mathcal E}
\newcommand{\Id}{\operatorname{Id}}
\newcommand{\HZanalyticref}[1][]{
  \hyperref[thm:general-kahler-shrinker-package]{%
    Theorem~\ref*{thm:general-kahler-shrinker-package}#1}}

\title[K\"ahler--Ricci Flow on Fano Bundles]
{Finite-Time Singularities of the K\"ahler--Ricci Flow on Fano Bundles}
\author{Wangjian Jian$^*$}
\address{$^*$ Institute of Mathematics, Academy of Mathematics and Systems Science, Chinese Academy of Sciences, Beijing, 100190, China}
\email{wangjian@amss.ac.cn}
\author{Jian Song$^\dagger$}
\address{$^\dagger$ Department of Mathematics, Rutgers University, Piscataway, NJ 08854, USA}
\email{jiansong@math.rutgers.edu}
\thanks{Wangjian Jian is supported in part by NSFC grants 12422103,
12201610, 12371058, and 12288201; BJNSF grant JR25002; and National Key
R\&D Program of China grants 2023YFA1009900 and 2021YFA1003100. Jian
Song is supported in part by the National Science Foundation
grant DMS-2505575.}
\date{}

\begin{document}
\raggedbottom
\begin{abstract}
We study collapsing finite-time singularities of the unnormalized
K\"ahler--Ricci flow on Fano bundles arising in the analytic
minimal model program.  For a Fano bundle $X^n\rightarrow Y^m$, we prove a maximal splitting theorem by the K\"ahler-Ricci flow that every
tangent space at a fixed limiting point splits globally as
\((\C^m,g_{\rm E},J_0)\times(Z',d',J')\).  If the fibre has complex dimension one, we prove a
global Type-I bound for the full curvature tensor and show that the
ambient and intrinsic diameters of every fibre are uniformly comparable
to \(\sqrt{T-t}\). Furthermore, every tangent flow at any fixed limiting point is the round
shrinking cylinder \(\C^m\times\PP^1\).
\end{abstract}

\maketitle

\tableofcontents

\section{Introduction}

The K\"ahler--Ricci flow provides an analytic approach to classifications of K\"ahler spaces and their singularities.  The analytic minimal model program
(AMMP) was proposed and developed by Song--Tian as a metric counterpart
of the minimal model program with scaling
\cite{TianAMMP,SongTianThrough}.  At finite time, the flow is expected
to realize algebraic contractions and fibre-space structures, as well
as the contractions preceding flips, while continuing canonically on
the resulting mildly singular spaces.  On the long-time side,
Song--Tian constructed canonical
limiting metrics and measures in the semiample setting
\cite{SongTianSurfaces,SongTianCanonical} and proved uniform scalar
curvature bounds \cite{SongTianScalar}.  Their weak-flow theory gives
existence and uniqueness on projective varieties with klt
singularities and continuation through divisorial contractions and
flips whenever the corresponding algebraic operations exist
\cite{SongTianThrough}.  Canonical surgical contractions were established
in concrete settings by Song--Weinkove
\cite{SongWeinkoveContraction}.

Let \(X^n\) be a smooth projective manifold, and let
\begin{equation}\label{eq:ammp-flow-intro}
 \frac{\partial}{\partial t}\omega(t)=-\Ric(\omega(t)),
 \qquad 0\le t<T<\infty,
\end{equation}
be a maximal unnormalized K\"ahler--Ricci flow.  Suppose that its
limiting class
\[
 [\omega_T]=[\omega_0]-2\pi T c_1(X)
\]
is semiample.  After Stein factorization, this class induces a
holomorphic map with connected fibres
\begin{equation}\label{eq:ammp-map-intro}
                       \pi:X^n\longrightarrow Y^m,
\end{equation}
which we call the limiting morphism.  The dimension of \(Y\) gives the
trichotomy
\[
 \begin{array}{ccl}
 m=0      &:& \text{total collapse},\\
 0<m<n    &:& \text{collapse along a Fano fibration},\\
 m=n      &:& \text{the noncollapsing, birational case}.
 \end{array}
\]

Jian--Song--Tian formulated a geometric picture for these three types of
finite-time singularities \cite[Section~1]{JianSongTian}.  When \(m=0\),
the flow becomes extinct, and Perelman's estimates yield the
corresponding parabolically scale-invariant scalar-curvature, diameter,
entropy, and noncollapsing controls in the unnormalized flow
\cite{PerelmanEntropy,SesumTian}.  When \(0<m<n\), the Fano directions
collapse toward \(Y\), and parabolic blowups over a point of the base
are expected to converge to complete shrinking models carrying a
Mori-fibration structure.  When \(m=n\), the limiting morphism is
birational and the singularity should realize a divisorial contraction
or flip, with shrinking and expanding models joined through a common
asymptotic cone.

Jian--Song--Tian conjectured a Type-I scalar-curvature bound for every
finite-time solution and the Type-I diameter rate for the collapsing
fibres.  They proved Li--Yau and Harnack estimates for weighted Ricci
potentials, local Type-I scalar-curvature estimates relative to Ricci
vertices, and subsequential compactness of suitably based rescalings to
ancient flows whose time slices are normal analytic varieties.  For
Fano bundles, their estimates also control the scalar curvature and
diameter on a fibre selected by a Ricci vertex in each prescribed
relatively compact region of the regular base
\cite{JianSongTian}.

The compactness and structure theories of Bamler and their K\"ahler
refinements provide the framework for these limits
\cite{BamlerCompactness,BamlerStructure,HallgrenJian,
HallgrenJianSongTian}.  Hallgren--Zhang showed that the singular
K\"ahler--Ricci shrinkers arising in the relevant noncollapsed settings
carry polarized Fano-fibration structures and are normal klt varieties
\cite{HallgrenZhang}.  In complex dimension two,
Conlon--Hallgren--Ma proved that every noncollapsed finite-time
singularity is Type I \cite{ConlonHallgrenMa}.  Projective and Fano
bundles in the collapsing case were studied by
Song--Sz\'ekelyhidi--Weinkove and Fu--Zhang
\cite{FuZhang,SongSzekelyhidiWeinkove}.  Independently, Xu--Zhang
proved the Type-I cylindrical model for ruled surfaces
\cite{XuZhangRuled}.  Chen--Hallgren--Lavoyer described
shrinker--cone--expander transitions in the noncollapsing surface case
under a holomorphic-realization hypothesis
\cite{ChenHallgrenLavoyer}.  Related foundational work includes
uniqueness, degeneration, compactness, and regularity results for
K\"ahler--Ricci solitons and flows
\cite{TianZhuUniqueness,TianZhangDegeneration,
TianQSZhangCompactness}, and the relation between
Perelman's entropy and convergence of the Fano flow
\cite{TianZhangZhangZhuEntropy}.

For the normalized K\"ahler--Ricci flow on Fano manifolds,
Tian--Zhang proved the Hamilton--Tian conjecture in complex dimensions
\(n\le3\), while Bamler and Chen--Wang obtained arbitrary-dimensional
proofs by different compactness methods
\cite{BamlerHamiltonTian,ChenWangWeakCompactness,
TianZhangRegularity}.

The purpose of this paper is to study the intermediate case
\begin{equation}\label{eq:ammp-intermediate}
                         0<m<n.
\end{equation}
We prove a full-base-rank splitting theorem for projective bundles of
arbitrary positive relative dimension.  In relative complex dimension
one, we also prove Type-I curvature and sharp fibre-diameter estimates.

\subsection{A conjectural AMMP picture for smooth Fano fibrations}
\label{sec:intro-main-conjecture}

Suppose that the limiting morphism
\[
                  \pi:X^n\longrightarrow Y^m,\qquad 0<m<n,
\]
is a smooth projective morphism and that every fibre
\[
                         F_y=\pi^{-1}(y)
\]
is a smooth Fano manifold of complex dimension \(n-m\).  The fibres need
not be mutually biholomorphic or locally holomorphically trivial, so
their complex structures may vary with \(y\).

Fix \(q\in Y\), let \(\xi=(\mu_t)_{t<T}\) be a limiting point based at
\(q\) in the sense of
\cref{def:limiting-point,def:type-I-label}, and choose
\(\tau_i\searrow0\).  Consider the parabolic
rescalings
\begin{equation}\label{eq:intro-conjectural-rescaling}
 g_i(s)=\tau_i^{-1}g(T+\tau_i s),
 \qquad -T/\tau_i\le s<0.
\end{equation}
A subsequential pointed \(\mathbb F\)-limit of the corresponding
metric-flow pairs will be called a \emph{Type-I-scale blowup at \(\xi\)
based at \(F_q\)}.  This term refers only to the scale
\(\tau_i^{-1}\) and does not assume a Type-I curvature bound.  We call
\(F_q\) the central Fano manifold.  Here the word \emph{central} refers
to the chosen point \(q\), not to a singularity of the fibre.

The natural metric condition is K-polystability rather than strict
K-stability.  By the Yau--Tian--Donaldson theorem, it is equivalent to
the existence of a K\"ahler--Einstein metric on a smooth Fano manifold
\cite{BermanKPolystability,ChenDonaldsonSunIII,TianKStability}.  It allows
positive-dimensional automorphism groups and includes \(\PP^1\).

\begin{conjecture}[K-polystable smooth fibres]
\label{conj:intro-kpolystable-fibres}
Assume that every fibre \(F_y\) is K-polystable with respect to
\(-K_{F_y}\).  Then the finite-time K\"ahler--Ricci flow is Type I:
\begin{equation}\label{eq:intro-conjectural-type-I}
 \sup_{0\le t<T}\sup_{x\in X}
 (T-t)|\Rm(g(t))|_{g(t)}(x)<\infty.
\end{equation}
Moreover, fix \(q\in Y\), and let
\(\omega_{q,{\rm KE}}\in2\pi c_1(F_q)\) be a normalized
K\"ahler--Einstein metric,
\[
                 \Ric(\omega_{q,{\rm KE}})=\omega_{q,{\rm KE}}.
\]
Write \(g_{q,{\rm KE}}=2g_{\omega_{q,{\rm KE}}}\), so that
\(\Ric(g_{q,{\rm KE}})=\tfrac12g_{q,{\rm KE}}\).  After forgetting the
conjugate heat measures, every Type-I-scale blowup at \(\xi\) based at
\(F_q\) is isomorphic, as a K\"ahler metric flow, to
\begin{equation}\label{eq:intro-conjectural-central-product}
 \left(
   \C^m\times F_q,\,
   J_0\oplus J_{F_q},\,
   g_{\rm E}+(-s)g_{q,{\rm KE}}
 \right)_{s<0}.
\end{equation}
Thus the residual factor is the central fibre with its own complex
structure.
\end{conjecture}

The product in \eqref{eq:intro-conjectural-central-product} is unique
only up to product isometry.  The K\"ahler--Einstein representative is
unique up to the identity component of \(\operatorname{Aut}(F_q)\)
\cite{BandoMabuchi}; no canonical marking of the factors is asserted.

\subsection{Full-base-rank splitting at the limiting time}

The metric-flow slice \(M_T\) at the limiting time is recalled in
\cref{sec:prelim}.  A limiting point is an element \(\xi\in M_T\),
represented by a probability conjugate heat flow \((\mu_t)_{t<T}\); it
is fixed when chosen independently of the blowup scales.  This is made
precise in \cref{def:limiting-point}.  Its base point \(q_\pi(\xi)\) and
Type-I quantitative localization are defined in
\cref{def:type-I-label}, while tangent spaces at \(\xi\) are defined in
\cref{def:tangent-space-at-limiting-point}.
In the semiample setting, these properties follow from
\cref{lem:automatic-label}; they are not additional assumptions.

\begin{definition}
\label{def:intro-tangent-space-at-limiting-point}
Let \(\xi=(\mu_t)_{t<T}\) be a fixed limiting point.  Given
\(\tau_i\searrow0\), set
\[
 g_i(s)=\tau_i^{-1}g(T+\tau_i s),\qquad
 \mu_{i,s}=\mu_{T+\tau_i s},\qquad -T/\tau_i\le s<0.
\]
A \emph{tangent space at \(\xi\)} is any subsequential pointed
\(\mathbb F\)-limit of these rescaled metric-flow pairs.  Thus the term
denotes the full ancient pointed metric-measure flow, rather than an
ordinary linear tangent space, and it does not include a uniqueness
assertion.
\end{definition}

We will now study the collapsing behavior of Fano bundles. Let 
\begin{equation}\label{fb}
\pi: X^n \rightarrow Y^m
\end{equation} be a Fano bundle between two projective manifolds $X$ and $Y$, where the fibre is a Fano manifold. In our definition of a fibre bundle, we always assume that the bundle is Zariski local trivial, i.e., there exists an open  covering $\{ U_i\} $ of $Y$ in the Zariski topology such that the fibre bundle is trivial over each $U_i$.

\begin{maintheorem}
\label{thm:intro-splitting}
Let $\pi: X \rightarrow Y$ be a Fano bundle between two projective manifolds with $0<\dim_\mathbb{C} Y =m< \dim_\mathbb{C}X =n$.    Suppose the
limiting cohomology class of  \eqref{eq:ammp-flow-intro} is the pullback of 
 a K\"ahler class \(Y\) by $\pi$.  Let \(\xi\) be a
fixed limiting point based at \(q\in Y\).  Every tangent space at \(\xi\) is a complete K\"ahler-Ricci shrinker that 
admits a global holomorphic isometric splitting
\begin{equation}\label{eq:intro-minimal-splitting}
 (Z,d,J)\cong
 (\C^m,g_{\rm E},J_0)\times(Z',d',J'),
\end{equation}
where $Z'$ is a normal analytic normal variety with klt singularities. 
\end{maintheorem}

\begin{remark}
The splitting theorem also holds locally over the regular part of an isotrivial Fano
fibration satisfying the Zariski-open-product property.  Suppose that a
limiting point is based at \(q\in Y^{\rm prod}\), the
Zariski-open-product locus defined in \cref{sec:prelim}.  Choose a
Zariski-open product region \(U_{\rm prod}\ni q\) and
\(q\in V\Subset U_{\rm prod}\).  Fu-Zhang's horizontal estimate gives
the required two-sided quotient control on \(V\), and the same splitting
follows; see
\cref{prop:zariski-product-horizontal,cor:local-full-rank-package}.
\end{remark}

\subsection{Type-I curvature and fibre diameter in relative dimension one}

The first theorem below proves the Type-I
curvature bound for rank-two projective bundles.  The second identifies
the exact collapse rate of every fibre.

\begin{maintheorem}
\label{thm:intro-type-I} Let $\pi: X\rightarrow Y$ be a $\mathbb{P}^1$-bundle between two projective manifolds. Suppose the
limiting cohomology class of  \eqref{eq:ammp-flow-intro} is the pullback of 
 a K\"ahler class on \(Y\) by $\pi$. Then there exit 
 \(0< c< C<\infty\) such that for any $0\leq t <T$ and $q\in Y$, we have 
 \smallskip
\begin{equation}\label{ti:eq:P1-regular-TypeI-front}
 \sup_{x\in X}
 (T-t)|\Rm(g(t))|_{g(t)}(x)\le C, 
\end{equation} 
\begin{equation}\label{ti:eq:P1-fibre-diameter-front}
 c\sqrt{T-t}
 \le \operatorname{diam}_{(X,g(t))}X_q
 \le \operatorname{diam}_{(X_q,g(t)|_{X_q})}X_q
 \le C\sqrt{T-t},
\end{equation}
where $X_q=\pi^{-1}(q)$ is the fibre of $X$ over $q$, $\operatorname{diam}_{(X,g(t))}X_q$ is the extrinsic diameter by the ambient distance on \(X\) and $\operatorname{diam}_{(X_q,g(t)|_{X_q})}X_q$ is 
 the intrinsic diameter by the induced metric on \(X_q\).
\end{maintheorem}
 
We would like to remark that the same conclusions hold for $\mathbb{P}^1$-fibrations away from singular fibres.

\begin{corollary}
\label{ti:cor:projective-P1}
Under the hypotheses of \cref{thm:intro-type-I}, the underlying
K\"ahler metric flow of every tangent space at every fixed limiting
point, after forgetting its distinguished conjugate heat measures, is
the round shrinking cylinder
\begin{equation}\label{ti:eq:intro-cylinder}
 \left(\C^m\times\PP^1,\,
 J_0\oplus J_{\rm std},\,
 g_{\rm E}+(-s)g_{S^2}\right)_{s<0},
 \qquad \Rc(g_{S^2})=\tfrac12g_{S^2}.
\end{equation}
\end{corollary}

We briefly describe the main arguments.  The splitting theorem for a
fixed limiting point is stated and proved in
\cref{sec:spectral-splitting}; see \cref{sp:prop:fullrank}.  For each
tangent space, we equip its time-\((-1)\) model with the canonical
Friedrichs drift operator, prove the sharp half-spectral gap, and show that the
limits of the magnified base coordinates have parallel differentials.
Fu--Zhang's compact-local horizontal estimate prevents a loss of base
rank.  A finite Zariski-local trivializing cover gives the corresponding
global control for projective bundles.

In relative dimension one, classification of the residual normal curve
reduces the time-\((-1)\) model to the Gaussian space or the round
cylinder.  The packing and exact-level arguments exclude the Gaussian
alternative for every tangent space at a fixed limiting point.  The
fixed-fibre argument in \cref{sec:surface-prototype} also gives
curvature, packing, and pointed-ball estimates and supplies the
bounded-drift input used in the moving-pole argument.  The two stages
have different quantifiers:
\[
  \text{for every fixed \(q\), there is \(C(q)\)}
  \qquad\text{and}\qquad
  \text{there is one \(C_K\) for all \(q\in K\)}.
\]
The two-scale contradiction in \cref{ti:sec:general-base} converts the
fixed-fibre estimates into the compact-local bound.  An entropy-based
alternative, using the same cylindrical classification at fixed
limiting points but not the fixed-fibre curvature constants, is given in
the companion paper \cite{JianSongEntropy}.

The paper is organized as follows.  In \cref{sec:prelim} we collect the
analytic and geometric inputs for finite-time blowups.  In
\cref{sec:spectral-splitting} we prove the full-base-rank splitting
theorem.  In \cref{sec:surface-prototype} we classify the
relative-dimension-one models, exclude the Gaussian alternative, and
establish the fixed-fibre estimates.
Finally, in \cref{ti:sec:general-base} we prove the compact-local Type-I
curvature bound and the sharp two-sided fibre-diameter estimate.

\bigskip
\noindent\textbf{Acknowledgements.} The authors thank Professor Gang Tian
for his continued encouragement and support.  We also thank Max
Hallgren for helpful communications. The results on  $\mathbb{P}^1$-bundles in this paper have overlaps with the recent results of Tongxin Xu and Zhenlei Zhang \cite{XuZhangRuled}. We would also like to acknowledge the assistance of  GPT-5.6-sol in the development and exposition of
\cref{sec:spectral-splitting},  which brought
the work of He-Ou \cite{HeOu} to the authors' attention, although the
proof presented in the paper does not rely on the results in \cite{HeOu}.

\bigskip

\section{Preliminaries and analytic inputs}\label{sec:prelim}

\subsection{Flow, fibration, limiting points, and quotient conventions}
\label{sec:fano-fibration-setup}

We first fix the Ricci-flow notation used throughout the paper.  For a
compact real Ricci flow satisfying \(\partial_tg=-2\Rc\), put
\[
 \Box=\partial_t-\Delta_{g(t)},
 \qquad
 \Box^*=-\partial_t-\Delta_{g(t)}+R_{g(t)}.
\]
For \(s<t\), its heat kernel \(K(x,t;y,s)\) is characterized by
\begin{equation}\label{surf:eq:heat-kernel}
 \Box_{(x,t)}K(x,t;y,s)=0,
 \qquad
 \Box^*_{(y,s)}K(x,t;y,s)=0,
\end{equation}
with the usual delta limits.  The associated conjugate heat measure,
potential, and pointed Nash entropy are
\begin{align}
 d\nu_{x,t;s}(y)
 &=K(x,t;y,s)\,dV_{g(s)}(y)
 =(4\pi\tau)^{-n}e^{-f_{x,t}(y,s)}\,dV_{g(s)}(y),
 \qquad \tau=t-s,
 \label{surf:eq:conjugate-measure}\\
 \mathcal N_{x,t}(\tau)
 &=\int_X f_{x,t}(\,\cdot\,,t-\tau)\,
 d\nu_{x,t;t-\tau}-n .
\end{align}
Here and below \(n=\dim_{\C}X\), so the real dimension is \(2n\).

For probability measures \(\mu_1,\mu_2\) on \((X,g)\), set
\[
 \Var_g(\mu_1,\mu_2)
 :=\iint_{X\times X}d_g(x,y)^2\,d\mu_1(x)d\mu_2(y),
 \qquad
 \Var_g(\mu):=\Var_g(\mu,\mu),
\]
and, for every positive integer \(D\), put
\(H_D=\frac{(D-1)\pi^2}{2}+4\).  At backward scale
\(\tau>0\), an \(H_D\)-center of \(\mu\) is a point \(p\) satisfying
\[
 \int_Xd_g(p,y)^2\,d\mu(y)\le H_D\tau.
\]
Such a center exists when \(\Var_g(\mu)\le H_D\tau\).  A Ricci flow is
called \(H_D\)-concentrated when
\[
 \Var_{g(s)}(\nu_{x_1,t;s},\nu_{x_2,t;s})
 \le d_{g(t)}(x_1,x_2)^2+H_D(t-s)
 \qquad(s<t).
\]
This is Bamler's convention.  The variance estimate is
\cite[Corollary~3.8 and equation~(3.9)]{BamlerHeat}, and the center
consequence is \cite[Proposition~3.12]{BamlerHeat}.
Our ordinary parabolic notation is
\[
 P(x_0,t_0;r,-a,b)
 =B_{g(t_0)}(x_0,r)\times
 \bigl([t_0-a,t_0+b]\cap I\bigr)
\]
for a flow on \(I\); if \(b=0\), we write \(P(x_0,t_0;r,-a)\).

We use the unnormalized K\"ahler--Ricci flow
\begin{equation}\label{ti:eq:intro-KRF}
 \partial_t\omega(t)=-\Ric(\omega(t)),\qquad 0\le t<T<\infty,
\end{equation}
and the standard real Ricci-flow metric
\begin{equation}\label{eq:standard-real-metric}
 g(t)=2g_{\omega(t)},\qquad \partial_tg=-2\Rc(g).
\end{equation}
If
\(\omega=\sqrt{-1}\,h_{i\bar j}dz^i\wedge d\bar z^j\), we use the
complex Laplacian
\(\Delta_\omega=h^{i\bar j}\nabla_i\nabla_{\bar j}\) and the analogous
complex trace for scalar curvature.  With the preceding standard-real
normalization,
\begin{equation}\label{sp:eq:normalization-conventions}
 \Delta_g=\Delta_\omega,\qquad R_g=R_\omega,
 \qquad dV_g=2^n\frac{\omega^n}{n!}.
\end{equation}
Every parabolic blowup at the limiting time uses a sequence
$\tau_i\searrow0$ and
\begin{equation}\label{eq:rescaling}
 g_i(s)=\tau_i^{-1}g(T+\tau_i s),
 \qquad -T/\tau_i\le s<0.
\end{equation}
Let $(X^n,J_X)$ be connected, smooth, and projective, where $J_X$ is the
fixed complex structure underlying the K\"ahler--Ricci flow.  We assume that the
limiting class is semiample and induces a proper surjective morphism with
connected fibres
\begin{equation}\label{ti:eq:intro-semiample-map}
 \pi:X\longrightarrow Y\subset\PP^N,
 \qquad [\omega_0]-2\pi T c_1(X)=\pi^*[\eta],
\end{equation}
where $Y$ is normal and projective and $\eta$ is the restriction of a
fixed Fubini--Study form.  In tensor inequalities we also write $\eta$
for its induced Hermitian metric on $Y_{\rm reg}$.  The notation $d_\eta$
and $B_\eta$ refers to the distance and balls obtained by restricting the
fixed ambient Fubini--Study distance to $Y$; on relatively compact subsets
of $Y_{\rm reg}$ this is uniformly equivalent to the intrinsic distance of
the associated smooth K\"ahler metric.  We work in the collapsing range
\begin{equation}\label{ti:eq:intro-fibration-range}
 0<m:=\dim_\C Y<n,
\end{equation}
and write $Y^{\rm rv}\subset Y_{\rm reg}$ for the maximal Zariski-open
regular-value locus and $\Delta_\pi=Y\setminus Y^{\rm rv}$.
For the results that require horizontal control, we make a separate
local hypothesis.  Write $Y^{\rm prod}\subset Y^{\rm rv}$ for the locus
of points $q$
for which there is a Zariski-open neighbourhood $U\ni q$ and a
biregular product isomorphism
\begin{equation}\label{eq:zariski-product-locus}
                 \pi^{-1}(U)\cong U\times F
\end{equation}
over $U$, with \(U\subset Y^{\rm rv}\).  A compact set
$K\Subset Y^{\rm rv}$ is said to satisfy the
\emph{Zariski-open-product property} if it is covered by finitely many
such product regions.  This is an additional assumption only in the
statements where it is explicitly invoked; it is not a consequence of
\eqref{ti:eq:intro-fibration-range}.
Restriction to a regular fibre $F_y$ gives
\begin{equation}\label{ti:eq:fibre-class}
 [\omega(t)]|_{F_y}=2\pi(T-t)c_1(F_y).
\end{equation}
Thus in relative complex dimension one every regular fibre is $\PP^1$,
and its area in the standard real convention is $8\pi(T-t)$.

\begin{definition}
\label{def:limiting-point}
A \emph{limiting point} at time \(T\) is an element
\(\xi\in M_T\) of Bamler's metric-flow slice at the limiting time
\cite[Definition~2.32 and the paragraph following
equation~(2.34)]{BamlerStructure}.  We represent it by a probability
conjugate heat flow
\[
 \xi=(\mu_t)_{t<T},\qquad d\mu_t=v_t\,dV_{g(t)},
\]
where
\[
 \partial_tv_t=-\Delta_{g(t)}v_t+R_{g(t)}v_t,
 \qquad v_t>0,
 \qquad \int_Xv_t\,dV_{g(t)}=1.
\]
The representation has the following two properties of the limiting slice:
\begin{enumerate}[label=\textup{(P\arabic*)},leftmargin=3em]
\item \(\xi\) is a subsequential limit of point conjugate heat flows
whose pole times increase to \(T\);
\item for \(D=2n\),
\begin{equation}\label{eq:limiting-variance}
 \Var_{g(t)}(\mu_t)
 :=\iint_{X\times X}d_{g(t)}(x,y)^2\,d\mu_t(x)d\mu_t(y)
 \le H_{2n}(T-t).
\end{equation}
\end{enumerate}
The limiting point is called \emph{fixed} when it is chosen before any
sequence of blow-up scales.  Every tangent argument below keeps this
same limiting point while the scales vary.
\end{definition}

For \(0<\tau\le T\), write
\begin{equation}\label{eq:limiting-point-potential}
 d\mu_{T-\tau}
 =(4\pi\tau)^{-n}e^{-f_{\xi,T}(\,\cdot\,,T-\tau)}
 \,dV_{g(T-\tau)}
\end{equation}
and define the pointed Nash entropy
\begin{equation}\label{eq:limiting-point-nash-entropy}
 \mathcal N_{\xi,T}(\tau)
 :=\int_X f_{\xi,T}(\,\cdot\,,T-\tau)\,d\mu_{T-\tau}-n.
\end{equation}
The convention is Bamler's
\cite[Definition~4.27]{BamlerStructure}.  By
\cite[the paragraph preceding Theorem~2.37 and Theorem~2.37]{BamlerStructure},
the limit exists and satisfies
\[
 \mathcal N_{\xi,T}(0)
 :=\lim_{\tau\searrow0}\mathcal N_{\xi,T}(\tau)
 \in(-\infty,0].
\]
The subtraction by \(n\) corresponds to real dimension \(2n\).

For point kernels, the variance estimate is
\cite[Corollary~3.8 and equation~(3.9)]{BamlerHeat}; property
\textup{(P2)} is obtained by passage to the limiting slice.
\begin{definition}
\label{def:tangent-space-at-limiting-point}
Let \(\xi=(\mu_t)_{t<T}\) be a fixed limiting point.  Given
\(\tau_i\searrow0\), set
\begin{equation}\label{sp:eq:tangent-space-at-limiting-point-definition}
 g_i(s)=\tau_i^{-1}g(T+\tau_i s),\qquad
 \mu_{i,s}=\mu_{T+\tau_i s},\qquad -T/\tau_i\le s<0.
\end{equation}
A \emph{tangent space at \(\xi\)} is any subsequential pointed limit
of these metric-flow pairs in Bamler's \(\mathbb F\)-topology
\cite[Definitions~5.4 and~5.6]{BamlerCompactness}.  We use the same term
for the full ancient pointed metric-measure flow and, when no ambiguity
can arise, for its time-\((-1)\) shrinking-soliton model.  No uniqueness
is included in the definition.  The limiting point is fixed before the
scales are selected.  If
\(d\mu_t=v_t\,dV_{g(t)}\), then the density of \(\mu_{i,s}\) relative
to \(g_i(s)\) is
\(\widetilde v_{i,s}=\tau_i^n v_{T+\tau_i s}\).
\end{definition}

Let $U\subset Y^{\rm rv}$ be open and $x\in\pi^{-1}(U)$.  The quotient
metric on $T_{\pi(x)}Y$ and quotient cometric on $T^*_{\pi(x)}Y$ are
\[
 S_{t,x}(\zeta,\bar\zeta)
 :=\inf_{d\pi_x(v)=\zeta}g(t)_x(v,\bar v),
 \qquad
 Q_{t,x}^{\alpha\bar\beta}
 :=g(t)^{i\bar j}\pi_i^\alpha\overline{\pi_j^\beta},
\]
where $(g^{i\bar j})$ is the inverse $(1,0)$ complex block of the
standard real metric $g(t)$.  Thus $Q=d\pi\,g^{-1}d\pi^*$ is understood
in this complex-block sense.  For every real target function $\varphi$,
the exact real-normalization identities are
\begin{equation}\label{eq:real-Q-normalization}
 |d(\varphi\circ\pi)|_{g(t)}^2
 =2Q_t(\partial\varphi,\bar\partial\varphi),
 \qquad
 \Delta_{g(t)}(\varphi\circ\pi)
 =2Q_t^{\alpha\bar\beta}\varphi_{\alpha\bar\beta}.
\end{equation}
Complex-valued target functions are treated componentwise.  The tensors
$S_t$ and $Q_t$ are Hermitian tensors on the pullbacks $\pi^*TY$ and $\pi^*T^*Y$
and depend on the point $x$ in the fibre.  We suppress $x$ only when an
estimate is uniform on $\pi^{-1}(U)$.  The restriction of $d\pi_x$ to
$(\ker d\pi_x)^{\perp_{g(t)}}$ identifies the two quotient tensors and
gives $Q_{t,x}=S_{t,x}^{-1}$.

\begin{definition}
\label{def:quotient-conditions}
Let \(U\subset Y^{\rm rv}\) be open.  We say that
\((\mathrm S)_U\) holds if there are \(t_U<T\) and \(c_U>0\), independent
of \(x\in\pi^{-1}(U)\) and \(t\in[t_U,T)\), such that
\begin{equation}\label{eq:S}
 S_{t,x}\ge c_U\eta,
 \qquad\text{equivalently}\qquad
 Q_{t,x}\le c_U^{-1}\eta^{-1}.
\end{equation}
We say that \((\mathrm Q)_U\) holds if there are \(t_U<T\) and
\(C_U<\infty\), with the same uniformity, such that
\begin{equation}\label{eq:Q}
 S_{t,x}\le C_U\eta,
 \qquad\text{equivalently}\qquad
 Q_{t,x}\ge C_U^{-1}\eta^{-1}.
\end{equation}
When an estimate holds for all \(0\le t<T\), as will be the case for
the estimates proved below, we use the same notation.  Condition
\((\mathrm S)_U\) supplies the upper differential bound for pulled-back
base coordinates, whereas \((\mathrm Q)_U\) is the lower differential
bound that prevents loss of base rank.
\end{definition}
\begin{lemma}
\label{lem:ambient-schwarz}
\label{ti:lem:ambient-schwarz}
Let \(\iota:Y\hookrightarrow\PP^N\) be the fixed projective embedding
and put \(F=\iota\circ\pi\).  Then
\begin{equation}\label{eq:ambient-schwarz}
 \operatorname{tr}_{\omega(t)}F^*\omega_{\rm FS}\le C,
 \qquad
 \omega(t)\ge C^{-1}F^*\omega_{\rm FS},
 \qquad 0\le t<T.
\end{equation}
Consequently \((\mathrm S)_U\) holds for every
\(U\Subset Y^{\rm rv}\); indeed the estimate is global on the locus on
which the quotient metric is defined.
\end{lemma}

This is the Song--Tian parabolic Schwarz lemma in the finite-time AMMP
form used by Jian--Song--Tian
\cite[Lemma~4.1, equation~(4.10)]{JianSongTian}.  Their statement is normalized to
\(T=1\); the displayed estimate follows for general \(T\) by the
constant parabolic rescaling.  On a compact regular target chart, the
restriction of the ambient Fubini--Study metric is uniformly equivalent
to any fixed smooth representative denoted by \(\eta\).

\begin{corollary}
\label{cor:basefunction}\label{ti:cor:basefunction}
Assume $(\mathrm S)_U$.  Let \(\varphi\) be a smooth real function on
\(Y_{\rm reg}\) whose first and second derivatives have compact support
in \(U\), and suppose that \(\varphi\circ\pi\) extends smoothly to
\(X\).  Then
\[
 |d(\varphi\circ\pi)|_{g(t)}+
 |\Delta_{g(t)}(\varphi\circ\pi)|\le C_\varphi.
\]
\end{corollary}

\begin{definition}
\label{def:squared-distance-barrier}
For \(p=[w]\in\PP^N\), define the ambient chord-square barrier
\begin{equation}\label{eq:ambient-chord-square}
 \widehat\rho_p([z])
 :=1-\frac{|\langle z,w\rangle|^2}{|z|^2|w|^2}.
\end{equation}
It is a smooth nonnegative function, vanishes precisely at \(p\), and
the family \(\{\widehat\rho_p\}_{p\in\PP^N}\) has uniform first- and
second-derivative bounds.  Moreover, for uniform \(0<c<C<\infty\),
\begin{equation}\label{eq:ambient-barrier-comparison}
 c\,d_{\rm FS}(y,p)^2\le\widehat\rho_p(y)
 \le C\,d_{\rm FS}(y,p)^2,
 \qquad y,p\in\PP^N.
\end{equation}
Indeed, in the standard normalization
\(\widehat\rho_p=\sin^2d_{\rm FS}(\,\cdot\,,p)\), and the
Fubini--Study diameter is finite.

Now let \(q\in U_0\Subset U_1\Subset Y^{\rm rv}\).  Compose
\(\widehat\rho_{\iota(q)}|_Y\) with a smooth nondecreasing function
that is the identity near zero and is constant once its argument leaves
a sufficiently small sublevel set contained in \(U_1\).  The resulting
function \(\rho_q\) is called a \emph{localized squared-distance
barrier at \(q\)}.  It is the restriction of a smooth ambient function,
has zero set \(\{q\}\), is uniformly comparable to
\(d_\eta(\,\cdot\,,q)^2\) near \(q\), is bounded below on
\(Y\setminus U_0\), and has its first two derivatives supported in
\(U_1\).
\end{definition}

\begin{definition}
\label{def:type-I-label}\label{ti:def:regularly-based}
A limiting point \(\xi=(\mu_t)_{t<T}\) is \emph{based at \(q\)} with
respect to the limiting morphism \(\pi\) if
\(\pi_*\mu_t\rightharpoonup\delta_q\).  Such a \(q\) is unique when it
exists, and we denote it by \(q_\pi(\xi)\).  It is
\emph{Type-I based at \(q\)} if there is \(C_\xi<\infty\) such that
\begin{equation}\label{eq:label-moment}
 \int_Xd_\eta(\pi(x),q)^2\,d\mu_t(x)\le C_\xi(T-t)
 \qquad\text{for every }0\le t<T.
\end{equation}
If \(q\in Y^{\rm rv}\) and \(\rho_q\) is a localized
squared-distance barrier from \cref{def:squared-distance-barrier},
compactness and the defining comparisons for \(\rho_q\) show that this
is equivalent to
\[
 \int_X\rho_q\circ\pi\,d\mu_t\le C'_\xi(T-t).
\]
For an arbitrary holomorphic map, qualitative basing alone need not
imply this rate.  For the semiample K\"ahler--Ricci flow considered here,
\Cref{lem:automatic-label} proves that every limiting point is based at
a unique point \(q_\pi(\xi)\) and is Type-I based there.  Thus
\eqref{eq:label-moment} is never an additional hypothesis below.  When
\(q_\pi(\xi)\in Y^{\rm rv}\), we simply say that \(\xi\) is based at
the regular value \(q_\pi(\xi)\).
\end{definition}

\pagebreak[3]
The following proposition is a compact-local consequence of the
weighted horizontal estimate of Fu--Zhang
\cite[Lemma~2.2 and its proof]{FuZhang}.
\nopagebreak[4]

\begin{proposition}
\label{prop:zariski-product-horizontal}\label{surf:hor}
Let \(q\in Y^{\rm rv}\), and suppose that a Zariski-open neighbourhood
\(U_{\rm prod}\ni q\) admits a biregular product trivialization
\[
 \pi^{-1}(U_{\rm prod})\cong U_{\rm prod}\times F
\]
over \(U_{\rm prod}\).  Then there is a principal Zariski-open product
neighbourhood \(U\ni q\), \(U\subset U_{\rm prod}\), such that for every
\(V\Subset U\) there are constants \(0<c_V\le C_V<\infty\), independent
of \(t<T\), for which
\begin{equation}\label{eq:zariski-product-horizontal-SQ}
 c_V\eta\le S_t\le C_V\eta
 \qquad\text{on }\pi^{-1}(V).
\end{equation}
Moreover, uniformly in \(f\in F\),
\begin{equation}\label{eq:FuZhang-local-slice}
 \omega(t)|_{V\times\{f\}}
 \le C_V\pi^*\eta|_{V\times\{f\}},
 \qquad 0\le t<T.
\end{equation}
Consequently, if \(K\Subset Y^{\rm rv}\) has the
Zariski-open-product property, then
\[
 c_K\eta\le S_t\le C_K\eta
 \qquad\text{on }\pi^{-1}(K)
\]
for constants independent of \(t<T\).
\end{proposition}

\begin{proof}
For the fixed projective embedding of \(Y\), choose \(k\) sufficiently
large and a section
\(s\in H^0(Y,\mathcal O_Y(k))\) whose zero divisor contains
\(Y\setminus(U_{\rm prod}\cap Y_{\rm reg})\) but misses \(q\).  Set
\(L=\mathcal O_Y(k)\) and equip it with the Hermitian metric induced by
the Fubini--Study metric.  Then
\[
 U:=Y\setminus(s=0)
\]
is a principal affine neighbourhood of \(q\), contained in
\(U_{\rm prod}\cap Y_{\rm reg}\).

Since \(U\) is smooth and affine, there are finitely many algebraic
vector fields \(W_1,\ldots,W_{N_1}\) that generate \(T^{1,0}U\).
Under \(\pi^{-1}(U)\cong U\times F\), let
\(\widetilde W_a=(W_a,0)\) be their product-horizontal lifts.  Since
\(\pi^{-1}(U)=X\setminus(\pi^*s=0)\) is a principal open subset of the
projective variety \(X\), localization for the coherent tangent sheaf
gives
\[
 \Gamma\bigl(\pi^{-1}(U),T^{1,0}X\bigr)
 =\varinjlim_{\ell}
 \Gamma\bigl(X,T^{1,0}X\otimes\pi^*L^\ell\bigr).
\]
There are only finitely many lifts, so one sufficiently large integer
\(\ell\) works for all of them: every
\[
 \mathcal V_a:=(\pi^*s)^\ell\widetilde W_a
\]
extends to a global holomorphic section of
\(T^{1,0}X\otimes\pi^*L^\ell\).

The curvature of the induced metric on \(L\) is a fixed multiple of
\(\eta\) on \(Y_{\rm reg}\).  The twisted-vector-field Bochner formula
along the K\"ahler--Ricci flow, followed by the global Schwarz estimate
\cref{lem:ambient-schwarz}, gives
\[
 (\partial_t-\Delta_{\omega(t)})|\mathcal V_a|^2
 \le C\,\tr_{\omega(t)}\pi^*\eta\,|\mathcal V_a|^2
 \le C'|\mathcal V_a|^2.
\]
The maximum principle on the compact manifold \(X\) therefore bounds
\(|\mathcal V_a|_{g(t)}\) uniformly for \(0\le t<T\).  This is a
twisted-vector-field variant of the weighted horizontal
maximum-principle argument of Fu--Zhang
\cite[Lemma~2.2 and its proof]{FuZhang}.

Now fix \(V\Subset U\).  Since \(|s|\) is bounded below on \(V\), all
\(\widetilde W_a\) are uniformly bounded on \(\pi^{-1}(V)\).  The
surjective bundle map
\[
 \C^{N_1}\times V\longrightarrow T^{1,0}V,\qquad
 (c_a)\longmapsto\sum_a c_aW_a
\]
has a right inverse whose norm is uniformly bounded on \(\overline V\).
Thus every \(\eta\)-unit vector at a point of \(V\) has a
product-horizontal lift of uniformly bounded \(g(t)\)-length.  The
lift lies in every slice \(V\times\{f\}\), uniformly in \(f\), proving
\eqref{eq:FuZhang-local-slice} and \(S_t\le C_V\eta\).  The reverse
bound is \((\mathrm S)_V\), supplied by
\cref{lem:ambient-schwarz}.  A finite refinement of a product cover of
\(K\), followed by taking the minimum and maximum of the resulting
constants, proves the last assertion.
\end{proof}

\begin{remark}[Projective bundles]
\label{rem:projective-from-product}
If \(X=\PP(E)\) with \(E\to Y\) holomorphic and \(Y\) projective, GAGA
makes \(E\) algebraic and hence Zariski locally trivial.  A finite
product cover and \cref{prop:zariski-product-horizontal} therefore give
\[
 c\eta\le S_t\le C\eta\qquad\text{on }X,\qquad 0\le t<T.
\]
This is the quotient estimate used in the projective-bundle theorem.
\end{remark}

\begin{lemma}
\label{lem:general-block-volume}
Put \(r=n-m\).  On every relatively compact Zariski product region on
which the horizontal estimate of
\cref{prop:zariski-product-horizontal} holds,
there are constants \(0<c'\le C'<\infty\) such that
\begin{equation}\label{eq:general-block-volume}
 c'\,\pi^*\eta^m\wedge\omega(t)^r
 \le \omega(t)^n
 \le C'\,\pi^*\eta^m\wedge\omega(t)^r.
\end{equation}
This follows from the block determinant identity and
\eqref{eq:zariski-product-horizontal-SQ}; compare
\cite[Lemma~9.2]{JianSongTian}.
\end{lemma}

\begin{corollary}
\label{cor:type-I-tube-volume}
Fix \(q\in V\), where \(V\Subset Y^{\rm rv}\) lies in a Zariski product
region on whose preimage \cref{lem:general-block-volume} applies.  If
\(t_i\nearrow T\), \(\tau_i=T-t_i\), and
\begin{equation}\label{eq:general-tube-definition}
 \Omega_i(R):=
 \pi^{-1}\!\left(B_\eta(q,R\sqrt{\tau_i})\right)
\end{equation}
then, for every fixed \(R<\infty\),
\begin{equation}\label{eq:general-tube-volume}
 \Vol_{\tau_i^{-1}g(t_i)}\Omega_i(R)\le C_R
\end{equation}
for all sufficiently large \(i\).
\end{corollary}

\begin{proof}
The upper bound in \eqref{eq:general-block-volume}, fibre integration,
\eqref{ti:eq:fibre-class}, and
\(\Vol_\eta(B_\eta(q,R\sqrt{\tau_i}))=O_R(\tau_i^m)\) give
\(\int_{\Omega_i(R)}\pi^*\eta^m\wedge\omega(t_i)^r
=O_R(\tau_i^{m+r})\).  Multiplication by \(\tau_i^{-n}\), with
\(n=m+r\), yields
\eqref{eq:general-tube-volume}.
\end{proof}

\subsection{Base localization}
\label{sec:base-localization}

This subsection derives from the AMMP Schwarz estimate the quantitative
base localization of point conjugate kernels, limiting points, and
their centers.  In particular, basing is automatic in the present
semiample setting, as is Type-I basing.

\begin{lemma}
\label{surf:lem:base-localization}
\label{lem:automatic-label}\label{ti:lem:label}
There is \(C_0<\infty\), depending only on the fixed geometric data, such
that the following three statements hold.
\begin{enumerate}[label=\textup{(\roman*)},leftmargin=2.2em]
\item For the point kernel based at \((x_0,t_0)\), every
\(H_{2n}\)-center \(z_t\) at time \(t<t_0\) and backward scale
\(t_0-t\) satisfies
\[
 z_t\in\pi^{-1}\!\left(
 B_\eta\bigl(\pi(x_0),C_0\sqrt{t_0-t}\bigr)\right).
\]
\item Every limiting point \(\xi=(\mu_t)_{t<T}\) has a unique point
\(q_\pi(\xi)\in Y\) such that
\[
 \pi_*\mu_t\rightharpoonup\delta_{q_\pi(\xi)},
 \qquad
 \int_Xd_\eta(\pi(x),q_\pi(\xi))^2\,d\mu_t(x)
 \le C_0(T-t).
\]
Every \(H_{2n}\)-center \(z_t\) of \(\mu_t\) at backward scale
\(T-t\) satisfies
\[
 z_t\in\pi^{-1}\!\left(
 B_\eta\bigl(q_\pi(\xi),C_0\sqrt{T-t}\bigr)\right),
 \qquad t<T.
\]
\item In either of the preceding settings, put
\[
 (\lambda_t,t_{\rm pole},q_{\rm pole})
 =\begin{cases}
 (\nu_{x_0,t_0;t},t_0,\pi(x_0)),&\text{in \textup{(i)}},\\
 (\mu_t,T,q_\pi(\xi)),&\text{in \textup{(ii)}}.
 \end{cases}
\]
If, for a fixed finite \(A\ge0\), a point \(z_t\) satisfies
\begin{equation}\label{eq:A-centre-bound}
 \int_Xd_{g(t)}(z_t,x)^2\,d\lambda_t(x)
 \le A(t_{\rm pole}-t),
\end{equation}
then
\[
 d_\eta(\pi(z_t),q_{\rm pole})
 \le C_0(A)\sqrt{t_{\rm pole}-t}.
\]
Thus the center conclusions in \textup{(i)} and \textup{(ii)} are the
special case \(A=H_{2n}\).
\end{enumerate}
\end{lemma}

\begin{proof}
A local point-kernel analogue of \textup{(i)} appears in
\cite[Lemma~6.6]{JianSongTian}; their Lemma~6.7 localizes minimizing
\(\mathcal L\)-geodesics using the \(W_1\)--\(\mathcal L\) estimate of
their Lemma~3.6.  Those results are formulated for local tubular
neighbourhoods.  We give a global ambient-barrier proof because it
yields the quantitative second-moment estimate uniformly in the pole
and therefore also permits passage to limiting points in
\textup{(ii)}.

Let \(F=\iota\circ\pi:X\to\PP^N\), where \(\iota:Y\hookrightarrow\PP^N\)
is the fixed projective embedding, and use the ambient chord-square
barriers \(\widehat\rho_p\) from
\cref{def:squared-distance-barrier}.  In particular, their first two
derivatives and the comparison with squared Fubini--Study distance are
uniform in \(p\).

The global ambient Schwarz estimate is
\begin{equation}\label{eq:base-localization-ambient-Schwarz}
 \operatorname{tr}_{g(t)}F^*g_{\rm FS}\le C.
\end{equation}
Because \(F\) is holomorphic, tracing
\(\partial\bar\partial(\widehat\rho_p\circ F)\) produces no second
derivative of \(F\).  The uniform ambient Hessian bound and
\eqref{eq:base-localization-ambient-Schwarz} therefore give
\begin{equation}\label{eq:base-localization-barrier-bounds}
 |d(\widehat\rho_p\circ F)|_{g(t)}
 +|\Delta_{g(t)}(\widehat\rho_p\circ F)|\le C
\end{equation}
for every \(p\in\PP^N\) and \(0\le t<T\).  This calculation is entirely
on the smooth ambient projective space, so it remains valid when
\(\pi(x)\) is a singular or critical target point.

First consider the point conjugate heat kernel
$\nu_{x_0,t_0;t}$ based at $(x_0,t_0)$, and put
$h=\widehat\rho_{F(x_0)}\circ F$.  Its gradient and Laplacian are bounded
by \eqref{eq:base-localization-barrier-bounds}.  To display the
duality calculation, write
$d\nu_{x_0,t_0;t}=v_t\,dV_{g(t)}$.  Since
\[
 \partial_tv_t=-\Delta_{g(t)}v_t+R_{g(t)}v_t,
 \qquad
 \partial_t dV_{g(t)}=-R_{g(t)}dV_{g(t)},
\]
self-adjointness of $\Delta_{g(t)}$ on the compact manifold gives
\begin{align*}
 \frac{d}{dt}\int_Xh\,d\nu_{x_0,t_0;t}
 &=\int_Xh\bigl(-\Delta v_t+Rv_t\bigr)dV
   -\int_XhRv_t\,dV \\
 &=-\int_X(\Delta h)v_t\,dV
  =-\int_X\Delta_{g(t)}h\,d\nu_{x_0,t_0;t}.
\end{align*}
At the pole time the measure is $\delta_{x_0}$ and $h(x_0)=0$.
Integrating from $t$ to $t_0$ therefore yields
\[
 \int_Xh\,d\nu_{x_0,t_0;t}
 =\int_t^{t_0}\!\int_X\Delta_{g(s)}h\,
       d\nu_{x_0,t_0;s}\,ds
 \le C(t_0-t).
\]
The ambient barrier comparison and the convention that \(d_\eta\) is the
restricted ambient Fubini--Study distance give
$d_\eta(y,\pi(x_0))^2\le C\widehat\rho_{F(x_0)}(\iota(y))$ and hence
\[
 \int_X d_\eta(\pi(x),\pi(x_0))^2\,d\nu_{x_0,t_0;t}(x)
 \le C(t_0-t).
\]
The constant is uniform in the pole by the uniform ambient construction.

Suppose more generally that \(z_t\) satisfies
\eqref{eq:A-centre-bound} for the point kernel, with constant \(A\).
The ambient Schwarz estimate
\eqref{eq:base-localization-ambient-Schwarz} says that \(F\) is uniformly
Lipschitz from \((X,g(t))\) to \((\PP^N,g_{\rm FS})\), so
\[
 \int_Xd_\eta(\pi(z_t),\pi(x))^2\,d\nu_{x_0,t_0;t}(x)
 \le C A(t_0-t).
\]
For every $x$, the triangle inequality gives
\[
 d_\eta(\pi(z_t),\pi(x_0))^2
 \le 2d_\eta(\pi(z_t),\pi(x))^2
     +2d_\eta(\pi(x),\pi(x_0))^2.
\]
Integrating this constant left-hand side proves the point-kernel case of
part~\textup{(iii)}; setting \(A=H_{2n}\) proves the center conclusion
in part~\textup{(i)}.

For a limiting point, choose the point-kernel realization in the
definition of $M_T$, with pole spacetime points $(x_j,t_j)$ and
$t_j\nearrow T$.  Compactness of the projective target permits a
subsequence for which $q_j:=\pi(x_j)\to q'\in Y$, while the kernels
converge to $\mu_t$ at every fixed earlier time.  Applying the already proved
point-kernel estimate gives
\[
 \int_Xd_\eta(\pi(x),q_j)^2\,
 d\nu_{x_j,t_j;t}(x)\le C(t_j-t).
\]
For fixed \(t<T\), put
\[
 \nu_j:=\nu_{x_j,t_j;t},\qquad
 f_j:=d_\eta(\pi(\,\cdot\,),q_j)^2,
 \qquad
 f:=d_\eta(\pi(\,\cdot\,),q')^2.
\]
Since \(Y\subset\PP^N\) is compact,
\(D:=\operatorname{diam}_{d_\eta}Y<\infty\).  The triangle inequality gives
\[
 \|f_j-f\|_{C^0(X)}
 \le 2D\,d_\eta(q_j,q')\longrightarrow0.
\]
Thus \(f_j\to f\) uniformly.  Moreover, \(f\) is bounded and
continuous, \(\nu_j\rightharpoonup\mu_t\), and all these measures have
total mass one.  Therefore
\[
 \left|\int_Xf_j\,d\nu_j-\int_Xf\,d\mu_t\right|
 \le \|f_j-f\|_{C^0(X)}
 +\left|\int_Xf\,d\nu_j-\int_Xf\,d\mu_t\right|
 \longrightarrow0.
\]
Passing to the limit in the preceding moment estimate now yields
\[
 \int_Xd_\eta(\pi(x),q')^2\,d\mu_t(x)\le C(T-t).
\]
Thus $\pi_*\mu_t\rightharpoonup\delta_{q'}$ as $t\nearrow T$.  If a
second subsequence of pole images converged to $q''$, the same argument
would give $\pi_*\mu_t\rightharpoonup\delta_{q''}$ for the very same
one-parameter family of measures.  Weak limits in the compact Hausdorff
target are unique, so $q''=q'$.  Hence every convergent subsequence of
pole images has the same limit and compactness gives convergence of the
whole sequence.  Define this unique point to be $q_\pi(\xi)$.  The last
display is therefore exactly
\[
 \int_Xd_\eta(\pi(x),q_\pi(\xi))^2\,d\mu_t(x)
 \le C(T-t).
\]
If \(z_t\) satisfies \eqref{eq:A-centre-bound} for the limiting point,
the same ambient Schwarz inequality gives
\[
 \int_Xd_\eta(\pi(z_t),\pi(x))^2\,d\mu_t(x)
 \le C A(T-t).
\]
Integrating the same triangle inequality, now with
$q_\pi(\xi)$ in place of $\pi(x_0)$, proves the limiting-point case of
part~\textup{(iii)}.  Taking \(A=H_{2n}\) proves the center conclusion
in part~\textup{(ii)}.  The constants in \textup{(i)}--\textup{(ii)}
depend only on the fixed flow, the target metric, and the dimension; in
part~\textup{(iii)} they also depend on \(A\), but never on the chosen
limiting point.
\end{proof}

\begin{proposition}
\label{prop:automatic-anchored-point}
Fix \(x_0\in X\) and \(T_j\nearrow T\).  A subsequence of the point
conjugate heat kernels based at \((x_0,T_j)\) converges at every time
\(t<T\) to a limiting point based at \(\pi(x_0)\).  This limiting point
satisfies the Type-I target-moment estimate and may be fixed before any
tangent scale is chosen.  No uniqueness is asserted.
\end{proposition}

\begin{proof}
Bamler's construction of the limiting slice gives, after a diagonal
subsequence, a compatible family of positive probability conjugate
heat measures at all earlier rational times; the reproduction formula
extends this family to a conjugate heat flow.  It is a limiting point,
and the concentration estimate
passes to the limit
\cite[Definition~2.32 and the paragraph following
equation~(2.34)]{BamlerStructure}
\cite[Corollary~3.8 and equation~(3.9)]{BamlerHeat}.  The resulting
limiting point is chosen before any tangent scale.

It remains only to identify its base.  The point-kernel estimate in
\cref{surf:lem:base-localization}\textup{(i)} is uniform in \(j\) and
gives
\[
 \int_Xd_\eta(\pi(x),\pi(x_0))^2\,
 d\nu_{x_0,T_j;t}(x)\le C(T_j-t).
\]
Passing to the diagonal limit gives the same estimate with \(T-t\);
\cref{surf:lem:base-localization}\textup{(ii)} then identifies the
automatic base point with \(\pi(x_0)\).
\end{proof}

\subsection{Spacetime and convergence conventions; the sequence-level
K\"ahler shrinker theorem}
Following Kleiner--Lott \cite[Definition~1.2]{KleinerLott}, a
K\"ahler--Ricci flow spacetime is a tuple
\((\mathcal M,\mathfrak t,\partial_{\mathfrak t},g,\mathcal J)\) in
which \(\mathfrak t\) is a submersion,
\(\partial_{\mathfrak t}\mathfrak t=1\), \(g\) is a metric on
\(\ker d\mathfrak t\) satisfying
\[
 \mathcal L_{\partial_{\mathfrak t}}g=-2\Rc(g),
\]
and \(\mathcal J\) is time invariant and restricts to a K\"ahler
complex structure on each time slice.

An orbifold shrinking gradient K\"ahler--Ricci soliton of complex
dimension \(n\) is a K\"ahler orbifold \((\widehat X,\widehat g,J)\),
together with a smooth function \(f\) on its regular locus, such that
\begin{equation}\label{surf:eq:shrinker-equation}
 \Rc(\widehat g)+\nabla^2f=\tfrac12\widehat g,
 \qquad \mathcal L_{\nabla f}J=0.
\end{equation}
At an isolated singularity, its local group is a nontrivial finite
subgroup of \(U(n)\) acting freely on
\(\C^n\setminus\{0\}\).  When \(\nabla f\) is complete, let \(\Theta_s\)
solve
\[
 \partial_s\Theta_s=(-s)^{-1}\nabla f\circ\Theta_s,
 \qquad \Theta_{-1}=\mathrm{id}.
\]
Then
\begin{equation}\label{surf:eq:induced-shrinker-spacetime}
 \widehat g(s)=(-s)\Theta_s^*\widehat g,
 \qquad s<0,
\end{equation}
is its standard ancient K\"ahler--Ricci flow spacetime on
\(\Reg(\widehat X)\times(-\infty,0)\).

A correspondence for metric flows embeds their time slices
isometrically into common metric spaces.  Bamler's \(\mathbb F\)-distance
uses couplings of the distinguished measures whose backward heat-kernel
measures are close in \(W_1\), outside a set of times of arbitrarily
small measure; see
\cite[Definitions~5.4 and~5.6]{BamlerCompactness}.  For a metric-flow
slice \(\mathcal X_s\), write \(\mathcal R_s\) for its metric regular
locus.  Once a normal complex structure is present, \(\Reg Z\) and
\(\Sing Z\) denote the analytic regular and singular loci; the results
below identify these with the metric loci.

\begin{theorem}
\label{thm:general-kahler-shrinker-package}
Let \(n\ge2\).  For each \(j\), let
\[
 (M_j^{2n},g_j(s),J_j,\nu^j_s,x_j),\qquad -A_j\le s\le0,
\]
be a pointed compact K\"ahler--Ricci flow with its point conjugate heat
kernel based at \((x_j,0)\), in the standard real convention
  \(\partial_sg_j=-2\Rc(g_j)\), where $A_j\ge1$.  Assume:
\begin{enumerate}[label=\textup{(\roman*)},leftmargin=2.4em]
\item \(A_j\to\infty\), the pairs are uniformly \(H_{2n}\)-concentrated,
  and
  \[
            \inf_j\mathcal N^{g_j}_{x_j,0}(1)>-\infty;
  \]
\item the pointed metric-flow pairs converge in \(\mathbb F\) on every
  compact negative time interval to
  \((\mathcal Z,g_{\mathcal Z}(s),\nu^Z_s,z)\);
\item the limit is a connected noncollapsed metric soliton and convergence is smooth
  on its full regular spacetime;
\item on the regular convergence charts, \(J_j\) converges smoothly to a
  parallel K\"ahler complex structure \(J_Z\), while the distinguished
  point-kernel densities converge smoothly to the distinguished soliton
  density.
\end{enumerate}
Put \(Z:=(\mathcal Z_{-1},d_Z)\), let
\(\mathcal R:=\mathcal R_{-1}\), and denote the limiting regular
K\"ahler structure by \((g_Z,J_Z,\omega_Z)\).  Then the following two
conclusions hold.
\begin{enumerate}[label=\textup{(\alph*)},leftmargin=2.4em]
\item The space \(Z\) carries a compatible normal complex-analytic
structure with klt singularities.  Its analytic and metric topologies
agree, and its analytic regular locus is precisely \(\mathcal R\).
\item The form \(\omega_Z\) extends locally across \(\Sing Z\) as a
positive current with bounded potential.  Every \(p\in Z\) has a
neighbourhood \(U\), a closed holomorphic embedding
\(\iota_U:U\hookrightarrow\Omega_U\subset\C^L\), and \(c_U>0\) such that
\begin{equation}\label{eq:general-shrinker-ambient-lower}
 \omega_Z\ge c_U\,\iota_U^*\omega_{\C^L}
 \qquad\text{on }U\cap\mathcal R.
\end{equation}
\end{enumerate}
\end{theorem}

\begin{proof}
Hypotheses \textup{(i)--(iv)} put the time-$-1$ slice in the
noncollapsed K\"ahler-shrinker class of Hallgren--Zhang.  Their analytic
structure theorem gives exactly conclusions \textup{(a)--(b)}:
\cite[Theorem~B(i)--(iii)]{HallgrenZhang}.
Their K\"ahler metric convention differs from the real Ricci-flow
convention here by the fixed factor \(2\).  This changes only the
constant \(c_U\) in \eqref{eq:general-shrinker-ambient-lower}, and does
not change normality, the klt property, regular loci, topologies, or
local boundedness of a potential.
\end{proof}

\subsection{Tangent spaces at a fixed limiting point}
\label{sec:fixed-limiting-package}

Fix a limiting point \(\xi=(\mu_t)_{t<T}\) before selecting any
scales, and let \(\tau_i\searrow0\).  If
\(d\mu_{i,-1}=(4\pi)^{-n}e^{-f_{i,-1}}dV_{g_i(-1)}\), set
\(
 \mathcal N_i(1):=\int_Xf_{i,-1}\,d\mu_{i,-1}-n
\).
Parabolic scaling gives the first equality below, while Bamler's
tangent-entropy theorem gives the convergence:
\begin{equation}\label{sp:eq:entropycheck}
 \mathcal N_i(1)=\mathcal N_{\xi,T}(\tau_i)
 \longrightarrow\mathcal N_{\xi,T}(0)>-\infty.
\end{equation}
This is
\cite[the paragraph preceding Theorem~2.37 and Theorem~2.37]{BamlerStructure}.
Thus every limiting-point rescaling
has the entropy lower bound required by the noncollapsed regularity
theory.

For a tangent space at \(\xi\), write its time-\((-1)\) model as
\((Z,d,\mathcal R,g,J,f,\mathfrak m)\).  On its regular locus define
\begin{equation}\label{sp:eq:intro-rRm}
 r_{\rm Rm}(x)=\sup\bigl\{0<r\le1:
 B_g(x,r)\Subset\mathcal R\ \text{and}\
 \sup_{B_g(x,r)}|\Rm_g|\le r^{-2}\bigr\},
\end{equation}
and extend \(r_{\rm Rm}\) by zero on the singular set.  No curvature
derivative bound is part of this definition.

\begin{proposition}
\label{sp:prop:limiting-package}\label{prop:limiting-package}
Let \(n\ge2\).  Every tangent space at \(\xi\) is a metric soliton
flow with a distinguished probability measure.  Its time-\((-1)\) model
has the following properties.
\smallskip
\noindent\emph{Metric-flow conclusions.}
\begin{enumerate}[label=\textup{(H\arabic*)},leftmargin=3em,
series=limitingpackage]
\item The regular locus \(\mathcal R\) is connected and dense,
\(d|_{\mathcal R}\) is its intrinsic length distance, and \(Z\) is a
complete locally compact length space equal to the metric completion of
\(\mathcal R\).
\item There is a smooth function \(f\) on \(\mathcal R\) for which
\begin{equation}\label{sp:eq:shrinker}
 \Rc_g+\nabla^2f=\frac12g\qquad\text{on }\mathcal R;
\end{equation}
the distinguished probability measure satisfies
\begin{equation}\label{sp:eq:measure}
 d\mathfrak m=(4\pi)^{-n}e^{-f}\,dV_g\quad\text{on }\mathcal R,
 \qquad \mathfrak m(Z\setminus\mathcal R)=0,
 \qquad \mathfrak m(Z)=1,
\end{equation}
and \(\mathfrak m\) has finite second moment.
\item The tangent space is an exact self-similar metric soliton flow.
On regular convergence charts, in time-compatible gauges, the rescaled
metrics and distinguished densities converge smoothly to \(g\) and the
positive density in \eqref{sp:eq:measure}.
\item For every compact \(K\subset Z\) and every \(p\in(2,4)\), the
zero extension of \(r_{\rm Rm}\) is one-Lipschitz and
\begin{equation}\label{sp:eq:Minkowski}
 \Vol_g\bigl(\{r_{\rm Rm}<\varepsilon\}\cap K\cap\mathcal R\bigr)
 \le C_{K,p}\varepsilon^p,
 \qquad 0<\varepsilon<1;
\end{equation}
\end{enumerate}

\smallskip
\noindent\emph{K\"ahler conclusions.}
\begin{enumerate}[label=\textup{(H\arabic*)},leftmargin=3em,
resume=limitingpackage]
\item The regular soliton \((\mathcal R,g,J)\) is K\"ahler, \(J\) is
parallel and time invariant, and the rescaled complex structures converge
smoothly to \(J\) on regular convergence charts.
\item If \(n=2\), then \(Z\) is a K\"ahler orbifold and
\(\Sing Z\) consists of isolated orbifold points.
\end{enumerate}

\smallskip
\noindent\emph{Complex-analytic conclusions.}
\begin{enumerate}[label=\textup{(H\arabic*)},leftmargin=3em,
resume=limitingpackage]
\item The space \(Z\) is a normal klt complex space.  Its analytic and
metric topologies agree, and its analytic and metric regular loci
coincide.  The regular K\"ahler form extends locally across \(\Sing Z\)
as a positive current with bounded potential.  Every point has a local
closed analytic embedding \(V\hookrightarrow\Omega_V\subset\C^L\) such
that, on \(V\cap\mathcal R\),
\begin{equation}\label{sp:eq:ambientlower}
 \omega_g\ge c_V\omega_{\C^L}|_V.
\end{equation}
\item The function \(f\) extends uniquely to a locally Lipschitz function
on \(Z\).  The flow of \(\nabla f\) is complete, preserves
\(\mathcal R\), and extends to a one-parameter group of biholomorphisms
of \(Z\).
\end{enumerate}
\end{proposition}

\begin{proof}
Items \textup{(H1)--(H3)} are consequences of Bamler's metric-flow and
tangent-soliton structure results
\cite[Theorems~2.4(a)--(c), 2.5, 2.14, 2.18, and~2.37]{BamlerStructure}.
Smooth density convergence and compatible time gauges come separately
from \cite[Theorem~9.31(a) and Lemma~9.33]{BamlerCompactness}.
Item \textup{(H4)} is \cref{sp:prop:automatic-scale}; it is derived from
Bamler's spacetime curvature-radius estimate and exact self-similarity.
Item \textup{(H5)} follows from Bamler's regular gauges and the
parallel-complex-structure argument of
\cite[Theorem~2.5 and its proof]{HallgrenJian}.  Item \textup{(H6)}
follows from Bamler's real four-dimensional orbifold theorem
\cite[Theorem~2.46]{BamlerStructure} and the local complex-structure
extension argument in the final paragraph of the proof of
\cite[Theorem~2.8]{ConlonHallgrenMa}, applied to \textup{(H5)}.
Item \textup{(H7)} is
\HZanalyticref[\textup{(a)--(b)}].  Item \textup{(H8)} follows from
Hallgren--Zhang as follows.

Suppose first that \(f\) is nonconstant.  The paragraph preceding
equation~(2.4) in Hallgren--Zhang gives the locally Lipschitz extension
of \(f\) to \(Z\), and the extension is unique because \(\mathcal R\)
is dense.  Their Theorem~B(iii) identifies the analytic regular locus
with \(\mathcal R\), while their Proposition~5.3(i) gives completeness
of \(\nabla f\) on \(\mathcal R\) and extension of its flow maps to
biholomorphisms of \(Z\)
\cite[Theorem~B(iii), the paragraph preceding equation~(2.4), and
Proposition~5.3(i)]{HallgrenZhang}.
The factor-two difference between the K\"ahler and real Ricci-flow
conventions only reparametrizes this flow.
Each extended flow map preserves the analytic regular locus and hence
preserves \(\mathcal R\).  The group identities hold on \(\mathcal R\)
and therefore on \(Z\) by continuity and density.  If \(f\) is constant,
then \(\nabla f=0\), so all the assertions in \textup{(H8)} are immediate.

For later use, we record the quantitative growth estimate.  On the
connected regular locus, the contracted Bianchi identity and the shrinker
equation give
\[
                  R+|\nabla f|^2-f=C_0.
\]
The scalar curvature of the original compact flow has a uniform lower
bound.  Rescaling at the limiting time and smooth regular-chart
convergence therefore give \(R\ge0\) on \(\mathcal R\).  With
\(\widehat f=f+C_0\),
\[
              \widehat f=R+|\nabla f|^2\ge|\nabla f|^2.
\]
For every \(\varepsilon>0\),
\[
 \left|\nabla\sqrt{\widehat f+\varepsilon}\right|
 =\frac{|\nabla f|}{2\sqrt{\widehat f+\varepsilon}}
 \le\frac12.
\]
Fix \(o,x\in\mathcal R\).  For every \(\delta>0\), the intrinsic-distance
assertion in \textup{(H1)} allows us to choose a piecewise smooth curve
\(\gamma_\delta\subset\mathcal R\) from \(o\) to \(x\) such that
\[
                  L_g(\gamma_\delta)\le d(o,x)+\delta.
\]
Then
\[
\begin{aligned}
 \sqrt{\widehat f(x)+\varepsilon}
 &\le \sqrt{\widehat f(o)+\varepsilon}
   +\int_{\gamma_\delta}
     \left|\nabla\sqrt{\widehat f+\varepsilon}\right|\,ds \\
 &\le \sqrt{\widehat f(o)+\varepsilon}
   +\frac12\bigl(d(o,x)+\delta\bigr).
\end{aligned}
\]
Letting first \(\delta\searrow0\) and then
\(\varepsilon\searrow0\), and using
\(|\nabla f|\le\sqrt{\widehat f}\), yields
\begin{equation}\label{sp:eq:preliminary-flow-growth}
             |\nabla f|(x)\le\frac12d(o,x)+B,
             \qquad B:=\sqrt{\widehat f(o)}.
\end{equation}

\end{proof}

\begingroup
\small
\medskip
\noindent\textbf{Sources and verification of the finite-time
tangent-space package.}\par\smallskip

\medskip\noindent\emph{Metric soliton, measure, and regular convergence.}

Bamler's Theorem~2.37 applies to rescalings of the same fixed
\(\xi\in M_T\) and, together with \eqref{sp:eq:entropycheck}, supplies the
metric-soliton model and noncollapse.  Theorems~2.5 and~2.14 give smooth
regular-spacetime convergence for a general limiting-point conjugate
heat flow;
Theorems~2.4(a)--(c) and~2.18 give density of the regular locus, zero
singular mass, the intrinsic completion, the normalized
shrinker equation, and self-similar representation
\cite[Theorems~2.4(a)--(c), 2.5, 2.14, 2.18, and 2.37]{BamlerStructure}.
Smooth convergence of the distinguished density is the conclusion of
\cite[Theorem~9.31(a)]{BamlerCompactness}; part~(f) of that theorem
concerns an additional conjugate flow and is not used.  Bamler's
metric-soliton description in Theorem~2.18 identifies the limit with the
smooth strictly positive soliton density
\((4\pi)^{-n}e^{-f}\) on the regular spacetime.
Time-vector-compatible regular gauges are furnished by
\cite[Lemma~9.33]{BamlerCompactness}.

The general conjugate-flow density and zero-singular-mass assertion are
recorded in the paragraph preceding Bamler's Theorem~2.10
\cite[Subsection~2.2, paragraph preceding Theorem~2.10]{BamlerStructure}.
The distinguished measure in Bamler's metric-soliton limit is a
probability measure, so \(\mathfrak m(Z)=1\).  The uniform
\(H_D\)-concentration, with \(D=2n\), in \textup{(P2)} is the
accompanying tightness
input in the pointed convergence and, after rescaling, gives
\[
 \iint d_{g_i(-1)}(x,y)^2\,d\mu_{i,-1}(x)d\mu_{i,-1}(y)\le H_D.
\]
Lower semicontinuity gives finite pairwise second moment on the limit,
which implies finite second moment about any fixed point.  These facts
verify \textup{(H1)--(H3)}.

\medskip\noindent\emph{Passage of the K\"ahler structure.}

Hallgren--Jian's printed Theorem~2.5(i)--(iv) is formulated for point
kernels on flows defined through time zero, not literally for the present
general limiting point.  We therefore use Bamler's gauges for general
limiting points
above and rerun only the parallel-complex-structure argument in the proof
of that theorem
\cite[Theorem~2.5(i)--(iv) and its proof]{HallgrenJian}.
On Bamler's regular exhaustion charts,
\(|J_i|_{g_i}=\sqrt{2n}\) and \(\nabla^{g_i}J_i=0\).  Smooth connection
convergence, local derivative bounds, and a diagonal
Arzel\`a--Ascoli argument give a smooth parallel limit \(J\) on every
regular component.  The identities \(J^2=-\Id\), integrability, K\"ahler
compatibility, and time invariance pass to the limit.  This proves
\textup{(H5)}; no broader conclusion of the printed Hallgren--Jian theorem
is used outside its hypotheses.

When \(n=2\), Bamler's real four-dimensional orbifold theorem
\cite[Theorem~2.46]{BamlerStructure}, together with the local
complex-structure extension argument in
\cite[the final paragraph of the proof of Theorem~2.8]{ConlonHallgrenMa},
gives
\textup{(H6)}.
That local argument uses only the parallel complex structure from
\textup{(H5)} and is therefore independent of the point-kernel
hypothesis in the statement of the cited theorem.

\medskip\noindent\emph{Analytic structure and the soliton flow.}

Property \textup{(P1)}, the \(W_1\)-contraction for conjugate heat flows
\cite[the paragraph following equation~(2.34)]{BamlerStructure}, and a
diagonal choice of pole times give a point-kernel realization of each
prescribed tangent space at a fixed limiting point.  The diagonal, including the
necessary control
of the pole-time shift, is written out in the proof of
\cref{sp:prop:automatic-scale}; that argument uses only
\textup{(H1)--(H3)}, \textup{(P1)}, the \(W_1\)-contraction, and the
entropy bound, and hence does not presuppose \textup{(H7)}.  Bamler's
regular convergence and the preceding
parallel-tensor argument then verify the four hypotheses of
\HZanalyticref.  Its conclusions \textup{(a)--(b)} give exactly
\textup{(H7)}.

For nonconstant \(f\), the paragraph preceding equation~(2.4),
Theorem~B(iii), and Proposition~5.3(i) of Hallgren--Zhang give
\textup{(H8)}: the locally Lipschitz extension of \(f\), completeness of
\(\nabla f\), preservation of the regular locus, and extension of the
flow to a one-parameter group of biholomorphisms
\cite[Theorem~B(iii), the paragraph preceding equation~(2.4), and
Proposition~5.3(i)]{HallgrenZhang}.
The constant-potential case is immediate.  The proof of
\cref{sp:prop:limiting-package} also records the elementary quantitative
estimate \eqref{sp:eq:preliminary-flow-growth} used in the spectral
argument below.

\medskip\noindent\emph{The fixed-slice quantitative curvature estimate.}

Property \textup{(H4)} is proved in
\cref{sp:prop:automatic-scale}.  Bamler's generalized parabolic radius
is defined after Theorem~2.29, and Theorem~2.31 gives its spacetime
inverse-radius estimate.  The lower entropy bound on the compact
self-similar track is obtained from the original compact flow by
\cite[Proposition~4.35 and equation~(4.34)]{BamlerStructure}, parabolic
invariance, and entropy convergence
\cite[Theorem~2.10]{BamlerStructure}.  Exact self-similarity
\cite[Theorems~2.18 and~2.37]{BamlerStructure} then passes the
spacetime integral to the time-\((-1)\) slice.  Comparing parabolic and
spatial radii and applying Chebyshev gives every exponent
\(p\in(2,4)\), but not the endpoint \(p=4\).  The internal smoothing
argument in \cref{sp:lem:capacity} turns this estimate into the
\(W^{1,\gamma}\)-cutoffs used by the spectral theory.
\par\medskip
\noindent\hfill\(\square\)
\endgroup

\begin{lemma}
\label{lem:regular-chart-localization}
Let \(\xi=(\mu_t)_{t<T}\) be a fixed limiting point based at
\(q\in Y^{\rm rv}\), and let \(\tau_i\searrow0\) realize a tangent
space at \(\xi\) with time-\((-1)\) regular locus \(\mathcal R\).  Choose
\(q\in U_0\Subset U_1\Subset Y^{\rm rv}\), and let \(z_i\) be an
\(H_{2n}\)-center of \(\mu_{T-\tau_i}\).

For every compact \(K\Subset\mathcal R\) and corresponding time-\((-1)\)
regular convergence maps \(\psi_i:K\to X\), there is \(C_K<\infty\)
such that, for all sufficiently large \(i\),
\[
 \psi_i(K)\subset\pi^{-1}(U_1),
 \qquad
 \sup_{x\in K}
 d_{\tau_i^{-1}g(T-\tau_i)}\bigl(z_i,\psi_i(x)\bigr)
 \le C_K.
\]
\end{lemma}

\begin{proof}
In the rescaled metric \(g_i(-1)=\tau_i^{-1}g(T-\tau_i)\), the defining
center inequality becomes
\[
 \int_Xd_{g_i(-1)}(z_i,x)^2\,d\mu_{i,-1}(x)\le H_{2n}.
\]
Cover a fixed compact regular chart by finitely many smaller regular
balls whose closures remain in convergence charts.  Positivity of the
limiting density and smooth density and volume convergence give a
uniform lower bound \(\delta>0\) for the \(\mu_{i,-1}\)-mass of each
ball.  If the center of one such ball were at distance \(D_i\to\infty\)
from \(z_i\), a fixed-radius subball would contribute at least
\(\delta(D_i-r)^2\) to the preceding integral, a contradiction.  Hence
the whole compact chart remains at bounded rescaled distance from
\(z_i\).

Moreover, \cref{surf:lem:base-localization}\textup{(ii)} gives
\(\pi(z_i)\to q\), so \(\pi(z_i)\in U_0\) for large \(i\).  Put
\(\delta_0=d_\eta(\overline U_0,Y\setminus U_1)>0\).  The global AMMP
Schwarz estimate \cref{lem:ambient-schwarz} gives
\[
 d_{g(T-\tau_i)}(z_i,X\setminus\pi^{-1}(U_1))
 \ge c^{1/2}\delta_0.
\]
After rescaling, this lower bound is
\(c^{1/2}\delta_0\tau_i^{-1/2}\to\infty\).  The bounded chart must
therefore lie in \(\pi^{-1}(U_1)\).
\end{proof}

\begin{remark}[Logical separation from fixed limiting points]
\label{rem:general-secondary-separation}
\HZanalyticref\ supplies the analytic structure of a general K\"ahler
shrinker limit.  It is neither a hypothesis on the original flow nor a
source of metric-completion or curvature conclusions.  The moving-point
argument below verifies its hypotheses on one quantitative two-level
diagonal.  By contrast, \cref{prop:limiting-package} concerns tangent
spaces at a single fixed limiting point.
\end{remark}

\section{Spectral splitting at the limiting time}
\label{sec:spectral-splitting}

This section proves the full-base-rank splitting theorem in the AMMP
setting of \cref{sec:fano-fibration-setup}.  We use the flow, the limiting
morphism
\[
                         \pi:X^n\longrightarrow Y^m,
                         \qquad 0<m<n,
\]
the limiting points and their tangent spaces, the quotient tensors, and
conditions
\((\mathrm S)_U\) and \((\mathrm Q)_U\) from \cref{sec:prelim}.
The Song--Tian Schwarz estimate
\cref{lem:ambient-schwarz} makes \((\mathrm S)_U\) automatic, while the
Fu--Zhang estimate
\cref{prop:zariski-product-horizontal} supplies
\((\mathrm Q)_U\) on the Zariski-open product regions used in the
applications.

At a limiting scale \(\tau_i\searrow0\), base coordinates
\(a=(a^1,\ldots,a^m)\) centered at the automatic base point have the
natural magnification
\[
                         F_i=\tau_i^{-1/2}a\circ\pi.
\]
The goal is to prove that their limits furnish \(m\) independent
parallel holomorphic coordinates on the time-\((-1)\) model of every
tangent space at a limiting point.  We record
the precise local criterion that the intrinsic argument establishes.

\begin{proposition}
\label{sp:prop:fullrank}
Let \(q\in Y^{\rm rv}\).  Choose a relatively compact holomorphic
coordinate neighbourhood \(U\Subset Y^{\rm rv}\), coordinates
\(a=(a^1,\ldots,a^m):U\to\C^m\) with \(a(q)=0\), and open sets
\[
                         q\in U_0\Subset U_1\Subset U.
\]
Let \(\xi=(\mu_t)_{t<T}\) be a fixed limiting point based at \(q\).
If both \((\mathrm S)_U\) and \((\mathrm Q)_U\) hold, then the
time-\((-1)\) model of every tangent space at \(\xi\) splits globally,
isometrically, and biholomorphically as
\begin{equation}\label{sp:eq:fullproduct}
 (Z,d,J)\cong(\C^m,g_{\rm E},J_0)\times(Z',d',J').
\end{equation}
\end{proposition}

Two analytic issues arise.  First, the regular locus of a tangent space at a limiting point
may be incomplete, so complete-shrinker spectral theory does not apply
directly.  Second, the limiting coordinate functions initially exist only
on regular convergence charts.  The three subsections below organize the
proof into six steps:

\begin{enumerate}[label=\textup{(\arabic*)},leftmargin=3em]
\item derive the curvature-radius estimate, prove zero weighted
\(W^{1,2}\)-capacity of the singular set, and construct the canonical
Friedrichs drift operator;
\item pass the exact conjugate-kernel Poincar\'e inequality from the
fixed limiting point to obtain the sharp spectral gap \(1/2\);
\item use the complete soliton flow and a spectral cutoff to show that
Lipschitz holomorphic half-modes have vanishing real and imaginary
Hessians;
\item globalize the resulting parallel holomorphic map across the normal
klt singular set;
\item construct the magnified base-coordinate limit on every compact
regular chart;
\item use \((\mathrm Q)_U\) to retain rank \(m\), prove the criterion,
and record the normalized-coordinate and exact-zero-fibre consequences
needed later.
\end{enumerate}

We settle each singular-domain issue before using the corresponding smooth
calculation.  In particular, we neither apply the de Rham theorem to the
incomplete regular locus nor assert essential self-adjointness.  Until
\cref{sp:sec:fibrationmap}, fix one tangent space; the fibration re-enters
when we construct the magnified coordinates.

\subsection{Weighted drift analysis and spectral rigidity}
\label{sp:sec:capacity}
\label{sp:sec:poincare}
\label{sp:sec:spectral-rigidity}

\paragraph{Quantitative curvature control.}

\begin{proposition}
\label{sp:prop:automatic-scale}
Let \((Z,d,\mathcal R,g)\) be the time-\((-1)\) model of a tangent
space at a fixed limiting point, and put
\(N=\dim_{\R}\mathcal R=2n\).  Then,
for every compact \(K\subset Z\) and every \(p\in(2,4)\),
\begin{equation}\label{sp:eq:inverse-radius}
 \int_{K\cap\mathcal R}r_{\rm Rm}^{-p}\,dV_g<\infty.
\end{equation}
Consequently \eqref{sp:eq:Minkowski} holds, and the zero extension of
\(r_{\rm Rm}\) is one-Lipschitz on \(Z\).
\end{proposition}

\begin{proof}[Proof of \cref{sp:prop:automatic-scale}]
We keep separate Bamler's generalized parabolic curvature radius and
the spatial radius used in this paper.  Write the exact self-similar
tangent flow as
\[
 (\mathcal Z_s,d_s,\mathcal R_s,g(s))_{s<0},\qquad
 (Z,d,\mathcal R,g)=(\mathcal Z_{-1},d_{-1},\mathcal R_{-1},g(-1)),
\]
and abbreviate \(Z_s:=\mathcal Z_s\).  For
\((x,s)\in\mathcal R_s\), let
\(\widehat r_{\rm Rm}(x,s)=r'_{\rm Rm}(x,s)\) be Bamler's generalized
radius \cite[the definition following Theorem~2.29]{BamlerStructure}.
Define the
uncapped spatial radius on each regular time slice by
\begin{equation}\label{sp:eq:uncapped-radius}
 \rho_s(x)=\sup\bigl\{r>0:B_{g(s)}(x,r)\Subset\cR_s,
              \ \sup_{B_{g(s)}(x,r)}|\Rm_{g(s)}|\le r^{-2}\bigr\}.
\end{equation}
Thus \(r_{\rm Rm}=\min\{1,\rho_{-1}\}\).

The proof has three stages.  First we apply Bamler's spacetime
inverse-radius estimate on a finite cover of the self-similar track of a
compact set.  Second, exact self-similarity turns the spacetime integral
into an integral on the prescribed time slice.  Third, Chebyshev's
inequality gives the Minkowski estimate, and the definition of the
spatial radius gives the Lipschitz property.  The second stage is the
reason exact self-similarity is essential.

We first check the entropy hypothesis at every center used below.  By
the limiting-point realization property \textup{(P1)} and the
\(W_1\)-contraction for conjugate heat flows
\cite[the paragraph following equation~(2.34)]{BamlerStructure}, first
pass to a subsequence with \(T/\tau_j>j+1\).  This does not change the
prescribed tangent space.  We may then choose point-kernel poles
\((y_i,T_i)\) diagonally so that
\[
 \bar g_i(s)=\tau_i^{-1}g(T_i+\tau_i s),
 \qquad T_i\nearrow T,
 \qquad \lambda_i:=\frac{T-T_i}{\tau_i}\longrightarrow0,
\]
and the associated point-kernel flows converge on every compact
negative time interval to the prescribed tangent space.  To obtain the
diagonal, at stage \(j\) use the scale \(\tau_j\), choose
\(T_j>T-\tau_j/j\), and require at the latest time that
\begin{equation}\label{sp:eq:point-kernel-diagonal}
 \tau_j^{-1/2}d_{W_1}^{g(T-\tau_j/j)}
 \left(\nu_{y_j,T_j;T-\tau_j/j},\mu_{T-\tau_j/j}\right)<j^{-1}.
\end{equation}
Thus \(\lambda_j<j^{-1}\), and \(W_1\)-contraction controls the error
on the whole interval \([-j,-j^{-1}]\).  After relabeling and letting
\(j\to\infty\), this gives the assertion.  In
particular,
\[
                         \bar g_i(s)=g_i(s-\lambda_i),
\]
so the vanishing time translation does not change the chosen tangent
space.
Since the original flow is smooth on the compact manifold at time zero,
Bamler's uniform pointed-entropy estimate
\cite[Proposition~4.35 and equation~(4.34)]{BamlerStructure} yields a
constant \(Y_0<\infty\), depending only on the original flow and \(T\).
Every admissible pointed Nash entropy used below is at least \(-Y_0\).
Indeed, for a fixed spacetime point \((z,s)\) of
the limit and a fixed admissible scale \(r\), choose corresponding
points \(z_i\) in the realizing flows.  Parabolic invariance gives
\[
 \mathcal N^{\bar g_i}_{z_i,s}(r^2)
 =\mathcal N^g_{z_i,T_i+\tau_i s}(\tau_i r^2)\ge-Y_0
\]
for all large \(i\).  Entropy convergence
\cite[Theorem~2.10]{BamlerStructure} passes this bound to the limit.
Thus every center and scale in the finite cover below satisfies the
hypothesis of Theorem~2.31 with the same constant \(Y_0\).  No
point-versus-soliton entropy comparison is used.

Fix a compact \(K\subset Z\) and let
\(I=[-7/8,-3/4]\Subset(-1,0)\).  The self-similar representation supplies
homeomorphisms \(\overline\Theta_s:Z_s\to Z\), whose restrictions
\(\Theta_s:\cR_s\to\cR\) are diffeomorphisms, such that
\begin{equation}\label{sp:eq:selfsimilar-radius-volume}
 \rho_s(\Theta_s^{-1}x)=(-s)^{1/2}\rho_{-1}(x),\qquad
 (\Theta_s)_*dV_{g(s)}=(-s)^{N/2}dV_g.
\end{equation}
This is the regular-locus form of Bamler's exact soliton
self-similarity
\cite[Theorems~2.18 and~2.37]{BamlerStructure}.  Put
\(K_s=\overline\Theta_s^{-1}K\).  The self-similar homeomorphisms depend
continuously on \(s\), so the corresponding full spacetime track is the
continuous image of \(K\times I\) and is compact.  Put
\begin{equation}\label{sp:eq:p-delta}
                         p=4-2\delta,
                         \qquad \delta=\frac{4-p}{2}>0.
\end{equation}
For each spacetime point \(\mathfrak z=(z,s_z)\) of this track, choose
\(r_{\mathfrak z}>0\) small enough that
\[
 [s_z-\delta^{-1}r_{\mathfrak z}^2,
   s_z+r_{\mathfrak z}^2]\Subset(-\infty,0)
\]
and that
\[
 U_{\mathfrak z}:=\operatorname{Int}
 P^*(\mathfrak z;r_{\mathfrak z}/2)
\]
contains \(\mathfrak z\).  Here $P^*(\mathfrak z;r)$ denotes Bamler's
Wasserstein parabolic neighbourhood and is unrelated to the pullback
notation $\pi^*[\eta]$ for the limiting morphism.  The definition of
\(P^*\) uses a closed time interval.  After decreasing
\(r_{\mathfrak z}\), the strict Wasserstein inequality and continuity
imply
\(\mathfrak z\in\operatorname{Int}
P^*(\mathfrak z;r_{\mathfrak z}/2)\)
in the natural spacetime topology; see
\cite[Section~4.5]{BamlerStructure}.  Moreover,
\cite[Proposition~4.25(b)]{BamlerStructure} gives
\(U_{\mathfrak z}\subset P^*(\mathfrak z;r_{\mathfrak z})\).
Compactness supplies a finite subcover
\(U_{\mathfrak z_j}\), \(1\le j\le J\), where
\(\mathfrak z_j=(z_j,s_j)\); put
\(r_j:=r_{\mathfrak z_j}\).  Apply Theorem~2.31 of
\cite{BamlerStructure} on each \(P^*(\mathfrak z_j;r_j)\), with its
parameter \(\varepsilon\) equal to \(\delta\).  The uniform entropy bound
and \eqref{sp:eq:p-delta} give the scale-invariant estimate
\begin{equation}\label{sp:eq:Bamler-local-inverse-radius}
 \int_{[s_j-r_j^2,s_j+r_j^2]\cap I}
 \int_{P^*(\mathfrak z_j;r_j)\cap\cR_s}
       \widehat r_{\rm Rm}(x,s)^{-p}\,dV_{g(s)}(x)\,ds
 \le C(N,Y_0,\delta)r_j^{N+2-p},
\end{equation}
because \(N-2+2\delta=N+2-p\).  The radius and time normalizations in
the definition of \(P^*\) are precisely those of the cited theorem.
Summing these \(J\) enlarged-neighborhood estimates bounds the
nonnegative integral over the track.  Therefore
\begin{equation}\label{sp:eq:parabolic-radius-tube}
 \int_I\!\int_{K_s\cap\cR_s}
       \widehat r_{\rm Rm}(x,s)^{-p}\,dV_{g(s)}(x)\,ds<\infty.
\end{equation}

The central time-slice conditions in the definition of
\(\widehat r_{\rm Rm}\) imply
\begin{equation}\label{sp:eq:radius-comparison}
                 \widehat r_{\rm Rm}(x,s)\le C_N\rho_s(x)
                 \qquad (x\in\cR_s);
\end{equation}
one may take \(C_N=1\) when the same curvature threshold is used.  This
comparison follows directly from the central-slice condition in the
definition of \(r'_{\rm Rm}\)
\cite[the definition following Theorem~2.29, including
assertion~(a)]{BamlerStructure}.
Combining
\eqref{sp:eq:selfsimilar-radius-volume}--\eqref{sp:eq:radius-comparison} and
using Tonelli's theorem gives, with extended-real integrals,
\begin{align}
\infty
&>\int_I\!\int_{K_s\cap\cR_s}\widehat r_{\rm Rm}^{-p}
                       \,dV_{g(s)}\,ds \notag\\
&\ge C_N^{-p}
 \left(\int_I(-s)^{(N-p)/2}\,ds\right)
 \int_{K\cap\cR}\rho_{-1}^{-p}\,dV_g.\label{sp:eq:slice-radius-Lp}
\end{align}
The time factor is finite and strictly positive, so the last spatial
integral is finite.  Since
\(r_{\rm Rm}^{-p}=\max\{1,\rho_{-1}^{-p}\}\le
1+\rho_{-1}^{-p}\), and compact subsets have finite regular volume
(also immediate from \eqref{sp:eq:measure} and local boundedness of \(f\)),
\eqref{sp:eq:inverse-radius} follows.  Chebyshev's inequality now gives
\begin{equation}\label{sp:eq:automatic-chebyshev}
 \operatorname{Vol}_g\bigl(\{r_{\rm Rm}<\varepsilon\}\cap K\cap\cR\bigr)
 \le\varepsilon^p\int_{K\cap\cR}r_{\rm Rm}^{-p}\,dV_g
 \le C_{K,p}\varepsilon^p,
\end{equation}
which is \eqref{sp:eq:Minkowski}.  Notice that no endpoint \(p=4\) is
claimed.

It remains to prove the Lipschitz assertion.  If
\(d(x,y)\ge r_{\rm Rm}(x)\), then
\(r_{\rm Rm}(y)\ge r_{\rm Rm}(x)-d(x,y)\) is trivial.  Otherwise choose
\(d(x,y)<a<r_{\rm Rm}(x)\).  For every
\(0<b<a-d(x,y)\),
\[
 \overline{B_g(y,b)}\subset B_g(x,a)\Subset\cR,
 \qquad \sup_{B_g(y,b)}|\Rm_g|\le a^{-2}\le b^{-2}.
\]
Thus \(r_{\rm Rm}(y)\ge b\).  Letting \(b\nearrow a-d(x,y)\), then
\(a\nearrow r_{\rm Rm}(x)\), and interchanging \(x,y\), proves that
\(r_{\rm Rm}\) is one-Lipschitz on \(\cR\).  Moreover
\begin{equation}\label{sp:eq:radius-to-singular-set}
                 r_{\rm Rm}(x)\le d(x,\Sing Z),
\end{equation}
because a larger ball cannot be compactly contained in \(\cR\).  Hence
the zero extension of \(r_{\rm Rm}\) is one-Lipschitz across
\(\Sing Z\).
\end{proof}

\begin{remark}[Why the quantitative theorem is essential]
Qualitative codimension four of \(\Sing Z\) alone does not imply
\eqref{sp:eq:Minkowski}: the sublevel set of the curvature radius need not
be a metric tubular neighborhood of the singular set.  The additional
ingredient is Bamler's inverse parabolic-radius estimate.  Exact
self-similarity is also essential, since a spacetime \(L^p\) estimate
on a general limit need not control a prescribed time slice.
\end{remark}

The derived estimate \textup{(H4)}, rather than qualitative codimension
alone, is what removes the singular set from the first-order energy
theory.

\begin{remark}[Boundary of the smooth citations]\label{sp:rem:smooth-citations}
The results of Cheng--Zhou \cite{ChengZhou} and He--Ou \cite{HeOu}
assume a complete smooth space, whereas \(\cR\) may be incomplete.
These papers are cited only as methodological antecedents; no result
from either paper is invoked.  The
closed-form gap, ancient-orbit cutoff, and equality case required below
are proved here.
\end{remark}

The preceding quantitative estimate supplies precisely the exponent
\(p>2\) needed to remove the singular set from first-order weighted
energy arguments.

We now construct the operator on which the spectral argument will be
performed.  The regular locus \(\cR\) may be incomplete, so integration
by parts there is not automatic even though its metric completion is
complete.  Saying that \(\Sing Z\) has zero weighted two-capacity means
that functions vanishing near \(\Sing Z\) can approximate the constant
one in weighted \(W^{1,2}\).  This provides first-derivative cutoffs for
the closed Dirichlet form; it does not provide graph-norm cutoffs for a
second-order operator.

In the capacity and operator arguments below, all unlabelled integrals
are with respect to \(\mathfrak m\), unless another measure is displayed.

\begin{lemma}
\label{sp:lem:capacity}
For every compact $K\subset Z$ and every \(\gamma\in(2,4)\), there are
locally Lipschitz functions
$\eta_j:Z\to[0,1]$ whose restrictions lie in $C^\infty(\cR)$, such that
\begin{align}
 0\le\eta_j\le1,\qquad
 &\eta_j=0\text{ near }K\cap\Sing Z,\qquad
 \eta_j\longrightarrow1\text{ locally uniformly on }K\cap\cR,
 \label{sp:eq:capcutoff}\\
 &\int_K\bigl(|1-\eta_j|^2+|\nabla\eta_j|^2\bigr)\,d\mathfrak m
   \longrightarrow0.
 \label{sp:eq:capenergy}
\end{align}
Moreover,
\begin{equation}\label{sp:eq:q-capenergy}
                  \int_K|\nabla\eta_j|^\gamma\,d\mathfrak m\longrightarrow0.
\end{equation}
There is also a sequence $\beta_j\in C_c^\infty(\cR)$ with
\begin{equation}\label{sp:eq:globalcutoff}
 0\le\beta_j\le1,\qquad \beta_j\to1\quad\mathfrak m\text{-a.e.},
 \qquad
 \|1-\beta_j\|_{L^2(\mathfrak m)}^2+
 \int|\nabla\beta_j|^2\,d\mathfrak m\longrightarrow0.
\end{equation}
\end{lemma}

\begin{proof}
\emph{Step 1: the quantitative Lipschitz cutoff.}
Fix \(\gamma\in(2,4)\) and choose \(p\in(\gamma,4)\).  By
\cref{sp:prop:automatic-scale}, the zero extension of
$r_{\rm Rm}$ is one-Lipschitz on all of $Z$; in particular,
$r_{\rm Rm}(x)\le d(x,\Sing Z)$.  Choose a smooth
$\chi:[0,\infty)\to[0,1]$ which is zero on $[0,1]$, one on
$[2,\infty)$, and satisfies $|\chi'|\le2$, and put
\[
                    \eta_\varepsilon=
                    \chi(r_{\rm Rm}/\varepsilon).
\]
On a fixed compact set, $f$ is bounded because it is locally Lipschitz.
Thus the weighted and unweighted regular volumes are uniformly comparable
there.  By \eqref{sp:eq:Minkowski},
\begin{equation}\label{sp:eq:capacityestimate}
 \int_K|\nabla\eta_\varepsilon|^\gamma\,d\mathfrak m
 \le C_K\varepsilon^{-\gamma}
 \operatorname{Vol}_g(\{r_{\rm Rm}<2\varepsilon\}\cap K)
 \le C_{K,p,\gamma}\varepsilon^{p-\gamma}\longrightarrow0.
\end{equation}
Moreover, \(1-\eta_\varepsilon\) is supported in
\(\{r_{\rm Rm}<2\varepsilon\}\), and hence
\begin{equation}\label{sp:eq:capacity-L2}
 \int_K|1-\eta_\varepsilon|^2\,d\mathfrak m
 \le C_K\operatorname{Vol}_g
 \bigl(\{r_{\rm Rm}<2\varepsilon\}\cap K\cap\cR\bigr)
 \le C_{K,p}\varepsilon^p\longrightarrow0.
\end{equation}
Every compact subset of $\cR$ has a positive lower bound for
$r_{\rm Rm}$, so the convergence to one is locally uniform there.

\emph{Step 2: an internal smooth approximation preserving the collar.}
Choose \(K\Subset\operatorname{int}K'\) with \(K'\) compact.  The
global transition region is contained in the regular open set
\[
 \{\tfrac12\varepsilon<r_{\rm Rm}<3\varepsilon\}.
\]
Cover it by a locally finite family of regular coordinate balls and
mollify \(\eta_\varepsilon\) there using nonnegative mollifiers and a
convex smooth partition of unity.
Only finitely many balls meet \(K'\).  Leave the
function unchanged on
\(\{r_{\rm Rm}\le\tfrac34\varepsilon\}\) and on
\(\{r_{\rm Rm}\ge3\varepsilon\}\).  Convexity preserves the range
\([0,1]\) and the constant zero and one collars.  By density of smooth
functions in weighted \(W^{1,\gamma}\), choose the finitely many
relevant mollification scales so that the
\(W^{1,\gamma}(K',\mathfrak m)\)-error is at most
\(o(\varepsilon^{p-\gamma})\).  Denote the resulting function by
\(\chi_\varepsilon\).

Thus \(\chi_\varepsilon|_{\cR}\) is smooth, its zero extension is locally
Lipschitz on \(Z\), it vanishes on a neighbourhood of
\(K\cap\Sing Z\), and it equals one on every fixed compact subset of
\(\cR\) for all sufficiently small \(\varepsilon\).  Equations
\eqref{sp:eq:capacityestimate}--\eqref{sp:eq:capacity-L2} and the
chosen approximation imply
\[
 \int_K|\nabla\chi_\varepsilon|^\gamma\,d\mathfrak m\longrightarrow0,
 \qquad
 \int_K|1-\chi_\varepsilon|^2\,d\mathfrak m\longrightarrow0.
\]
H\"older's inequality, using \(\mathfrak m(K)<\infty\), then gives
\[
 \int_K|\nabla\chi_\varepsilon|^2\,d\mathfrak m
 \le \mathfrak m(K)^{1-2/\gamma}
 \left(\int_K|\nabla\chi_\varepsilon|^\gamma\,d\mathfrak m\right)^{2/\gamma}
 \longrightarrow0.
\]
Taking \(\varepsilon=\varepsilon_j\downarrow0\) proves
\eqref{sp:eq:capcutoff}--\eqref{sp:eq:q-capenergy}.

\emph{Step 3: radial exhaustion and diagonal approximation.}
For the global sequence, first note that \textup{(H1)} makes $Z$ a
complete locally compact length space.  Hence $Z$ is proper, so its
closed bounded balls are compact.  Fix $o\in\cR$ and choose a Lipschitz
distance cutoff $\theta_R$ which is one on $B(o,R)$, zero outside
$B(o,2R)$, and satisfies $|\nabla\theta_R|\le2/R$.  Multiply it by one
of the preceding smooth singular cutoffs $\chi_\varepsilon$ on
$\overline{B(o,2R)}$.  The product
\(v_{R,\varepsilon}=\theta_R\chi_\varepsilon\) is compactly supported a
positive distance from $\Sing Z$.  Range-preserving smoothing with
nonnegative local mollifiers and a convex partition of unity produces a
function in $C_c^\infty(\cR)$ with values in $[0,1]$ and arbitrarily
small $L^\infty$ and $W^{1,2}$ error.  Before this approximation, the
energy of $v_{R,\varepsilon}$ is bounded by twice
the singular-cutoff energy plus
\[
              \frac{8}{R^2}\mathfrak m(B(o,2R))\le\frac8{R^2}.
\]
Its \(L^2\) error also satisfies
\[
 \|1-v_{R,\varepsilon}\|_2^2
 \le 2\int_{B(o,R)}|1-\chi_\varepsilon|^2\,d\mathfrak m
     +2\mathfrak m\bigl(Z\setminus B(o,R)\bigr).
\]
Choose $R_j\to\infty$ so that
$\mathfrak m(Z\setminus B(o,R_j))\to0$.  For each fixed $R_j$, choose
$\varepsilon_j\le j^{-1}$ so that the singular-cutoff $L^2$ error and
energy on $B(o,2R_j)$ are both at most $2^{-j}$.  Smooth
$v_{R_j,\varepsilon_j}$ as above to obtain
$\beta_j\in C_c^\infty(\cR)$, $0\le\beta_j\le1$, with
\[
 \|\beta_j-v_{R_j,\varepsilon_j}\|_{L^\infty}
 +\|\beta_j-v_{R_j,\varepsilon_j}\|_{W^{1,2}(\mathfrak m)}
 \le2^{-j}.
\]
Together with the radial energy bound $8/R_j^2$, this proves both norm
limits in \eqref{sp:eq:globalcutoff}.  For every fixed $x\in\cR$, positivity
of $r_{\rm Rm}(x)$ and $R_j\to\infty$, $\varepsilon_j\to0$ give
$v_{R_j,\varepsilon_j}(x)=1$ for all large $j$, so the $L^\infty$
approximation gives $\beta_j(x)\to1$.  Since
$\mathfrak m(\Sing Z)=0$, the asserted almost-everywhere convergence
follows as well.
\end{proof}

The local cutoffs above remove the singular set on a fixed compact set.
The global cutoffs \(\beta_j\) additionally cut off spatial infinity.
The probability normalization \(\mathfrak m(Z)=1\) is what makes the
radial energy cost tend to zero.

Define on the complex Hilbert space $L^2(Z,\mathfrak m)$
\begin{equation}\label{sp:eq:E0}
 \cE_0(\phi,\psi)=
 \int_{\cR}\langle\nabla\phi,\nabla\overline\psi\rangle\,d\mathfrak m,
 \qquad \phi,\psi\in C_c^\infty(\cR).
\end{equation}
Here the Riemannian tangent-vector pairing is extended complex
bilinearly, and we
use the convention
\[
 \langle v,w\rangle_{L^2(\mathfrak m)}
 =\int_Zv\overline w\,d\mathfrak m;
\]
thus the \(L^2\) inner product is linear in its first argument.
Because $\mathfrak m(\Sing Z)=0$, we identify
$L^2(Z,\mathfrak m)=L^2(\cR,\mathfrak m)$.  On the smooth second-countable
manifold $\cR$, endowed with its smooth positive density,
$C_c^\infty(\cR)$ is dense in $L^2(\cR,\mathfrak m)$.  The gradient
operator on $C_c^\infty(\cR)$ is closable: if $\phi_j\to0$ and
$\nabla\phi_j\to V$ in $L^2$, testing against compactly supported smooth
vector fields gives $V=0$ distributionally.  Hence $\cE_0$ is closable.
Denote its closure by
$(\cE,D(\cE))$ and its nonnegative self-adjoint generator by $A$:
\begin{equation}\label{sp:eq:representation}
 \cE(u,v)=\langle Au,v\rangle_{L^2}
 \qquad(u\in\Dom A,\ v\in D(\cE)).
\end{equation}
Equivalently,
\[
 \Dom A=\bigl\{u\in D(\cE):\text{there is }h\in L^2
 \text{ such that }\cE(u,v)=\langle h,v\rangle
 \text{ for every }v\in D(\cE)\bigr\},
 \qquad Au=h.
\]
On $C_c^\infty(\cR)$, $A=-\Delta_f$, where, in the real normalization
fixed in \eqref{sp:eq:normalization-conventions},
$\Delta_f=\Delta_g-\langle\nabla^g f,\nabla^g\!\cdot\rangle_g$.
Thus \(A\) is nonnegative and the forward weighted heat semigroup is
\(e^{-tA}\), \(t\ge0\).

\begin{lemma}
\label{sp:lem:kernel}
The constant function one lies in $D(\cE)$ and $\cE(1,1)=0$.
Moreover, $\cR$ is connected and
\begin{equation}\label{sp:eq:kernel}
                              \ker A=\C1.
\end{equation}
\end{lemma}

\begin{proof}
By \eqref{sp:eq:globalcutoff},
\[
 \|\beta_j-\beta_k\|_2\le
 \|1-\beta_j\|_2+\|1-\beta_k\|_2,
 \qquad
 \cE_0(\beta_j-\beta_k,\beta_j-\beta_k)^{1/2}
 \le\|\nabla\beta_j\|_2+\|\nabla\beta_k\|_2.
\]
Thus $(\beta_j)$ is form-Cauchy with limit $1$, so
$1\in D(\cE)$ and $\cE(1,1)=0$.  Cauchy--Schwarz for the closed form
then gives $\cE(1,v)=0$ for every $v\in D(\cE)$.  By the representation
theorem, $1\in\Dom A$ and $A1=0$.

By \textup{(H1)}, the regular locus \(\cR\) is connected.  If
\(u\in\ker A\), then
\[
 0=\langle Au,u\rangle=\cE(u,u)
   =\int_{\cR}|\nabla u|^2\,d\mathfrak m.
\]
Thus $u$ is almost everywhere constant on the connected regular locus,
and hence is a constant element of $L^2(Z,\mathfrak m)$.  Together with
$A1=0$, this proves \eqref{sp:eq:kernel}.
\end{proof}

\begin{remark}[Scope of the Friedrichs construction]
The construction supplies the Friedrichs generator, includes constants
in its form domain, and permits first-order integration by parts across
the singular set.  It does not assert essential self-adjointness of the
minimal second-order operator, uniqueness among Markovian extensions,
or equality with every maximal Sobolev domain.  None of these stronger
properties is used.
\end{remark}

Thus the weighted drift Laplacian has a canonical self-adjoint
realization.  We next obtain its sharp spectral gap directly from the
original smooth Ricci flow.

\medskip
\noindent\emph{The sharp tangent-space Poincar\'e inequality.}\par\smallskip

The intrinsic spectral result proved in this subsection is the following.

\begin{theorem}
\label{sp:thm:spectral-rigidity}
Let \(\xi=(\mu_t)_{t<T}\) be a fixed limiting point of a compact
finite-time K\"ahler--Ricci flow of complex dimension \(n\ge2\), and
let \((Z,d,\cR,g,J,f,\mathfrak m)\) be the time-\((-1)\) model of a
tangent space at \(\xi\).  Thus \(Z\) has
the finite-time tangent-space package of
\cref{sp:prop:limiting-package}: it is the intrinsic completion
of its K\"ahler regular locus, is normal and klt, carries the normalized
shrinker measure and complete soliton flow, and has smooth
metric--complex--density convergence and the local ambient metric lower
bound.  Let \(A\) be the nonnegative
self-adjoint operator associated with the closure of
\[
 \cE_0(\phi,\psi)=
 \int_{\cR}\langle\nabla\phi,\nabla\overline\psi\rangle\,d\mathfrak m,
 \qquad \phi,\psi\in C_c^\infty(\cR).
\]
Then
\begin{equation}\label{sp:eq:globalgap}
                     \sigma(A)\subset\{0\}\cup[1/2,\infty),
                     \qquad\ker A=\C1.
\end{equation}
If \(u:Z\to\C\) is holomorphic and globally \(d\)-Lipschitz, then
\(u\in L^2(Z,\mathfrak m)\), and
\begin{equation}\label{sp:eq:halfmode}
 A\left(u-\int_Zu\,d\mathfrak m\right)
 =\frac12\left(u-\int_Zu\,d\mathfrak m\right),
\end{equation}
and
\begin{equation}\label{sp:eq:hesszero}
          \nabla^2\Re u=\nabla^2\Im u=0\qquad\text{on }\cR.
\end{equation}
Here \(A\) is the canonical Friedrichs form generator; no
essential-self-adjointness assertion is made for the minimal operator on
the incomplete regular locus.
\end{theorem}

The \(L^2\) assertion follows from the finite second moment of
\(\mathfrak m\): a globally Lipschitz function has at most linear
growth.  In particular, every weighted mean appearing above is finite.

For a probability measure \(\mu\) and \(\phi\in L^2(\mu)\), set
\begin{equation}\label{eq:function-variance}
 \Var_\mu(\phi)
 :=\int\left|\phi-\int\phi\,d\mu\right|^2d\mu.
\end{equation}
We shall use the elementary identities
\[
 \Var_\mu(\phi)
 =\int|\phi|^2\,d\mu-\left|\int\phi\,d\mu\right|^2
 =\inf_{c\in\C}\int|\phi-c|^2\,d\mu,
\]
where the infimum is attained at \(c=\int\phi\,d\mu\).

We next prove
\begin{equation}\label{sp:eq:target-poincare-guide}
 \Var_{\mathfrak m}(\phi)
 \le2\int_{\cR}|\nabla\phi|^2\,d\mathfrak m.
\end{equation}
For \(A=-\Delta_f\), this is the spectral lower bound \(A\ge1/2\) on
the orthogonal complement of the constants.  The proof follows the
concrete chain
\[
 \begin{gathered}
 \text{point conjugate heat-kernel limits}
 \Longrightarrow\text{fixed limiting point}\\
 \Longrightarrow\text{time-\((-1)\) tangent-space model}
 \Longrightarrow\text{Friedrichs spectral gap}.
 \end{gathered}
\]
The sharp constant comes from Ricci-flow heat propagation, not from a
formal integration by parts across \(\Sing Z\).

Let $P_{a,b}$ denote the forward heat propagator from time $a$ to time
$b$ on a smooth compact Ricci flow satisfying
\eqref{eq:standard-real-metric}.  Thus
$u(s)=P_{a,s}\phi$ solves $\partial_su=\Delta_{g(s)}u$.

\begin{lemma}
\label{sp:lem:pointpoincare}
Let $d\nu_{x,b;a}$ be the conjugate heat-kernel probability at time $a$
with pole $(x,b)$, where $a<b<T$.  Then every smooth real or complex
function $\phi$ satisfies
\begin{equation}\label{sp:eq:pointP}
 \int\left|\phi-\int\phi\,d\nu_{x,b;a}\right|^2d\nu_{x,b;a}
 \le2(b-a)\int|\nabla\phi|_{g(a)}^2\,d\nu_{x,b;a}.
\end{equation}
\end{lemma}

\begin{proof}
It is enough to prove the real assertion and then apply it to real and
imaginary parts, or to use the Hermitian polarization.  Put
 $u(s)=P_{a,s}\phi$ and
\[
                    G(s)=P_{s,b}(u(s)^2)(x).
\]
The derivative of a propagator with respect to its lower time is
\[
 \frac d{ds}P_{s,b}\psi(s)=
 P_{s,b}\bigl((\partial_s-\Delta_{g(s)})\psi(s)\bigr).
\]
Consequently
\begin{equation}\label{sp:eq:Gprime}
                         G'(s)=-2P_{s,b}|\nabla u(s)|^2(x).
\end{equation}
Under \eqref{eq:standard-real-metric}, the Ricci terms cancel in the
gradient evolution:
\begin{equation}\label{sp:eq:gradientheat}
 (\partial_s-\Delta_{g(s)})|\nabla u|^2=-2|\nabla^2u|^2\le0.
\end{equation}
The maximum principle gives
\[
                 |\nabla u(s)|^2\le P_{a,s}|\nabla\phi|^2.
\]
Integrating \eqref{sp:eq:Gprime}, and using the semigroup law, yields
\begin{align*}
 P_{a,b}(\phi^2)(x)-(P_{a,b}\phi(x))^2
 &=2\int_a^bP_{s,b}|\nabla P_{a,s}\phi|^2(x)\,ds\\
 &\le2(b-a)P_{a,b}|\nabla\phi|^2(x).
\end{align*}
By forward--conjugate duality this is \eqref{sp:eq:pointP}.
\end{proof}

\begin{lemma}
\label{sp:lem:tangentP}
For the fixed limiting point \(\xi=(\mu_t)_{t<T}\) in
\cref{sp:thm:spectral-rigidity}, every $t<T$ and every smooth $\phi$ on
$X$ satisfies
\begin{equation}\label{sp:eq:limiting-point-P}
 \Var_{\mu_t}(\phi)
 \le2(T-t)\int_X|\nabla\phi|_{g(t)}^2\,d\mu_t.
\end{equation}
If \((Z,d,\cR,g,\mathfrak m)\) is the time-\((-1)\) model of any tangent
space at \(\xi\), then
\begin{equation}\label{sp:eq:tangentP}
 \int_Z\left|\phi-\int_Z\phi\,d\mathfrak m\right|^2d\mathfrak m
 \le2\int_{\cR}|\nabla\phi|_g^2\,d\mathfrak m
 \qquad(\phi\in C_c^\infty(\cR)).
\end{equation}
\end{lemma}

\begin{proof}
By property \textup{(P1)}, $\mu_t$ is the weak limit, at the fixed time
$t$, of point conjugate kernels with pole times $b_j\nearrow T$.
Apply \cref{sp:lem:pointpoincare} with $a=t$, and pass to the limit.  Since
$X$ is compact and $\phi$ and $|\nabla\phi|^2$ are continuous, all three
integrals converge.  This proves \eqref{sp:eq:limiting-point-P}.

Set $t_i=T-\tau_i$.  Since
$g_i(-1)=\tau_i^{-1}g(t_i)$,
\[
       |\nabla\phi|_{g(t_i)}^2
       =\tau_i^{-1}|\nabla\phi|_{g_i(-1)}^2.
\]
Thus \eqref{sp:eq:limiting-point-P} becomes
\begin{equation}\label{sp:eq:scaledP}
 \Var_{\mu_{i,-1}}(\phi)
 \le2\int_X|\nabla\phi|_{g_i(-1)}^2\,d\mu_{i,-1}.
\end{equation}

Now fix $\phi\in C_c^\infty(\cR)$.  Choose regular exhaustion domains
$K_1\Subset\operatorname{int}K_2\Subset\cR$ with
$\supp\phi\Subset\operatorname{int}K_1$, and choose
$\chi\in C_c^\infty(\operatorname{int}K_2)$ equal to one on $K_1$.
Use regular convergence embeddings $\Phi_i:K_2\to X$ and define
$\phi_i$ on $\Phi_i(K_2)$ by
$\phi_i\circ\Phi_i=\chi\phi$.  Since $\chi\phi$ vanishes near
$\partial K_2$, its pushforward extends smoothly by zero to $X$ and
remains supported in the convergence region.  Smooth
metric and density convergence gives
\begin{align*}
 \int\phi_i\,d\mu_{i,-1}&\longrightarrow\int\phi\,d\mathfrak m,\\
 \int|\phi_i|^2\,d\mu_{i,-1}&\longrightarrow\int|\phi|^2\,d\mathfrak m,\\
 \int|\nabla\phi_i|_{g_i(-1)}^2\,d\mu_{i,-1}
 &\longrightarrow\int_{\cR}|\nabla\phi|_g^2\,d\mathfrak m.
\end{align*}
Passing \eqref{sp:eq:scaledP} to the limit proves \eqref{sp:eq:tangentP}.
Notice that the inequality comes from a uniform prelimit estimate; it is
not inferred from metric convergence alone.
\end{proof}

\begin{proposition}
\label{sp:prop:gap}
Inequality \eqref{sp:eq:tangentP} extends to every $u\in D(\cE)$, and the
Friedrichs generator satisfies \eqref{sp:eq:globalgap}.
\end{proposition}

\begin{proof}
Form-core approximation and convergence in $L^2$ extend
\eqref{sp:eq:tangentP} to $D(\cE)$.  By \cref{sp:lem:kernel}, constants lie
in the form domain and are exactly the kernel.  On the orthogonal
complement of constants, the extended inequality reads
\begin{equation}\label{sp:eq:formgap}
                           \cE(u,u)\ge\frac12\|u\|_2^2.
\end{equation}
The representation theorem for closed forms and the spectral theorem
therefore put the spectrum of $A$ on this complement in
$[1/2,\infty)$.  Together with \eqref{sp:eq:kernel}, this is
\eqref{sp:eq:globalgap}.
\end{proof}

This proves the spectral-gap part of the rigidity theorem.  To extract
geometry from the gap, we now show that every globally Lipschitz
holomorphic function lies exactly in the lowest nonconstant eigenspace.

\paragraph{Lipschitz holomorphic functions and the half mode.}

Let \(u:Z\to\C\) be globally Lipschitz and holomorphic.  We want to
prove that its nonconstant part is a \(1/2\)-eigenfunction of
\(A=-\Delta_f\).  Since \(u\) is not compactly supported and \(\cR\) is
incomplete, this does not follow from a formal spectral expansion.
Instead, the complete soliton flow
\(\Psi_a=\operatorname{Fl}_{\nabla f}^a\) produces the ancient family
\[
                         U(s,x)=u(\Psi_{-s}x),
                         \qquad s\le0.
\]
Holomorphicity makes \(u\) harmonic on \(\cR\), so locally
\(\partial_sU=\Delta_fU=-AU\).  The proof proceeds in five stages:
control growth under \(\Psi_a\), verify the local equation, remove the
singular set and infinity in the energy identity, identify \(U\) with
the forward Friedrichs semigroup, and use the sharp gap followed by the
Bochner equality case.

\begin{lemma}
\label{sp:lem:flowgrowth}
There are $o\in\cR$ and $B<\infty$ such that, with
$r(x)=d(o,x)$,
\begin{equation}\label{sp:eq:fupper}
                         |\nabla f|(x)\le\frac12r(x)+B
                         \qquad(x\in\cR).
\end{equation}
If $\Psi_a=\operatorname{Fl}_{\nabla f}^a$, then for every $a\ge0$,
\begin{equation}\label{sp:eq:flowdistance}
                     r(\Psi_a x)+2B\le e^{a/2}(r(x)+2B)
\end{equation}
after increasing $B$ if necessary.
\end{lemma}

\begin{proof}
Estimate \eqref{sp:eq:fupper} is
\eqref{sp:eq:preliminary-flow-growth}, established in the proof of
\cref{sp:prop:limiting-package}.  By
\cref{sp:prop:limiting-package}\textup{(H8)}, \(\Psi_a\) is defined for
every \(a\in\mathbb R\).  Along an integral curve of
\(\nabla f\), the upper Dini derivative of \(r\) is at most
\(|\nabla f|\).  Hence
\[
                 \frac d{da}(r(\Psi_ax)+2B)
                 \le\frac12(r(\Psi_ax)+2B).
\]
Gronwall's inequality proves \eqref{sp:eq:flowdistance}.
\end{proof}

\begin{lemma}
\label{sp:lem:energysolution}
Write \(D(\cE)^*\) for the continuous anti-dual of the form domain.
With the preceding first-slot-linear convention, the standard Gelfand
triple is
\[
                 D(\cE)\subset L^2(\mathfrak m)\subset D(\cE)^*.
\]
Let $u:Z\to\C$ be holomorphic and globally Lipschitz.  For $s\le0$ set
\begin{equation}\label{sp:eq:Udef}
                           U(s,x)=u(\Psi_{-s}x).
\end{equation}
Then
\begin{equation}\label{sp:eq:Ugrowth}
                         \|U(s)\|_{L^2(\mathfrak m)}\le Ce^{-s/2},
                         \qquad s\le0,
\end{equation}
and on every finite interval $[s_0,s_1]\subset(-\infty,0]$,
\begin{equation}\label{sp:eq:energyspace}
 U\in L^2([s_0,s_1];D(\cE)),\qquad
 \partial_sU+A_{\rm form}U=0\quad\text{in }D(\cE)^*,
\end{equation}
where $\langle A_{\rm form}v,\phi\rangle=\cE(v,\phi)$ is the form
operator.  On $\Dom A$, $A_{\rm form}$ agrees with the self-adjoint
operator $A$.
Moreover $s\mapsto U(s)$ is continuous into $L^2(\mathfrak m)$ and
\begin{equation}\label{sp:eq:semigroupidentity}
                              U(0)=e^{sA}U(s),
                              \qquad s<0.
\end{equation}
\end{lemma}

\begin{proof}
\emph{Step 1: growth.}
Global Lipschitz continuity gives $|u(x)|\le C(1+r(x))$.  Combining this
with \eqref{sp:eq:flowdistance} and the finite second moment of
$\mathfrak m$ yields
\[
 \|U(s)\|_2^2
 \le Ce^{-s}\int_Z(1+r)^2\,d\mathfrak m
 \le C'e^{-s},
\]
which is \eqref{sp:eq:Ugrowth}.

\emph{Step 2: the local equation on the regular locus.}
The flow $\Psi_a$ is biholomorphic and preserves $\cR$, so every
$U(s,\cdot)$ is holomorphic and hence harmonic on the regular K\"ahler
manifold.  Since $\nabla f$ is invariant under its own flow,
\begin{equation}\label{sp:eq:localUeq}
 \partial_sU=-dU(\nabla f)=\Delta_fU
 \qquad\text{pointwise on }\cR.
\end{equation}

\emph{Step 3: removal of the singular set and spatial infinity.}
We verify the global energy class rather than regarding the incomplete
regular locus as a complete manifold.  Fix
$[s_0,s_1]\Subset(-\infty,0]$.  Choose a Lipschitz distance cutoff
$\theta_R$ that equals one on $B(o,R)$, vanishes outside $B(o,2R)$, and
satisfies $|\nabla\theta_R|\le2/R$.  Smooth it on $\cR$ with an
arbitrarily small change in these bounds.  Let $\eta_j$ be a singular cutoff on
$\overline{B(o,2R)}$.  Test \eqref{sp:eq:localUeq} with
$\theta_R^2\eta_j^2\overline U$ and take real parts.  Local smoothness
and weighted integration by parts give, with
\(\beta=\theta_R\eta_j\),
\[
 \Re\int\beta^2\overline U\,\Delta_fU\,d\mathfrak m
 =-\int\beta^2|\nabla U|^2\,d\mathfrak m
  -2\Re\int\beta\overline U
       \langle\nabla\beta,\nabla U\rangle\,d\mathfrak m.
\]
Young's inequality bounds the last term by
\[
 \frac12\int\beta^2|\nabla U|^2\,d\mathfrak m
 +2\int|U|^2|\nabla\beta|^2\,d\mathfrak m.
\]
Consequently
\begin{equation}\label{sp:eq:Uenergyineq}
 \frac12\frac d{ds}\int\theta_R^2\eta_j^2|U|^2\,d\mathfrak m
 +\frac12\int\theta_R^2\eta_j^2|\nabla U|^2\,d\mathfrak m
 \le2\int|U|^2|\nabla(\theta_R\eta_j)|^2\,d\mathfrak m.
\end{equation}
On a fixed ball and finite time interval, $U$ is uniformly bounded, so
the singular-cutoff part on the right tends to zero by
\eqref{sp:eq:capenergy}.  The radial part is bounded by
\[
 \frac{C}{R^2}\sup_{s\in[s_0,s_1]}\|U(s)\|_2^2,
\]
which tends to zero as $R\to\infty$ by \eqref{sp:eq:Ugrowth}.  Crucially,
integrating \eqref{sp:eq:Uenergyineq} retains the initial-time term:
\begin{align}
 &\frac12\int\theta_R^2\eta_j^2|U(s_1)|^2\,d\mathfrak m
 +\frac12\int_{s_0}^{s_1}\!\!\int
       \theta_R^2\eta_j^2|\nabla U|^2\,d\mathfrak m\,ds \notag\\
 &\qquad\le
 \frac12\int\theta_R^2\eta_j^2|U(s_0)|^2\,d\mathfrak m
 +2\int_{s_0}^{s_1}\!\!\int
       |U|^2|\nabla(\theta_R\eta_j)|^2\,d\mathfrak m\,ds.
 \label{sp:eq:integrated-Uenergy}
\end{align}
The initial term is at most $\|U(s_0)\|_2^2/2$.  First let
$j\to\infty$, then $R\to\infty$, and take a diagonal subsequence.
Fatou's lemma in \eqref{sp:eq:integrated-Uenergy} gives global spacetime
integrability of $|\nabla U|^2$.  Choose both cutoffs smooth on $\cR$, as allowed in
\cref{sp:lem:capacity}.  For almost every $s$, the functions
$\theta_R\eta_jU(s)$ lie in $C_c^\infty(\cR)$ and converge to $U(s)$ in
the form norm: the singular cross term is bounded by
\[
 \sup_{B(o,2R)}|U(s)|^2\int|\nabla\eta_j|^2\,d\mathfrak m,
\]
the radial cross term is bounded by
$CR^{-2}\|U(s)\|_2^2$, and the remaining gradient tails vanish by the
just-proved global gradient integrability.  The same estimates integrated
in $s$ prove the first assertion in \eqref{sp:eq:energyspace}.

Initially test \eqref{sp:eq:localUeq} in time and space against
$C_c^\infty(\cR)$.  Form density then extends it, in the
time-distributional sense, to
\[
 \langle\partial_sU,\phi\rangle_{D(\cE)^*,D(\cE)}
      =-\cE(U,\phi),
 \qquad\phi\in D(\cE).
\]
This proves the weak equation.  On a finite $s$-interval,
\eqref{sp:eq:flowdistance} supplies an $L^2$ dominating function for
$U(s)$; pointwise continuity of the flow and dominated convergence make
$U$ continuous in $L^2$.

\emph{Step 4: uniqueness and identification with the semigroup.}
Fix $s<0$ and let
\[
                   V(t)=e^{-(t-s)A}U(s),\qquad s\le t\le0.
\]
This is the Friedrichs energy solution with initial value $U(s)$.  The
difference $W=U-V$ has zero initial value and satisfies the same weak
equation.  The standard energy identity in the Hilbert triple
$D(\cE)\subset L^2(\mathfrak m)\subset D(\cE)^*$, justified by Steklov
averaging in time, gives
\[
                   \frac12\frac d{dt}\|W(t)\|_2^2
                   =-\cE(W(t),W(t))\le0.
\]
Hence $W\equiv0$.  At $t=0$ this is
\eqref{sp:eq:semigroupidentity}.  It is forward evolution from the earlier
time $s$; no backward heat problem has been asserted.
\end{proof}

\begin{remark}[Direction of the heat evolution]
For \(s<0\), the identity in \eqref{sp:eq:semigroupidentity} reads
\[
                         U(0)=e^{sA}U(s)=e^{-(-s)A}U(s).
\]
The duration \(-s\) is positive.  Thus this is ordinary forward
Friedrichs heat evolution from the earlier time \(s\) to time zero, not
an ill-posed backward heat equation.
\end{remark}

\begin{proposition}
\label{sp:prop:spectralcutoff}
For every globally Lipschitz holomorphic \(u:Z\to\C\),
\[
 A\left(u-\int_Zu\,d\mathfrak m\right)
 =\frac12\left(u-\int_Zu\,d\mathfrak m\right).
\]
\end{proposition}

\begin{proof}
For $b>1/2$, let $P_b=\mathbf1_{[b,\infty)}(A)$.  Since $s<0$,
functional calculus, \eqref{sp:eq:semigroupidentity}, and
\eqref{sp:eq:Ugrowth} give
\begin{align}
 \|P_bu\|_2
 &=\|P_be^{sA}U(s)\|_2
 \le e^{sb}\|U(s)\|_2
 \le Ce^{s(b-1/2)}.\label{sp:eq:spectralcutoff}
\end{align}
Letting $s\to-\infty$ shows $P_bu=0$ for every $b>1/2$.
Therefore the spectral measure of $u$ is supported in $[0,1/2]$.
By \cref{sp:prop:gap}, only the constant subspace and the spectral point
$1/2$ remain.  The constant projection is
$\int_Zu\,d\mathfrak m$ because $\mathfrak m$ is a probability and
$\ker A=\C1$.  The nonconstant part is consequently in the
$1/2$-eigenspace, proving \eqref{sp:eq:halfmode}.
\end{proof}

\begin{proposition}
\label{sp:prop:hesszero}
Every globally Lipschitz holomorphic \(u:Z\to\C\) satisfies
\[
 \nabla^2\Re u=\nabla^2\Im u=0
 \qquad\text{on }\cR.
\]
\end{proposition}

\begin{proof}
Put
\[
                         c=\int_Zu\,d\mathfrak m,
\]
and let \(h=\Re(u-c)\) or \(h=\Im(u-c)\).  By
\cref{sp:prop:spectralcutoff},
\begin{equation}\label{sp:eq:heigen}
                       \Delta_fh=-\frac12h
                       \qquad\text{on }\cR.
\end{equation}
The form eigenvalue equation first gives this identity weakly against
$C_c^\infty(\cR)$.  Since the holomorphic function is already smooth on
$\cR$, local elliptic regularity makes the identity pointwise there.
The weighted Bochner formula and \eqref{sp:eq:shrinker} give pointwise
\begin{align}
 \frac12\Delta_f|\nabla h|^2
 &=|\nabla^2h|^2+
   (\Rc+\nabla^2f)(\nabla h,\nabla h)
   +\langle\nabla h,\nabla\Delta_fh\rangle\notag\\
 &=|\nabla^2h|^2.\label{sp:eq:bochnerequality}
\end{align}
Choose $\theta_R$ and $\eta_j$ exactly as in the proof of
\cref{sp:lem:energysolution}, and put $\beta=\theta_R\eta_j$.  Weighted
integration by parts,
$|\nabla|\nabla h|^2|\le2|\nabla^2h||\nabla h|$, and the global
Lipschitz bound yield the following estimate.  Indeed, by \textup{(H1)},
\(d|_{\cR}\) is the intrinsic length distance, so
\(|\nabla h|\le L:=\operatorname{Lip}(u)\) on \(\cR\).  Therefore
\begin{align*}
 I_{R,j}:=\int\beta^2|\nabla^2h|^2\,d\mathfrak m
 &=-\int\beta\langle\nabla\beta,
             \nabla|\nabla h|^2\rangle\,d\mathfrak m\\
 &\le2L I_{R,j}^{1/2}
       \left(\int|\nabla\beta|^2\,d\mathfrak m\right)^{1/2}.
\end{align*}
Thus
\begin{equation}\label{sp:eq:hesscutoff}
            I_{R,j}\le4L^2\int|\nabla\beta|^2\,d\mathfrak m.
\end{equation}
At fixed \(R\),
\[
 \int|\nabla(\theta_R\eta_j)|^2\,d\mathfrak m
 \le2\int\eta_j^2|\nabla\theta_R|^2\,d\mathfrak m
    +2\int\theta_R^2|\nabla\eta_j|^2\,d\mathfrak m.
\]
Let the singular scale tend to zero.  The second term vanishes by
\cref{sp:lem:capacity}, while the first is at most \(8R^{-2}\) because
\(|\nabla\theta_R|\le2/R\) and \(\mathfrak m(Z)=1\).  Hence the
right-hand side of \eqref{sp:eq:hesscutoff} is at most \(CL^2/R^2\).
If $K_0\Subset\cR\cap B(o,R)$, local uniform convergence
$\eta_j\to1$ and Fatou's lemma give
\[
       \int_{K_0}|\nabla^2h|^2\,d\mathfrak m
       \le\liminf_{j\to\infty}I_{R,j}\le\frac{CL^2}{R^2}.
\]
Now let $R\to\infty$ and then exhaust $\cR$.  This proves
$\nabla^2h=0$, the proposition, and
\cref{sp:thm:spectral-rigidity}.
\end{proof}

Consequently, the real and imaginary gradients of every limiting base
coordinate will be parallel on the regular locus.

\subsection{Global splitting from parallel holomorphic maps}
\label{sp:sec:globalization}

The spectral rigidity obtained above has the following global
consequence.

\begin{theorem}
\label{sp:thm:parallel-product}
In the setting of \cref{sp:thm:spectral-rigidity}, let
\(F=(F^1,\ldots,F^k):Z\to\C^k\) be holomorphic and globally Lipschitz,
and let
\begin{equation}\label{sp:eq:rankmodconstants}
 r=\dim_\C
 \frac{\Span_\C\{F^1,\ldots,F^k\}+\C1}{\C1}.
\end{equation}
Then, after selecting \(r\) complex-linear
combinations and making a complex-linear target change, there is a global
biholomorphic isometry
\[
             (Z,d,J)\cong(\C^r,g_{\rm E},J_0)\times(Z',d',J')
\]
under which the selected map is the projection.  The residual factor is
complete, normal, and klt, and
\begin{equation}\label{sp:eq:singproduct}
                         \Sing Z=\C^r\times\Sing Z'.
\end{equation}
For the tangent shrinker potential and its canonical ancient flow, the
corresponding potential, weighted-mass, and negative-time conclusions of
\cref{sp:thm:singular-product} also hold with \(m\) replaced by \(r\).
\end{theorem}

The preceding spectral argument produces parallel fields only on
\(\cR=\Reg Z\).  A de Rham argument on this incomplete manifold would
not identify either the metric completion or the complex structure at
singular points.  We therefore construct the global product directly.
The proof has six stages: normalize the constant Gram matrix, extend the
parallel holomorphic vector fields through \(\Sing Z\), prove their flows
complete, integrate them to a \(\C^r\)-action, identify the reduced zero
fibre with the residual factor, and finally split the shrinker potential
and its negative-time flow.

\begin{lemma}
\label{sp:lem:klt}
For a normal complex space $V$, the space $V$ is klt if and only if
$\C^r\times V$ is klt.
\end{lemma}

\begin{proof}
The assertion is local on $V$.  Let $p_2:\C^r\times V\to V$ be
projection.  Since the canonical bundle of $\C^r$ is trivial, for every
positive integer $q$ the reflexive canonical powers satisfy
\[
                 \omega_{\C^r\times V}^{[q]}
                 \cong p_2^*\omega_V^{[q]}.
\]
Hence, if $K_V$ is $\mathbb Q$-Cartier, so is the product.
Conversely, the zero slice is a regular embedding with trivial normal
bundle, so reflexive adjunction restricts a Cartier power of the product
canonical divisor to the corresponding Cartier power of $K_V$.  Hence
$V$ is $\mathbb Q$-Gorenstein exactly when the product is.

Choose a log resolution $\nu:\widetilde V\to V$ over the local
neighborhood under consideration.  Its smooth base change
$\Id_{\C^r}\times\nu$ resolves the product.  Let
\(\widetilde p_2:\C^r\times\widetilde V\to\widetilde V\) be the second
projection.  Then
\[
 K_{\C^r\times\widetilde V}
 -(\Id\times\nu)^*K_{\C^r\times V}
 =\widetilde p_2^*(K_{\widetilde V}-\nu^*K_V).
\]
Corresponding exceptional divisors have the same discrepancies.  The
condition that every discrepancy be greater than $-1$ is therefore
invariant under the smooth factor, proving the assertion.
\end{proof}

\begin{theorem}
\label{sp:thm:singular-product}\label{thm:singular-product}
Let \(Z\) be a connected normal complex space whose dense analytic
regular locus \(\cR\) carries a K\"ahler structure \((g,J)\).  Assume
that the analytic and metric topologies agree, that \(d|_{\cR}\) is the
intrinsic length distance, and that \((Z,d)\) is a complete locally
compact length space equal to the metric completion of \((\cR,g)\).
Assume also that \(Z\) is klt, that \(\omega_g\) extends locally as a
positive current with bounded potential, and that every point has a
local closed holomorphic embedding
\(V\hookrightarrow\Omega_V\subset\C^L\) for which
\begin{equation}\label{sp:eq:product-ambient-lower}
                     \omega_g\ge c_V\omega_{\C^L}
                     \quad\text{on }V\cap\cR.
\end{equation}
Let $G:Z\to\C^r$ be continuous and holomorphic, and
assume on $\cR$ that
\begin{equation}\label{sp:eq:parallelG}
 \nabla dG^a=0,\qquad
 \left[\langle dG^a,d\overline{G^b}\rangle_{g^{-1}}\right]_{a,b=1}^r
 \text{ is positive definite}.
\end{equation}
Then, after composing \(G\) with an element of
\(\mathrm{GL}(r,\C)\), without translating the target, there is a global
biholomorphic isometry
\begin{equation}\label{sp:eq:globalmetricproduct}
                         Z\cong\C^r\times Z'
\end{equation}
under which $G$ is the projection.  The factor $Z'$ is connected, complete,
normal, and klt; write \(g'\) for its induced regular K\"ahler metric.
It inherits the bounded-local-potential K\"ahler current
and the local ambient metric lower bound of $Z$, and
\begin{equation}\label{sp:eq:product-singular-set}
                         \Sing Z=\C^r\times\Sing Z'.
\end{equation}

If, in addition, $f$ is locally Lipschitz on $Z$, smooth on $\cR$, and
satisfies \(\Rc+\nabla^2f=g/2\) on $\cR$, then
\[
                 f(z,y)=\frac{|z|^2}{4}+\ell(z)+f'(y),
\]
where $\ell$ is real affine-linear and $f'$ is locally Lipschitz; if $f$
is proper, then $f'$ is proper.  If also
\begin{equation}\label{sp:eq:product-finite-mass}
                         \int_{\cR}e^{-f}\,dV_g<\infty,
\end{equation}
then \(\int_{\Reg Z'}e^{-f'}\,dV_{g'}<\infty\).  Finally, assume that
the flow of \(\nabla f\) is complete and preserves \(\cR\), and that its
canonical self-similar maps extend to the metric completions.  Then the
same Euclidean factor is static on every negative-time slice in that
canonical gauge.
\end{theorem}

In the proof, \(G^{-1}(0)_{\rm red}\) denotes the analytic zero fibre
with its reduced complex-space structure, that is, with nilpotent
elements of the structure sheaf removed.  This distinction is needed
only at the analytic level; the underlying zero set is unchanged.

\begin{proof}
\emph{Step 1: normalization and parallel fields.}
Because \(Z\) is connected and normal, its analytic regular locus
\(\cR\) is connected.  The Hermitian matrix
\[
             H_{a\bar b}
             =\langle dG^a,d\overline{G^b}\rangle_{g^{-1}}
\]
is parallel by \eqref{sp:eq:parallelG}, and is therefore a constant
positive-definite matrix.  Apply a complex-linear target change so that
$H$ is the standard Hermitian cometric of the Euclidean target
$(\C^r,g_{\rm E})$ in the normalization
\eqref{sp:eq:normalization-conventions}.  On $\cR$, let $W_a$ be the
corresponding $(1,0)$ Hermitian dual of $dG^a$.  Then
\begin{equation}\label{sp:eq:Wproperties}
 dG^b(W_a)=\delta_a^b,\qquad
 \nabla W_a=0,\qquad
 [W_a,W_b]=[W_a,\overline{W_b}]=0.
\end{equation}
K\"ahlerness makes $W_a$ holomorphic.  Its real and imaginary parts are
parallel real vector fields.
For the normalization used below, set
\begin{equation}\label{sp:eq:real-generators}
 X_a=W_a+\overline{W_a},\qquad
 Y_a=JX_a=\sqrt{-1}(W_a-\overline{W_a}).
\end{equation}
Then $X_a,Y_a$ are parallel real fields and
\begin{equation}\label{sp:eq:generator-coordinates}
       dG^b(X_a)=\delta_a^b,\qquad
       dG^b(Y_a)=\sqrt{-1}\,\delta_a^b.
\end{equation}

\emph{Step 2: extension through $\Sing Z$.}
Fix a closed local holomorphic embedding
\(\iota:V\hookrightarrow\Omega\subset\C^L\), let \(\zeta_A\) be the
ambient coordinates, and let
\(\mathcal I_V\subset\mathcal O_\Omega\) be the coherent ideal sheaf of
\(\iota(V)\).  By \eqref{sp:eq:product-ambient-lower},
\begin{equation}\label{sp:eq:ambientW}
 \sum_A|W_a(\zeta_A)|^2
 =|W_a|_{g_{\C^L}}^2
 \le c_V^{-1}|W_a|_g^2\le C_V.
\end{equation}
Each coefficient is a bounded holomorphic function on $V\cap\cR$ and
therefore extends uniquely over $V$ by the Riemann extension theorem for
normal complex spaces \cite[Chapter~7, Section~4]{GrauertRemmertSheaves}.
After shrinking \(\Omega\), the surjection
\(\mathcal O_\Omega\to\iota_*\mathcal O_V\) supplies ambient
holomorphic representatives $b_{aA}$.  Put
\[
                   \widetilde W_a=\sum_A b_{aA}\partial_{\zeta_A}.
\]
This ambient vector field preserves the ideal \(\mathcal I_V\).  Indeed,
if $h\in\mathcal I_V$, then $\widetilde W_a(h)$ vanishes on the dense regular locus
and hence on the reduced space $V$.  For local generators $h_j$ of
\(\mathcal I_V\), coherence gives
\[
                    \widetilde W_a(h_j)=\sum_kc_{jk}h_k.
\]
Along an ambient integral curve, the vector $(h_j)$ solves a homogeneous
linear ODE.  An integral curve starting in $V$ stays in $V$.  Hence the
extended derivation has a local holomorphic flow on $Z$.  Different
ambient representatives induce the same derivation of $\mathcal O_V$.
Uniqueness makes the local extensions and their flows agree on overlaps.
Each local flow is a biholomorphism and preserves
the regular locus.

\emph{Step 3: completeness and the $\C^r$ action.}
The real generators $X_a,Y_a$ have constant speed on
$\cR$.  The same Lipschitz-in-time estimate holds for an orbit starting
at a singular point.  Indeed, denote the local real flow of one generator
by \(\mathcal F_t\) and, on a common domain of the local flow, take
$x_j\in\cR$ with $x_j\to x$.  Continuity of the local flow, agreement
of the metric and analytic topologies, and the regular constant-speed
bound give
\[
 d(\mathcal F_t(x),\mathcal F_s(x))
 =\lim_{j\to\infty}d(\mathcal F_t(x_j),\mathcal F_s(x_j))
 \le C|t-s|.
\]
Subdivision through successive local-flow domains extends this estimate along
the whole orbit.  If a maximal real orbit had a finite endpoint time, it
would therefore be Cauchy.  Metric completeness would supply a limit in
$Z$, and a neighbourhood on which the local flow is defined would
continue the orbit.  This
contradiction proves completeness.

All brackets among the real fields \(X_a,Y_a\) vanish on \(\cR\).
Their local flow maps therefore commute on the dense regular locus and,
by continuity, on all of \(Z\).  Near zero complex time,
composition of the real and imaginary flows is the complex-time
holomorphic flow of $W_a$.  Denote their commuting composition by
$\mathcal A$.  More explicitly, the holomorphic ODE in each
ambient local-flow domain gives joint holomorphicity near
$\{0\}\times Z$; for an
arbitrary parameter $z_0$, the group identity
\[
 \mathcal A(z,x)=\mathcal A\bigl(z_0,\mathcal A(z-z_0,x)\bigr)
\]
translates those local complex-flow domains to a neighborhood of
$\{z_0\}\times Z$.  The complete commuting flows therefore define a
jointly holomorphic action
\begin{equation}\label{sp:eq:action}
                      \mathcal A:\C^r\times Z\longrightarrow Z.
\end{equation}
The action is Killing on $\cR$ and therefore preserves the intrinsic
regular distance.  Density and uniqueness of metric completion make it
globally isometric.  Taking the real and imaginary flow parameters in
\eqref{sp:eq:real-generators} and integrating
\eqref{sp:eq:generator-coordinates} gives exactly
\begin{equation}\label{sp:eq:equivariance}
                       G(\mathcal A(z,x))=G(x)+z.
\end{equation}

\emph{Step 4: analytic and metric product.}
Let $Z'=G^{-1}(0)_{\rm red}$.  Equation \eqref{sp:eq:equivariance} shows
that
\begin{align*}
 \Phi:\C^r\times Z'&\longrightarrow Z,
 &\Phi(z,y)&=\mathcal A(z,y),\\
 \Phi^{-1}:Z&\longrightarrow\C^r\times Z',
 &\Phi^{-1}(x)&=(G(x),\mathcal A(-G(x),x))
\end{align*}
are inverse holomorphic maps.  The second component of the inverse lands
set-theoretically in $G^{-1}(0)$ and factors through its reduction because
$Z$ is reduced.

To see that the reduced zero slice is normal, let a weakly holomorphic
function on $Z'$ be given.  Pull it back to a weakly holomorphic function
on $\C^r\times Z'$, constant in the first variable.  Through the product
biholomorphism, normality of $Z$ extends it holomorphically.  Restriction
to the zero slice extends the original function, so the weak Riemann
extension criterion \cite[Chapter~7, Section~4]{GrauertRemmertSheaves} makes
$Z'$ normal.  A product with a smooth factor is regular exactly over the
regular locus of the other factor.  Hence the
biholomorphism identifies
$\Reg Z=\C^r\times\Reg Z'$ and
$\Sing Z=\C^r\times\Sing Z'$.  Let \(g'\) be the restriction of \(g\)
to \(\ker dG\), transported to \(\Reg Z'\).  On the regular locus the
parallel horizontal fields and $\ker dG$ are orthogonal parallel
distributions, so
\begin{equation}\label{sp:eq:regularmetricproduct}
                          \Phi^*g=g_{\rm E}\oplus g'.
\end{equation}
Complete the regular Riemannian product.  Uniqueness of metric completion
gives $Z\cong\C^r\times\overline{Z'}$, where $\overline{Z'}$ is the
intrinsic completion of $(\Reg Z',g')$.  We identify it with the analytic
zero fibre.  If $y\in Z'$, approximate $(0,y)$ by regular points
$(z_j,y_j)$.  Continuity of $G$ gives $z_j\to0$, and the constant-speed
action gives
\[
 d(\mathcal A(-z_j,\Phi(z_j,y_j)),\Phi(z_j,y_j))\le C|z_j|\to0.
\]
Thus $(0,y)$ is a metric limit of points $(0,y_j)$ in the regular zero
fibre.  Conversely, every metric limit of regular zero-fibre points lies
in $G^{-1}(0)$ by continuity.  Hence the analytic and completed slices
agree, and \eqref{sp:eq:globalmetricproduct} is a global isometry.

\emph{Step 5: residual analytic structure.}
Restricting a bounded local potential to the holomorphic slice
preserves plurisubharmonicity by the disc criterion.  Restriction of the
ambient metric inequality gives the same type of lower bound on $Z'$.
By \cref{sp:lem:klt}, $Z'$ is klt.  Connectedness follows from the product
and connectedness of $Z$.  The singular-locus identity obtained above is
\eqref{sp:eq:product-singular-set}.

\emph{Step 6: shrinker potential.}
Assume now the additional hypotheses on $f$ in the theorem.
On the regular product, the Euclidean and mixed components of
$\Rc+\nabla^2f=g/2$ give
\[
 \nabla^2_{\C^r}f=\tfrac12g_{\rm E},
 \qquad \nabla^2f(U,V)=0
 \quad\bigl(U\in T\C^r,\ V\in T\Reg Z'\bigr).
\]
The first identity integrates to
\[
 f(z,y)=\frac{|z|^2}{4}+\langle b(y),z\rangle_{\R}+c(y).
\]
Here \(b:\Reg Z'\to\R^{2r}\) and \(c:\Reg Z'\to\R\) are smooth.
The identity \(\Reg Z=\C^r\times\Reg Z'\) and connectedness of \(\Reg Z\)
show that \(\Reg Z'\) is connected.  The mixed identity gives \(db=0\),
so \(b\) is constant.  Set
\(\ell(z)=\langle b,z\rangle_{\R}\) and
\(f'=c=f|_{\{0\}\times Z'}\).  The potential formula holds on the dense
regular product and extends to all of $Z$ by continuity.  Restriction to
the closed zero slice makes \(f'\) locally Lipschitz.  If $f$ is proper,
then $f'$ is proper.  Finite weighted mass is
not inferred from restriction: by the regular product and Tonelli's
theorem,
\begin{equation}\label{sp:eq:massfactor}
 \int_{\cR}e^{-f}\,dV_g
 =\left(\int_{\C^r}e^{-|z|^2/4-\ell(z)}\,dV_{g_{\rm E}}\right)
  \left(\int_{\Reg Z'}e^{-f'}\,dV_{g'}\right).
\end{equation}
The Gaussian factor is finite and strictly positive.  Under
\eqref{sp:eq:product-finite-mass}, the left side is finite, so $f'$ has
finite weighted mass.  Under the final completeness hypothesis of the
theorem, define, for \(s<0\),
\[
 \Upsilon_s:=\operatorname{Fl}_{\nabla f}^{a(s)},
 \qquad a(s):=-\log(-s).
\]
Then \(\Upsilon_{-1}=\Id\) and the canonical self-similar metric is
\(g(s)=(-s)\Upsilon_s^*g\).  On the Euclidean factor,
$f_{\rm E}=|z|^2/4+\langle b,z\rangle_{\mathbb R}$, and the Euclidean
component of $\nabla f$ is $z/2+b$ while its residual component is
$\nabla f'$.  Consequently the canonical flow preserves the product
decomposition.  In the soliton-flow parameter \(a\), the
Euclidean orbit solves
\[
 \frac d{da}\bigl(z(a)+2b\bigr)
 =\frac12\bigl(z(a)+2b\bigr),
 \qquad z(a)+2b=e^{a/2}(z(0)+2b).
\]
The self-similar time satisfies \(e^{a/2}=(-s)^{-1/2}\).  Therefore the Euclidean
component of the canonical map $\Upsilon_s$, normalized at $s=-1$,
satisfies
\[
                    \Upsilon_s^{\rm E}(z)+2b
                    =(-s)^{-1/2}(z+2b).
\]
Writing $\Upsilon_s'$ for the residual component, the full regular-locus
metric is therefore
\begin{equation}\label{sp:eq:negative-time-product}
             g(s)=g_{\rm E}\oplus(-s)(\Upsilon_s')^*g'.
\end{equation}
The identity persists on metric completions.  Hence the parallel
Euclidean distribution and its metric are static on every negative-time
slice in this gauge.
\end{proof}

\begin{proof}[Proof of \cref{sp:thm:parallel-product}]
Apply \cref{sp:thm:spectral-rigidity} to each component of $F$.  The
gradients of all real and imaginary parts are parallel.  Select
$r$ components after a complex-linear target change so that their classes
modulo constants are independent.  Their constant Hermitian Gram matrix
is positive definite: a vector in its kernel would give a holomorphic
linear combination with zero differential on the connected regular
locus, hence a constant combination, contrary to the selection.  Apply
\cref{sp:thm:singular-product}.  All asserted residual and potential
conclusions follow from that theorem.
\end{proof}

It remains to construct $m$ independent limiting holomorphic functions
from the limiting morphism.

\subsection{Magnified base coordinates and the splitting theorem}
\label{sp:sec:fibrationmap}
\label{sp:sec:proof-main}
\label{sp:sec:zero-interface}
\label{sp:sec:splitting-consequences}

We now return from the intrinsic shrinker analysis to the standing AMMP
limiting morphism \(\pi:X\to Y\).  At a limiting scale \(\tau_i\), a
base coordinate centered at \(q_\pi(\xi)\) is magnified by
\(\tau_i^{-1/2}\), turning the natural
\(O(\sqrt{\tau_i})\) base scale into an order-one scale.  Three facts are needed:
every bounded regular tangent chart eventually lies in the fixed
coordinate region; the magnified maps have uniform value and differential
bounds there; and the limit extends over \(\Sing Z\), with full rank under
\((\mathrm Q)\).  The compactness argument uses only automatic basing,
\((\mathrm S)\), regular-chart
convergence, and holomorphicity.  It contains no Hessian-dissipation
argument.

\begin{proposition}
\label{sp:prop:mapcompactness}
Let \(q\in Y^{\rm rv}\), choose a relatively compact holomorphic
coordinate neighbourhood \(U\Subset Y^{\rm rv}\), coordinates
\(a=(a^1,\ldots,a^m):U\to\C^m\) with \(a(q)=0\), and
\[
                         q\in U_0\Subset U_1\Subset U.
\]
Let \(\xi=(\mu_t)_{t<T}\) be a fixed limiting point based at \(q\),
and let \(\tau_i\searrow0\) realize a tangent space at \(\xi\), whose
time-\((-1)\) model is
\((Z,d,\cR,g,J,f,\mathfrak m)\).  On \(\pi^{-1}(U)\), set
\begin{equation}\label{sp:eq:Fi}
                             F_i=\tau_i^{-1/2}a\circ\pi.
\end{equation}
Assume $(\mathrm S)_U$.  After passing to a further subsequence still
realizing the same tangent space, the maps $F_i$ converge
smoothly on every compact time-$(-1)$ regular chart to a holomorphic map
\begin{equation}\label{sp:eq:limitF}
                              F:\cR\longrightarrow\C^m.
\end{equation}
The limit \(F\) has bounded differential and therefore extends uniquely
to a globally Lipschitz map \(F:Z\to\C^m\).  This extension is
holomorphic on the normal complex space \(Z\).

If $(\mathrm Q)_U$ also holds, then on $\cR$
\begin{equation}\label{sp:eq:twosidedF}
 cI_m\le
 \left[\langle dF^\alpha,d\overline{F^\beta}\rangle_{g^{-1}}
 \right]_{\alpha,\beta=1}^m
 \le CI_m.
\end{equation}
\end{proposition}

\begin{proof}
For each \(i\), choose an \(H_{2n}\)-center \(z_i\) of
\(\mu_{T-\tau_i}\).  By \cref{lem:regular-chart-localization}, every
compact regular convergence chart lies at uniformly bounded
\(g_i(-1)\)-distance from \(z_i\) and is contained in
\(\pi^{-1}(U_1)\) for all sufficiently large \(i\).

We first check the scaling.  Since
\(g_i(-1)=\tau_i^{-1}g(T-\tau_i)\), its inverse is
\(g_i(-1)^{-1}=\tau_i g(T-\tau_i)^{-1}\), while
\(dF_i=\tau_i^{-1/2}d(a\circ\pi)\).  Thus
\(\tau_i\cdot\tau_i^{-1/2}\cdot\tau_i^{-1/2}=1\), so the parabolic
scaling cancels in the Gram matrix:
\begin{equation}\label{sp:eq:scalingGram}
 \left[\langle dF_i^\alpha,d\overline{F_i^\beta}\rangle_{
 g_i(-1)^{-1}}\right]
 =\left[\langle d(a^\alpha\circ\pi),d\overline{(a^\beta\circ\pi)}
 \rangle_{g(T-\tau_i)^{-1}}\right].
\end{equation}
Condition $(\mathrm S)$ says $Q_t\le c^{-1}\eta^{-1}$, so the right
side is uniformly bounded above.  By \cref{lem:automatic-label}\textup{(ii)},
\(\pi(z_i)\to q\), so \(z_i\in\pi^{-1}(U_0)\) for all large \(i\).
At such an \(H_{2n}\)-center,
\[
 |F_i(z_i)|
 \le C\tau_i^{-1/2}d_\eta(\pi(z_i),q)\le C
\]
again by \cref{lem:automatic-label}\textup{(ii)}.  Every point of a
compact regular chart lies at bounded rescaled distance from \(z_i\).
Put \(\delta_0=d_\eta(\overline U_0,Y\setminus U_1)>0\).  The Schwarz
estimate \cref{lem:ambient-schwarz} gives
\[
 d_{g_i(-1)}\bigl(z_i,X\setminus\pi^{-1}(U_1)\bigr)
 \ge c^{1/2}\delta_0\tau_i^{-1/2}\longrightarrow\infty.
\]
Hence, for
large \(i\), every minimizing geodesic from \(z_i\) to the chart stays in
\(\pi^{-1}(U_1)\).  Integrating the differential bound along these
geodesics gives a uniform \(C^0\) bound on the chart.

Every component of $F_i$ is holomorphic and therefore harmonic on the
K\"ahler slice.  Smooth metric convergence and interior elliptic
Schauder estimates on a slightly larger regular chart give uniform
bounds for all derivatives.  A diagonal subsequence over a countable
regular atlas produces \eqref{sp:eq:limitF}; the local limits agree on
overlaps because the prelimit functions do.  Smooth convergence of the
complex structures gives $dF\circ J=i\,dF$.

The upper bound in \eqref{sp:eq:scalingGram} passes to the limit, so $F$ is
uniformly Lipschitz for the intrinsic regular-locus distance.  Since $Z$
is the metric completion, it extends uniquely and with the same Lipschitz
constant to $Z$.  The extension is locally bounded and holomorphic on
$\cR$.  The analytic and metric topologies agree, so this metric extension
is locally bounded in analytic charts.  The Riemann extension theorem for
normal complex spaces \cite[Chapter~7, Section~4]{GrauertRemmertSheaves}
therefore makes it holomorphic on $Z$.

Finally, $(\mathrm Q)$ says $Q_t\ge C^{-1}\eta^{-1}$.  In fixed
coordinates this is exactly the lower matrix bound in
\eqref{sp:eq:scalingGram}.  Smooth passage on every regular chart gives the
lower half of \eqref{sp:eq:twosidedF}.
\end{proof}

\begin{remark}[The map is not defined by a global cutoff]
The formula \eqref{sp:eq:Fi} is used only on the region in which $a$ is a
holomorphic coordinate system.  \Cref{lem:regular-chart-localization}
proves that every
bounded regular tangent chart eventually lies in that region.  Thus no
nonholomorphic global cutoff enters the limiting map or the spectral
argument.
\end{remark}

The limiting map has now been constructed.  Condition $(\mathrm Q)$
gives its full complex rank, while
\cref{sp:sec:spectral-rigidity,sp:sec:globalization} convert that rank into a
global product.

We now combine the magnified-coordinate construction with the intrinsic
spectral and globalization results to prove the splitting theorem and
record the coordinate interface used below.

\paragraph{Proof of the full-rank criterion.}

\begin{proof}[Proof of \cref{sp:prop:fullrank}]
\emph{Step 1: construct the limiting base map.}
Fix an arbitrary scale sequence and a subsequential tangent space at the
same fixed limiting point.  Passing to a further subsequence does not
change this tangent space.  By \cref{sp:prop:mapcompactness}, the
magnified coordinate maps
converge on every compact regular chart to a globally Lipschitz
holomorphic map
\[
                              F=(F^1,\ldots,F^m):Z\to\C^m.
\]

\emph{Step 2: make the limiting differentials parallel.}
Apply \cref{sp:thm:spectral-rigidity} to each component of \(F\).  After
subtracting its weighted mean, every component is an exact
\(1/2\)-eigenfunction of the Friedrichs drift operator.  Its real and
imaginary Hessians vanish on \(\cR\).  Hence all real and complex
gradients of the components of \(F\) are parallel.

\emph{Step 3: verify full base rank.}
The lower half of \eqref{sp:eq:twosidedF}, which comes from
\((\mathrm Q)\), says that the Hermitian Gram matrix
\[
 H_{\alpha\bar\beta}
 =\langle dF^\alpha,d\overline{F^\beta}\rangle_{g^{-1}}
\]
is positive definite.  Since the differentials are parallel, \(H\) is
constant.  Thus the components are linearly independent modulo constants
and have complex rank exactly \(m\).

\emph{Step 4: globalize the product.}
Apply \cref{sp:thm:singular-product}.  It gives a global biholomorphic
isometry \(Z\cong\C^m\times Z'\).  The residual, potential,
singular-set, and negative-time conclusions recorded below follow from
the same theorem.  Since the original tangent space was arbitrary, the
time-\((-1)\) model of every tangent space at the fixed limiting point
has this splitting.

\end{proof}

\paragraph{Consequences under compact-local product control.}

The spectral proof used only the two quotient estimates near
\(q_\pi(\xi)\).  On a Zariski-open product region,
\((\mathrm S)\) follows from the limiting-class Schwarz estimate and
\((\mathrm Q)\) follows from Fu--Zhang's product-slice estimate.  We
record the geometric splitting and then the centered-coordinate interface
needed in the curvature section.  Coordinate convergence and the exact-zero
identity are consequences of the splitting theorem, not hypotheses.

\begin{corollary}
\label{cor:local-full-rank-package}
Assume the semiample limiting-class setting of
\cref{sec:fano-fibration-setup}, let \(q\in Y^{\rm prod}\), and let
\(\xi=(\mu_t)_{t<T}\) be a fixed limiting point based at \(q\).  The
time-\((-1)\) model of every tangent space at \(\xi\) splits globally,
biholomorphically, and isometrically as
\begin{equation}\label{ti:eq:rigid-product}
                         Z\cong\C^m\times Z',
                         \qquad \dim_\C Z'=n-m.
\end{equation}
The factor \(Z'\) is complete, normal, and klt; its K\"ahler current has
bounded local potentials and a local ambient lower bound.  Write \(f_Z\)
for the normalized tangent shrinker potential.  At time \(-1\),
\begin{equation}\label{ti:eq:rigid-potential}
 f_Z(z,y)=\frac{|z|^2}{4}+\ell(z)+f'(y),
\end{equation}
where \(\ell\) is real affine-linear and \(f'\) is locally Lipschitz.
Moreover,
\begin{equation}\label{ti:eq:rigid-weighted-mass}
                 \int_{\Reg Z'}e^{-f'}\,dV_{g'}<\infty,
\end{equation}
and, writing \((\mathcal X_s,g_\infty(s))\) for the tangent flow, there
is a residual ancient flow \(g'(s)\) such that in the canonical gauge
\begin{equation}\label{ti:eq:rigid-flow}
 (\mathcal X_s,g_\infty(s))
 \cong(\C^m,g_{\rm E})\times(Z',g'(s)),\qquad s<0.
\end{equation}
\end{corollary}

\begin{proof}
By \cref{lem:automatic-label}, \(\xi\) has the parabolic
target-moment estimate at its base point \(q\).  The two quotient bounds
on nested relatively compact neighbourhoods of \(q\) are supplied by
\cref{prop:zariski-product-horizontal}.  The full-rank spectral criterion
\cref{sp:prop:fullrank} therefore gives
\eqref{ti:eq:rigid-product}.  In the proof of that criterion, the
singular global product theorem \cref{sp:thm:singular-product} is applied
to the limiting base map.  Its residual, potential, weighted-mass, and
ancient-flow conclusions give the remaining assertions.
\end{proof}

\begin{proposition}
\label{prop:local-coordinate-interface}
In the setting of the preceding corollary, let
\(w=(w^1,\ldots,w^m)\) be holomorphic coordinates centered at
\(q\), and let \(\tau_i\searrow0\) realize a tangent space at \(\xi\).
After passing to a further subsequence realizing the same tangent space,
there is a
choice of the product isometry in \eqref{ti:eq:rigid-product} and a
matrix \(L\in\mathrm{GL}(m,\C)\), depending on the chosen tangent-space
realization and on \(w\) but independent of \(i\), such that
\begin{equation}\label{ti:eq:rigid-coordinate-limit}
 L\!\left(\tau_i^{-1/2}w\circ\pi\right)
 \longrightarrow \operatorname{pr}_{\C^m}
\end{equation}
smoothly on every compact regular convergence chart.  The normalization
is complex linear and contains no translation.  Write
\(F_q:=\pi^{-1}(q)\).  Wherever the coordinate chart is defined,
\begin{equation}\label{ti:eq:rigid-exact-zero}
 (w\circ\pi)^{-1}(0)=F_q,
 \qquad
 Z'=\operatorname{pr}_{\C^m}^{-1}(0)_{\rm red}.
\end{equation}
The second identity is understood under the product isometry just chosen.
\end{proposition}

\begin{proof}
Center the maps at the fixed base point \(q\), rather than at the
auxiliary \(H_{2n}\)-centers used for pointed compactness.  The
normalization \(w(q)=0\) then makes the original fibre \(F_q\) the exact
zero set at every scale.
Shrink the coordinate domain, if necessary, and choose nested
neighbourhoods \(q\in U_0\Subset U_1\Subset U\) inside a product region.
The estimates \((\mathrm S)_U\) and \((\mathrm Q)_U\) follow from
\cref{prop:zariski-product-horizontal}.
\Cref{sp:prop:mapcompactness} gives, after a further subsequence
representing the same time-\((-1)\) metric model, convergence of the
unshifted maps
\[
 F_i=\tau_i^{-1/2}w\circ\pi\longrightarrow F
\]
on compact regular charts.  Its limiting Hermitian Gram matrix is
positive definite by \((\mathrm Q)_U\).  By
\cref{sp:thm:spectral-rigidity}, the real and imaginary Hessians of the
components of \(F\) vanish, so the Gram matrix is also constant.  Choose
\(L\in\mathrm{GL}(m,\C)\) to be the linear normalization used in
Step~1 of the proof of \cref{sp:thm:singular-product}, and apply that
theorem to \(G=LF\).  In its proof, the complete commuting horizontal
flows satisfy
\[
 G(\mathcal A(z,x))=G(x)+z,
\]
the residual factor is \(G^{-1}(0)_{\rm red}\), and the product
coordinate is \(G\) itself.  Consequently
\[
 LF_i\longrightarrow G=\operatorname{pr}_{\C^m}
\]
without a target translation.  Since \(L\) is invertible and \(w(q)=0\),
\[
 (LF_i)^{-1}(0)=(w\circ\pi)^{-1}(0)=F_q
\]
wherever the coordinate chart is defined.  This proves both assertions.
\end{proof}

\begin{proof}[Proof of \cref{thm:intro-splitting}]
By GAGA, \(E\) is algebraic.  An algebraic vector bundle is Zariski
locally trivial, so its projectivization is biregularly a product over a
Zariski-open neighbourhood of \(q\).  Hence
\(q\in Y^{\rm prod}\), and the product conclusion of
\cref{cor:local-full-rank-package} is precisely
\eqref{eq:intro-minimal-splitting}.
\end{proof}

The normalized-coordinate convergence and exact zero-level identity
above are the precise interface used in the curvature argument of
\cref{sec:surface-prototype}.

\section{Limiting curve models and fixed-fibre estimates in relative dimension one}
\label{sec:surface-prototype}\label{sec:surface-fixed-fibre}
\label{ti:sec:surface-base}\label{sec:general-fixed-fibre}

\subsection{Limiting points and rescaled base maps}
\label{surf:sec:setup-map}
\label{ti:sec:P1-horizontal}

We work in the semiample limiting-morphism setting of
\cref{sec:fano-fibration-setup} in relative complex dimension one:
\[
                         \pi:X^{m+1}\longrightarrow Y^m,
                         \qquad m\ge1.
\]
Thus \(X\) is smooth and projective, \(Y\) is normal and projective, and
\(\pi\) is the limiting morphism.  Let
\(\omega(t)\), \(0\le t<T\), be the maximal finite-time solution of the
unnormalized K\"ahler--Ricci flow.  As before,
\[
 g(t)=2g_{\omega(t)},
 \qquad \partial_tg(t)=-2\Rc(g(t)).
\]
The limiting class satisfies
\[
                  [\omega_0]-2\pi T c_1(X)=\pi^*[\eta]
\]
for a K\"ahler form \(\eta\) on \(Y\).  Every regular fibre
\(F_y=\pi^{-1}(y)\) is \(\PP^1\).  Indeed, on a regular fibre the
relative tangent sequence gives
\[
 0\longrightarrow TF_y\longrightarrow TX|_{F_y}
 \longrightarrow T_yY\otimes\mathcal O_{F_y}\longrightarrow0,
 \qquad c_1(X)|_{F_y}=c_1(TF_y),
\]
and the limiting class identity makes \(F_y\) a Fano curve.

After the constant parabolic normalization
$\widetilde\omega(s)=T^{-1}\omega(Ts)$, we may and do take $T=1$; we then
drop the tilde and absorb $T^{-1}\eta$ into the notation $\eta$.  Thus the
flow is defined for $0\le t<1$ and its limiting class satisfies
\[
 [\omega_0]-2\pi c_1(X)=\pi^*[\eta].
\]
In this section
\[
                         Y^\circ:=Y^{\rm rv},\qquad
                         X^\circ:=\pi^{-1}(Y^{\rm rv}).
\]

The first four subsections assemble the common package for both the
fixed-fibre argument and the moving-fibre uniformization in
\cref{ti:sec:general-base}.  They construct the limiting base maps,
classify the one-dimensional residual tangent, establish smooth pointed
capture and bounded-drift packing, and record the exact-zero and
exact-area tools.  The last two subsections close the fixed-fibre
estimate.  When \(m=1\), the direct orbifold-removal argument gives the
complex-surface case.  The final subsection uses full-rank splitting at
the limiting time for every \(m\ge1\) and proves the cylindrical classification
at a fixed limiting point needed in \cref{ti:sec:general-base}.

Within this section the variable $t<0$ in $g_{i,t}$ denotes rescaled
time; when rescaled and original times occur together, the latter are
denoted by $s_i$.

Fix \(q\in Y^{\rm prod}\).  By definition there is a Zariski-open
product neighbourhood \(U_{\rm prod}\ni q\).  After shrinking it
Zariski-openly, choose a principal product neighbourhood \(U_q\ni q\)
and a biregular product isomorphism over \(U_q\),
\[
 \pi^{-1}(U_q)\simeq U_q\times\PP^1.
\]
Set \(\omega_Y:=\eta\) and let
\(g_Y=2g_{\omega_Y}\) be its standard real metric.  Let
\[
 P_q:U_q\times\PP^1\longrightarrow\PP^1
\]
be the projection to the fibre factor; the submanifolds
\(P_q^{-1}(z)=U_q\times\{z\}\) are the product-horizontal base slices.

All horizontal and volume inputs have already been established in
\cref{sec:prelim}.  More precisely,
\cref{prop:zariski-product-horizontal} gives, on every
\(U_1\Subset U_q\),
\begin{equation}\label{eq:fixed-fibre-SQ}
 c_{U_1}\eta\le S_t\le C_{U_1}\eta,
 \qquad
 C_{U_1}^{-1}\eta^{-1}\le Q_t\le c_{U_1}^{-1}\eta^{-1},
 \qquad 0\le t<1.
\end{equation}
Its product-slice clause also gives, uniformly in \(z\in\PP^1\),
\begin{equation}\label{eq:fixed-fibre-product-slice}
 \omega(t)|_{P_q^{-1}(z)}\le C_{U_1}\omega_Y
 \qquad\text{on }U_1.
\end{equation}
Finally, the detailed Schur-complement computation is
\cref{lem:general-block-volume}, and the resulting uniform Type-I
tube-volume estimate is \cref{cor:type-I-tube-volume}.  Thus the surface and arbitrary-base
arguments use one common preliminary statement.

From now on, we fix the point $q\in Y^{\rm prod}$, and use the notation
above.

Fix \(x_0\in F_q:=\pi^{-1}(q)\) and choose a subsequential anchored
limiting point \(\xi=(\mu_t)_{0\le t<1}\), obtained from point conjugate
heat kernels with spatial pole \(x_0\) and pole times tending to \(1\).
This limiting point is fixed throughout the construction below that
uses \(x_0\); no uniqueness is asserted.

Consider any sequence \(t_i\nearrow1\).  Define its Type-I rescalings by
\[
 M_i:=X,\qquad
 g_{i,t}:=(1-t_i)^{-1}g(1+(1-t_i)t),\qquad
 \nu_t^i:=\mu_{1+(1-t_i)t},
 \qquad t\in[-T_i,0),
\]
where \(T_i=(1-t_i)^{-1}\to\infty\).  By
\cref{prop:limiting-package}, after possibly passing to a subsequence,
we obtain the $\mathbb{F}$-convergence
within a metric-flow correspondence $\mathfrak C$,
\begin{equation}
(M_i,(g_{i,t})_{t\in (-T_i,0)},(\nu_t^i)_{t\in (-T_i,0)})\xrightarrow[i\,\to\,\infty]{\mathbb{F},\,\mathfrak{C}} (\mathcal{X},(\nu_{x_{\infty};t})_{t\,\in\,(-\infty,0)})
\end{equation}
where the metric flow pair $({\mathcal{X}},({\nu}_{x_{\infty};t})_{t\in (-\infty,0)})$ is modeled on the normal klt shrinking gradient
K\"ahler--Ricci soliton
$(\widehat X,\widehat g,J_{\widehat X},f_{\widehat X})$ supplied by
\cref{prop:limiting-package}.  When \(m=1\), hence \(n=2\),
\cref{prop:limiting-package}\textup{(H6)} identifies this model as a
K\"ahler orbifold with isolated singularities.  We abbreviate
\[
 (\hat X,\hat g,J_{\hat X},f_{\hat X})
 :=(\widehat X,\widehat g,J_{\widehat X},f_{\widehat X}).
\]
We denote the time-$t$ slice by $\mathcal X_t$, its regular locus by
$\mathcal R_t:=\Reg(\mathcal X_t)$, and its smooth regular-slice metric
by $\hat g_t:=\widehat g(t)$, where $\widehat g(t)$ is the self-similar
family in \eqref{surf:eq:induced-shrinker-spacetime}.
First, we prove the following proposition.
\begin{proposition}\label{surf:conofPhi}
Fix \(t<0\).  There are adapted centres \(p_i\), converging in a fixed
regular chart, such that
\[
 \int_Xd_{g_{i,t}}(p_i,y)^2\,d\nu_t^i(y)\le C_t(-t).
\]
Put \(b_i=\pi(p_i)\), and let \(w=(w^1,\ldots,w^m)\) be holomorphic
coordinates near \(q\).  After passing to a subsequence, the maps
\[
 W_i\circ\pi,\qquad
 W_i(y)=(1-t_i)^{-1/2}\bigl(w(y)-w(b_i)\bigr),
\]
converge smoothly on compact subsets of \(\mathcal R_t\) to a
holomorphic map
\[
 \Pi_t:\mathcal R_t\longrightarrow\C^m.
\]
Thus \(\Pi_t\) is the pointed limit of
\[
 \pi:(X,g_{i,t},p_i)\longrightarrow
 \bigl(Y,(1-t_i)^{-1}g_Y,b_i\bigr).
\]
\end{proposition}
\begin{proof}
Fix a relatively compact regular coordinate ball
$E\Subset\mathcal R_t$ with smooth boundary and
$\nu_{x_\infty;t}(E)>0$.  Let $E_i$ be its image under a regular
convergence embedding.  Smooth convergence of the distinguished densities in
\cref{prop:limiting-package} gives
$\nu_t^i(E_i)\ge\delta>0$ for all large $i$.  Put
\[
 M_i(p):=\int_Xd_{g_{i,t}}(p,y)^2\,d\nu_t^i(y).
\]
Fubini and $H_{2m+2}$-concentration give
\[
 \int_{E_i}M_i(p)\,d\nu_t^i(p)
 \le \Var_{g_{i,t}}(\nu_t^i)\le H_{2m+2}(-t).
\]
Hence one may choose $p_i\in E_i$ with
$M_i(p_i)\le H_{2m+2}\delta^{-1}(-t)$.  After a further subsequence, the
preimages of $p_i$ converge in $\overline E$ to a point
$p_\infty\in\mathcal R_t$.  Thus $p_i$ converges in the correspondence by
construction.  Set $C_t:=H_{2m+2}\delta^{-1}$.

The conjugate flow \(\nu^i\) is obtained by rescaling the fixed anchored
limiting point \(\xi=(\mu_t)_{t<1}\), whose automatic base point is \(q\) by
\cref{prop:automatic-anchored-point}.  In the unrescaled flow, the
preceding adapted-center bound is at time
\[
s_i=1+(1-t_i)t
\]
and, since $g_{i,t}=(1-t_i)^{-1}g(s_i)$, has the form
\[
 \int_Xd_{g(s_i)}(p_i,y)^2\,d\mu_{s_i}(y)
 \le C_t(1-s_i).
\]
Applying \cref{surf:lem:base-localization}(iii) to this limiting point,
with its measure evaluated at \(t=s_i\), \(T=1\), and \(A=C_t\), first
gives the following inclusion with \(B_\eta\).  On the fixed regular
coordinate neighbourhood, \(d_\eta\) and the intrinsic \(g_Y\)-distance
are uniformly equivalent.  Absorbing that equivalence into \(C_0\), we
obtain
\[
p_i\in
\pi^{-1}\left(
   B_{g_Y}
\left(
q,\,
C_0(1-s_i)^{1/2}
\right)
\right).
\]
Therefore
\[
d_{g_Y}(b_i,q)
\leq
C_0(1-t_i)^{1/2}(-t)^{1/2}.
\]
Multiplying by the rescaling factor \((1-t_i)^{-1/2}\), we get
\[
d_{(1-t_i)^{-1} g_Y}(b_i,q)
\leq
C_0(-t)^{1/2}.
\]
Thus \(b_i\to q\) in \(g_Y\), while
\(d_{(1-t_i)^{-1}g_Y}(b_i,q)\) remains uniformly bounded.

Since \(q\in Y^{\rm prod}\), for all sufficiently large \(i\), we have
$b_i\in U_q$. Choose a holomorphic coordinate system
$w=(w^1,\ldots,w^m)$ centered at $q$ and
normalize $dw_q$ so that $g_Y(q)$ is the Euclidean Hermitian metric. On the shrinking
coordinate balls set
\[
 W_i(y)=(1-t_i)^{-1/2}\bigl(w(y)-w(b_i)\bigr).
\]
Choose a relatively compact holomorphic coordinate ball
$V\Subset U_q$ with $q\in V$. Then for every fixed radius
\(R<\infty\), we have
\[
B_{(1-t_i)^{-1} g_Y}(b_i,R)=B_{ g_Y}(b_i,(1-t_i)^{1/2} R)\Subset V
\]
for all sufficiently large \(i\). Thus the pointed base manifolds
$(Y,(1-t_i)^{-1} g_Y,b_i)$ converge smoothly on compact subsets to
$(T_qY,g_Y(q),0)\simeq \mathbb C^m$.

By the parabolic Schwarz lemma, in the form
\cref{lem:ambient-schwarz}, there exists a uniform
constant \(C<\infty\) such that
\begin{equation}\label{surf:nablaPhiupb}
|d\pi|^2_{g_{i,t},\,(1-t_i)^{-1} g_Y}
=
\operatorname{tr}_{(1-t_i)^{-1} g(s_i)}\pi^*((1-t_i)^{-1} g_Y)
=
\operatorname{tr}_{g(s_i)}\pi^*g_Y
\leq C.
\end{equation}
Here and below, \(|d\pi|\) is the real Hilbert--Schmidt norm and
\(\operatorname{tr}_g\) is the ordinary real trace.  For a holomorphic
map this real trace is twice the corresponding complex Hermitian trace;
that fixed factor is absorbed into \(C\) when the parabolic Schwarz
estimate is written in its complex convention.

Therefore the map
\[
\pi:(X,g_{i,t})\longrightarrow (Y,(1-t_i)^{-1} g_Y)
\]
is uniformly Lipschitz, independently of \(i\). In particular, if a sequence of points \(x_i\in X\) stays at bounded \(g_{i,t}\)-distance from \(p_i\), then the sequence \(\pi(x_i)\) stays at bounded \((1-t_i)^{-1} g_Y\)-distance from \(b_i\).

By construction, $p_i$ converges in a regular convergence chart to
$p_\infty\in\mathcal R_t$.  Given $K\Subset\mathcal R_t$, use
connectedness of the regular locus to choose a connected compact regular
set $K'\Subset\mathcal R_t$ whose interior contains
$K\cup\{p_\infty\}$.
\Cref{prop:limiting-package} supplies, for all sufficiently large
$i$, embeddings $\psi_i$ from an open neighborhood of $K'$ into $X$,
compatible with the correspondence, such that
$\psi_i^*g_{i,t}\to\hat g_t$ smoothly on $K'$ and the pulled-back complex
structures converge smoothly there.  Set $\pi_i:=\pi$ and consider the maps
\[
\pi_i\circ\psi_i:K'\longrightarrow (Y,(1-t_i)^{-1} g_Y,b_i).
\]
Composing with the target coordinate \(W_i\), we get holomorphic maps
\[
F_i
:=
W_i\circ \pi\circ\psi_i
:
K'\longrightarrow\mathbb C^m.
\]

Compatibility of the charts with the correspondence gives
\[
 \sup_{x\in K'}d_{g_{i,t}}\bigl(\psi_i(x),p_i\bigr)\le C_{K'}.
\]
The maps are uniformly Lipschitz, and the target normalization was chosen
so that $W_i(b_i)=0$. Hence the maps $F_i$ are uniformly bounded on $K'$.

Each component of \(F_i\) is harmonic for
\(\psi_i^*g_{i,t}\).  The local \(C^0\)-bound and interior elliptic
estimates for the smoothly converging metrics give uniform derivative
bounds on \(K\).  After passing to a subsequence, \(F_i\) therefore
converges smoothly on \(K\).  Passing the Cauchy--Riemann equations for
the smoothly converging pulled-back complex structures to the limit
shows that the limiting map is holomorphic; denote it by
\[
F_\infty:K\longrightarrow\mathbb C^m.
\]
A diagonal argument over a compact exhaustion of \(\mathcal R_t\)
produces a holomorphic map
\[
\Pi_t:\mathcal R_t\longrightarrow\mathbb C^m
\]
such that
\[
W_i\circ\pi\circ\psi_i
\longrightarrow
\Pi_t
\]
smoothly on every compact subset of \(\mathcal R_t\). Thus \(\Pi_t\) is
the pointed smooth limit of \(\pi\) on the regular part of the time slice
\(\mathcal X_t\). This completes the proof.
\end{proof}

We now recenter the spectral map at the base point of a fixed limiting
point.  The preceding map is centred at the adapted points
\(b_i=\pi(p_i)\);
that centring is natural for pointed convergence but does not preserve the
exact fibre \(F_q\).  For splitting and exact-level arguments we instead
use coordinates \(w=(w^1,\ldots,w^m)\) centred at \(q\) and put
\begin{equation}\label{ti:eq:P1-base-maps}
 F_i=\tau_i^{-1/2}w\circ\pi,\qquad
 \tau_i\searrow0.
\end{equation}
Let \(\xi=(\mu_t)_{t<1}\) be any fixed limiting point based at \(q\),
chosen before the scale sequence.  Its base point is automatic, and
\cref{ti:lem:label} supplies
\begin{equation}\label{ti:eq:P1-base-moment}
 \int_Xd_\eta(\pi(x),q)^2\,d\mu_t(x)\le C_\xi(1-t).
\end{equation}
No basing or quantitative moment condition is being added as a
hypothesis.

Let \(g_i(s)=\tau_i^{-1}g(1+\tau_i s)\).  After passing to a
subsequence realizing a tangent space, write
\((Z,d,g_Z,J_Z)\) for its time-\((-1)\) K\"ahler metric model.
Then
\cref{prop:local-coordinate-interface} gives one
\(L\in\mathrm{GL}(m,\C)\) and a continuous holomorphic map \(G:Z\to\C^m\) such
that
\begin{equation}\label{ti:eq:P1-map-convergence}
 G_i:=LF_i\longrightarrow G
\end{equation}
smoothly on compact regular convergence charts in the time-\((-1)\)
model.  At \(s=-1\), with \(t=1-\tau_i\),
\[
 g_i(s)^{-1}=\tau_i g(t)^{-1},
 \qquad
 dF_i^\alpha=\tau_i^{-1/2}d(w^\alpha\circ\pi).
\]
Hence
\begin{align}
 \left\langle dF_i^\alpha,d\overline{F_i^\beta}\right\rangle_{g_i(s)^{-1}}
 &=
 \left\langle d(w^\alpha\circ\pi),
 d\overline{(w^\beta\circ\pi)}\right\rangle_{g(t)^{-1}}.
 \label{eq:fixed-fibre-Gram-scaling}
\end{align}
Thus \eqref{eq:fixed-fibre-SQ} passes to the limit with no power of
\(\tau_i\).  On \(\Reg Z\),
\begin{equation}\label{eq:common-parallel-map}
 \nabla dG=0,\qquad
 c|a|^2\le |d(a\cdot G)|_{g_Z}^2\le C|a|^2
 \quad(a\in\C^m).
\end{equation}
The upper bound and the intrinsic-length property of \(\Reg Z\) make
\(G\) globally Lipschitz after completion.  The normalization \(L\) is
complex linear and contains no translation.  Consequently
\begin{equation}\label{ti:eq:P1-exact-level}
 G_i^{-1}(0)=F_i^{-1}(0)=F_q
\end{equation}
wherever the coordinate chart is defined.  This distinction between the
\(b_i\)-centred maps and the \(q\)-centred maps will be used in the final
exact-level argument.

\subsection{Residual curve models and smooth tangent spaces at limiting points}
\label{ti:sec:P1-classification}

The full-rank splitting theorem at the limiting time confines every possible
singularity to a normal complex one-dimensional factor.  The following
classification is the common model-theoretic input for the packing
argument below and for the secondary tangent in
\cref{ti:sec:general-base}.  At this stage
both possibilities remain; the exact-level argument in
\cref{ti:sec:arbitrary-base-fixed-fibre} will later exclude the Gaussian
one for a fixed limiting point.

\begin{proposition}
\label{ti:prop:P1-models}
Let \(Z\) be a connected complete normal klt complex space of dimension
\(m+1\), whose regular locus carries a K\"ahler metric \(g_Z\) with
complex structure \(J_Z\).  Let \(f_Z\) be locally Lipschitz on \(Z\),
smooth on \(\Reg Z\), and satisfy
\[
 \Rc(g_Z)+\nabla^2f_Z=\frac12g_Z
 \qquad\text{on }\Reg Z.
\]
Assume the finite-time
tangent-space package used in \cref{thm:singular-product}: its analytic and metric
regular loci agree, \(Z\) is the metric completion of its regular locus,
the restricted distance is intrinsic, the K\"ahler current has bounded
local potentials and the required local ambient lower bound, and the
soliton measure
\[
 d\mathfrak m=(4\pi)^{-(m+1)}e^{-f_Z}\,dV_{g_Z}
 \quad\text{on }\Reg Z,
 \qquad \mathfrak m(Z\setminus\Reg Z)=0,
 \qquad \mathfrak m(Z)=1,
\]
is normalized to be a probability measure.  Suppose there is a globally
Lipschitz holomorphic map
\[
 H:Z\longrightarrow\C^m
\]
such that on \(\Reg Z\)
\begin{equation}\label{eq:abstract-parallel-base-map}
 \nabla dH=0,\qquad
 c|a|^2\le |d(a\cdot H)|_{g_Z}^2\le C|a|^2
 \quad(a\in\C^m).
\end{equation}
Then there is a fixed \(L\in\mathrm{GL}(m,\C)\), with no target
translation, such that
\begin{equation}\label{eq:abstract-residual-product}
 Z\cong\C^m\times\Sigma,\qquad LH=\operatorname{pr}_{\C^m},
\end{equation}
where \(\Sigma\) is a complete smooth connected Riemann surface.  In particular,
\begin{equation}\label{ti:eq:P1-two-models}
 Z\cong\C^{m+1}
 \qquad\text{or}\qquad
 Z\cong\C^m\times\PP^1,
\end{equation}
where the first model is Gaussian and the sphere is round with
\(\Rc=\frac12g\).  In these zero-preserving coordinates,
\((LH)^{-1}(0)\) is respectively \(\C\) or \(\PP^1\).  Define the
normalized Nash entropy by
\[
 \mathcal N(Z):=\int_{\Reg Z}f_Z\,d\mathfrak m-(m+1).
\]
If \(\mathcal N(Z)<0\), only the spherical model occurs.
\end{proposition}

\begin{proof}
Apply \cref{thm:singular-product} to
\eqref{eq:abstract-parallel-base-map}.  It gives
\eqref{eq:abstract-residual-product}, where \(\Sigma\) is a complete
normal complex curve.  A one-dimensional normal complex analytic space
is nonsingular: every local ring is a one-dimensional Noetherian
integrally closed local domain, hence a discrete valuation ring and
therefore regular.  Thus \(\Sigma\) is a smooth Riemann surface.  The
analytic product is smooth; equality of the analytic and metric regular
loci then shows that no metric singularity remains.

The vertical part of the product shrinker equation is
\begin{equation}\label{ti:eq:P1-residual-shrinker}
 \Rc_{g_\Sigma}+\nabla^2f_\Sigma=\frac12g_\Sigma.
\end{equation}
The Euclidean--Euclidean component gives
\(\nabla_z^2f_Z=\frac12g_{\rm E}\), while the mixed component makes the
Euclidean linear term independent of \(\Sigma\).  Hence
\begin{equation}\label{eq:residual-potential-splitting}
 f_Z(z,y)=\frac{|z|^2}{4}+\ell(z)+f_\Sigma(y)
\end{equation}
for a real affine-linear function \(\ell\).  Finite mass of the
normalized soliton measure, completion of the square in \(z\), and
Fubini imply
\begin{equation}\label{ti:eq:P1-weighted-area}
 \int_\Sigma e^{-f_\Sigma}\,dA_{g_\Sigma}<\infty.
\end{equation}

Bernstein--Mettler's classification gives bounded scalar curvature for
the complete two-dimensional gradient soliton
\cite[Corollary~1]{BernsteinMettler}.  Hence the scalar
curvature lies in \(L^2(e^{-f_\Sigma}dA)\) by
\eqref{ti:eq:P1-weighted-area}.  Petersen--Wylie's shrinking-surface
classification under this weighted \(L^2\) hypothesis implies that the
residual shrinker is flat or compact and round
\cite[Appendix~A, especially Corollary~5 and the paragraph following
it]{PetersenWylie}.

In the flat case, lift to the Euclidean universal cover.  Equation
\eqref{ti:eq:P1-residual-shrinker} gives
\[
 \widetilde f_\Sigma(x)=\frac14|x-a|^2+c.
\]
Every deck transformation preserves this function and fixes \(a\), so
a nonidentity deck transformation cannot act freely.  Thus
\(\Sigma\cong\C\).  In the nonflat case the universal cover is the round
\(S^2\).  The complex orientation excludes \(\mathbb{RP}^2\), and no
nontrivial orientation-preserving finite group acts freely on \(S^2\).
Therefore \(\Sigma\cong\PP^1\), with
\(\Rc(g_\Sigma)=\frac12g_\Sigma\).

Finally, the normalized Gaussian shrinker has Nash entropy zero.  Indeed,
after completing the square and normalizing the measure,
\[
 f(x)=|x-a|^2/4,\qquad
 R=0,\qquad
 \int|\nabla f|^2\,d\nu=\int f\,d\nu=m+1,
\]
so \(\mathcal N(Z)=0\).  Strictly negative Nash entropy therefore
excludes the Gaussian model.
\end{proof}

\begin{corollary}
\label{cor:fixed-point-two-models}
Let \(\xi\) be a fixed limiting point based at
\(q=q_\pi(\xi)\in Y^{\rm prod}\).  Let \(Z\) denote the underlying
time-\((-1)\) K\"ahler metric model of any tangent space at \(\xi\).
Then \(Z\) is globally smooth and is one of the following two models:
\begin{equation}\label{ti:eq:P1-split}
 Z\cong\C^{m+1}
 \qquad\text{or}\qquad
 Z\cong\C^m\times\PP^1.
\end{equation}
In the complex-linear normalization of
\cref{prop:local-coordinate-interface}, the limiting map is
\(G=\operatorname{pr}_{\C^m}\), its zero slice is respectively \(\C\) or \(\PP^1\),
and the ancient tangent flow has globally bounded curvature on every
compact negative-time interval.
\end{corollary}

\begin{proof}
The base point and Type-I moment are automatic by \cref{ti:lem:label}.
The full-rank splitting
\cref{cor:local-full-rank-package,prop:local-coordinate-interface}
gives \(Z\cong\C^m\times\Sigma\), together with the centred projection
map and finite residual weighted mass.  Apply
\cref{ti:prop:P1-models}.  Both resulting products are smooth and their
self-similar ancient flows have globally bounded curvature on each
\([-A,-A^{-1}]\Subset(-\infty,0)\).
\end{proof}

\begin{corollary}
\label{surf:global-bounded-curvature}
\label{surf:split-bounded-curvature}
Let $(\hat X,\hat g_s)$ be the underlying metric flow of a tangent space
at a fixed limiting point whose base lies in $Y^{\rm prod}$, in the
relative-dimension-one setting.  For every $A>1$,
\begin{equation}\label{surf:eq:split-slab-curvature}
 K_A:=\sup_{s\in[-A,-A^{-1}]}\sup_{\hat X}
 |\Rm_{\hat g_s}|_{\hat g_s}<\infty.
\end{equation}
The number $K_A$ may depend on the realized tangent-space model, but is
fixed before
the packing number $N$ is chosen.
\end{corollary}

\begin{proof}
By \cref{cor:fixed-point-two-models}, the time-\((-1)\) model is
\(\C^{m+1}\) or \(\C^m\times\PP^1\), both of which have the same conclusion.
Self-similarity
and curvature scaling give
\[
 \sup_{\hat X}|\Rm_{\hat g_s}|_{\hat g_s}
 =(-s)^{-1}\sup_{\hat X}|\Rm_{\hat g_{-1}}|_{\hat g_{-1}},
\]
which is uniformly finite on $[-A,-A^{-1}]$.
\end{proof}

\subsection{Smooth pointed capture and bounded-drift packing}

\begin{lemma}
\label{surf:locsmoothcon}
Let
\[
 (X,g_i(s))\longrightarrow(\widehat X,\hat g_s),\qquad s<0,
\]
be smooth spacetime convergence on an exhaustion of a complete smooth
limit, with time-preserving convergence embeddings \(\psi_i\) defined
on domains \(U_i\Subset\widehat X\times(-\infty,0)\).  Write
\(U_{i,t}=U_i\cap(\widehat X\times\{t\})\) and \(\psi_{i,t}\) for the
time-\(t\) maps.  Fix \(t_0<0\), \(D>0\), an integer \(k\ge0\), and
\(\varepsilon>0\).  If \(p_i\to p_\infty\) in the correspondence at
time \(t_0\), then there is a relatively compact smooth domain
\(\Omega\Subset\widehat X\) such that, for all large \(i\),
\[
\overline{B_{g_i(t_0)}(p_i,D)}
\Subset
\psi_{i,t_0}(\Omega)
\Subset
\psi_{i,t_0}(U_{i,t_0}),
\]
\[
\left\|
\psi_{i,t_0}^{*}g_i(t_0)-\hat g_{t_0}
\right\|_{C^k(\overline\Omega,\hat g_{t_0})}
\leq\varepsilon.
\]
\end{lemma}
\begin{proof}
By hypothesis, \(U_i\) exhaust \(\widehat X\times(-\infty,0)\), and
\(\psi_i^*g_{i,t}\to\hat g_t\) smoothly on compact spacetime sets.

Completeness and Hopf--Rinow allow us to choose a connected smooth
domain \(\Omega\Subset\widehat X\) containing
\(\overline{B_{\hat g_{t_0}}(p_\infty,D+4)}\).  Choose \(\delta>0\)
with \([t_0-\delta,t_0+\delta]\Subset(-\infty,0)\).  Then
\[
 \mathcal C=\overline\Omega\times[t_0-\delta,t_0+\delta]\Subset U_i
\]
for all large \(i\).  Smooth convergence on \(\mathcal C\) gives the
required \(C^k\)-estimate and, after increasing \(i\),
\[
 (1-\vartheta)\hat g_{t_0}
 \le\psi_{i,t_0}^{*}g_i(t_0)
 \le(1+\vartheta)\hat g_{t_0},
 \qquad
 \sqrt{1-\vartheta}(D+4)>D+3.
\]
Set \(p_i^\sharp=\psi_{i,t_0}(p_\infty)\).  Correspondence convergence
gives \(d_{g_i(t_0)}(p_i,p_i^\sharp)<1\).

Every curve from \(p_i^\sharp\) to
\(X\setminus\psi_{i,t_0}(\Omega)\) has a first-exit segment whose
pullback joins \(p_\infty\) to \(\partial\Omega\).  Its length is at
least
\[
 \sqrt{1-\vartheta}\,
 d_{\hat g_{t_0}}(p_\infty,\partial\Omega)
 \ge\sqrt{1-\vartheta}(D+4)>D+3.
\]
Consequently,
\[
 d_{g_i(t_0)}\bigl(p_i,X\setminus\psi_{i,t_0}(\Omega)\bigr)>D+2,
\]
and hence
\[
 \overline{B_{g_i(t_0)}(p_i,D)}
 \Subset\psi_{i,t_0}(\Omega)
 \Subset\psi_{i,t_0}(U_{i,t_0}).
\]
\end{proof}

We next combine smooth exhaustion with a curvature bound already
available for the realized tangent-space model.  In the arbitrary-base-dimension
route this is \cref{surf:global-bounded-curvature}; in the direct surface
route it will be supplied by \cref{surf:2dKRS}.  The following lemma
isolates the buffered parabolic capture used in both routes.

\begin{lemma}
\label{lem:buffered-parabolic-capture}
Let
\[
 (X,g_i(s))\longrightarrow (Z,g_\infty(s)),\qquad s<0,
\]
converge smoothly in spacetime on an exhaustion of a complete smooth
limit, and suppose that \(p_i\to p_\infty\) in the correspondence at
time \(-1\).  Assume
\[
 K_0:=1+\sup_{Z\times[-1-1/8,-1]}|\Rm(g_\infty(s))|<\infty,
\]
choose
\[
 0<\rho\le \min\left\{\frac14,\frac1{\sqrt8},
 (4K_0)^{-1/2}\right\},
 \qquad I_\rho=[-1-\rho^2,-1],
\]
and fix an integer \(N\).  Suppose that
\[
 d_{g_i(-1)}(y_{i,k},p_i)\le R_N,
 \qquad 1\le k\le N,
\]
where \(R_N\) is independent of \(i\).  After an \(N\)-dependent
subsequence, one convergence image contains the worldline tracks over
\(I_\rho\) of all the balls \(B_{g_i(-1)}(y_{i,k},\rho)\).  On their
union,
\begin{equation}\label{eq:buffered-parabolic-curvature}
 |\Rm(g_i(s))|_{g_i(s)}\le 2K_0\le\rho^{-2}.
\end{equation}
The convergence domain and index threshold may depend on \(N\), but
\(\rho\) does not.
\end{lemma}

\begin{proof}
Apply \cref{surf:locsmoothcon} at time \(-1\), with radius
\(R_N+2\), to place
\(\overline{B_{g_i(-1)}(p_i,R_N+2)}\) in a fixed precompact convergence
image.  Pulling back the finitely many \(y_{i,k}\) and passing to a
further subsequence gives
\[
 y_{i,k}\longrightarrow y_{\infty,k}\in Z,
 \qquad 1\le k\le N.
\]
Since \(\rho\le\frac14\), smooth convergence at time \(-1\) also places
the pullback of \(B_{g_i(-1)}(y_{i,k},\rho)\) in the compact set
\[
 E_{\infty,k}:=
 \overline{B_{g_\infty(-1)}(y_{\infty,k},2\rho)}.
\]

Let \(\Phi_{s,-1}\) be the limiting worldline map, and set
\begin{equation}\label{eq:buffered-limit-parabolic-compactum}
 \mathcal K_N:=
 \bigcup_{k=1}^N
 \left\{(\Phi_{s,-1}(x),s):
 x\in E_{\infty,k},\ s\in I_\rho\right\}.
\end{equation}
Completeness and Hopf--Rinow make every \(E_{\infty,k}\) compact, so
\(\mathcal K_N\) is compact.  Choose a precompact smooth spacetime
domain \(\mathcal U_N\) containing it with
\begin{equation}\label{eq:buffered-spacetime-distance}
 \delta_N:=
 \inf_{s\in I_\rho}
 d_{g_\infty(s)}
 \bigl(\mathcal K_{N,s},\partial\mathcal U_{N,s}\bigr)>0.
\end{equation}

On \(\overline{\mathcal U_N}\), the pulled-back metrics and spacetime
vectors converge uniformly.  Continuous dependence for the worldline
ODE is therefore uniform for initial points in
\(\bigcup_kE_{\infty,k}\).  A first exit of an approximating track from
the convergence image would lie on \(\partial\mathcal U_{N,s}\) while
remaining \(o(1)\)-close to \(\mathcal K_{N,s}\), contradicting
\eqref{eq:buffered-spacetime-distance}.  Thus all the stated tracks are
captured throughout \(I_\rho\).  Smooth curvature convergence gives
\[
 |\Rm(g_i(s))|\le2K_0\le\rho^{-2}.
\]
Only the subsequence, domain, and index threshold were chosen after
\(N\); \(\rho\) was fixed beforehand.
\end{proof}

We next use horizontal transport to pass from actual
\(H_{2m+2}\)-centres to fibres whose labels move
a bounded distance in the rescaled base.  This is the common pointed
package used later both with $q_i=q$ and in the bounded-drift branch of
\cref{ti:sec:general-base}.

\begin{lemma}
\label{lem:general-horizontal-transport}
Let \(\xi=(\mu_t)_{t<1}\) be a fixed limiting point based at
\(q=q_\pi(\xi)\in Y^{\rm prod}\).  Let \(\tau_i\searrow0\), set
\(t_i=1-\tau_i\), and let \(p_i\) be an \(H_{2m+2}\)-centre of
\(\mu_{t_i}\) at backward scale \(\tau_i\) with respect to \(g(t_i)\).
Let
\(q_i\in Y^{\rm prod}\).  Fix a finite \(A\ge0\) and suppose
\[
 d_\eta(q_i,q)\le A\sqrt{\tau_i}.
\]
Then for all large \(i\) there is \(\bar p_i\in F_{q_i}\) satisfying
\begin{equation}\label{eq:general-centre-to-fibre}
 d_{\tau_i^{-1}g(t_i)}(p_i,\bar p_i)\le C_A.
\end{equation}
\end{lemma}

\begin{proof}
The automatic moment estimate for \(\mu_t\) and the defining centre
estimate are
\[
 \int_Xd_\eta(\pi(x),q)^2\,d\mu_{t_i}(x)\le C_\xi\tau_i,
 \qquad
 \int_Xd_{g(t_i)}(p_i,x)^2\,d\mu_{t_i}(x)
 \le H_{2m+2}\tau_i.
\]
The Schwarz estimate makes \(\pi:(X,g(t_i))\to(Y,g_Y)\) uniformly
Lipschitz.  For every \(x\),
\[
 d_\eta(\pi(p_i),q)^2
 \le C d_{g(t_i)}(p_i,x)^2+2d_\eta(\pi(x),q)^2.
\]
Integrating gives
\begin{equation}\label{eq:general-centre-base-distance}
 d_\eta(\pi(p_i),q)\le C_0\sqrt{\tau_i},
\end{equation}
and hence
\[
 d_\eta(\pi(p_i),q_i)\le(C_0+A)\sqrt{\tau_i}.
\]
Choose a strongly geodesically convex neighbourhood
\(q\in V\Subset U_q\).  On \(V\), the restricted ambient distance
\(d_\eta\) is uniformly equivalent to the intrinsic distance
\(d_{g_Y}\).  Hence,
after changing the constant,
\[
 d_{g_Y}(\pi(p_i),q_i)\le C_A\sqrt{\tau_i}.
\]
For large \(i\), the minimizing \(g_Y\)-geodesic \(\gamma_i\) from
\(\pi(p_i)\) to \(q_i\) lies in \(V\).  Lift it by the
\(g(t_i)\)-orthogonal complement of the vertical tangent bundle.
Properness of the submersion gives the lift on the whole interval, and
its endpoint \(\bar p_i\) lies in \(F_{q_i}\).  By the definition of the
quotient metric and \eqref{eq:fixed-fibre-SQ},
\[
 |\dot{\widetilde\gamma}_i|_{g(t_i)}^2
 =2S_{t_i}\!\left(
   \dot\gamma_i^{1,0},\overline{\dot\gamma_i^{1,0}}\right)
 \le C|\dot\gamma_i|_{g_Y}^2.
\]
Thus
\[
 L_{g(t_i)}(\widetilde\gamma_i)\le C_A\sqrt{\tau_i}.
\]
Rescaling lengths by \(\tau_i^{-1/2}\) proves
\eqref{eq:general-centre-to-fibre}.
\end{proof}

\begin{proposition}
\label{prop:general-bounded-drift-packing}
\label{prop:general-packing-diameter}
\label{prop:general-packing-package}
Fix a limiting point \(\xi=(\mu_t)_{t<1}\) based at
\(q\in Y^{\rm prod}\), let
\(\tau_i\searrow0\), and set
\[
 g_i(s)=\tau_i^{-1}g(1+\tau_i s).
\]
Choose \(H_{2m+2}\)-centres \(p_i\) of the time-\((-1)\) measures
\(\mu_{i,-1}:=\mu_{1-\tau_i}\) with respect to \(g_i(-1)\).
After passing to a subsequence realizing a tangent space at \(\xi\), fix
a finite \(A\ge0\) and any \(q_i\) satisfying
\begin{equation}\label{eq:general-bounded-drift-fibres}
 d_\eta(q_i,q)\le A\sqrt{\tau_i}.
\end{equation}
There is \(D_A<\infty\) such that
\begin{align}
 \operatorname{diam}_{g_i(-1)}F_{q_i}&\le D_A,
 \label{eq:general-packing-diameter}\\
 F_{q_i}&\subset B_{g_i(-1)}(p_i,D_A).
 \label{eq:general-packing-ball}
\end{align}
For every finite \(B\ge0\), there is \(C_{A,B}<\infty\) such that
\begin{equation}\label{eq:general-packing-enlarged-curvature}
 \sup_{B_{g_i(-1)}(p_i,D_A+B)}
 |\Rm(g_i(-1))|_{g_i(-1)}\le C_{A,B}.
\end{equation}
The constants may depend on the fixed limiting point and the realizing
subsequence, but not on \(i\).
\end{proposition}

\begin{proof}
\emph{Step 1: align the actual centres with the smooth tangent-space
model.}
Let \(Z\) be the time-\((-1)\) model and let \(\Psi_i\) be smooth
convergence embeddings on an exhaustion.  Denote the limiting tangent
flow and its distinguished measures by
\((\mathcal Z,g_\infty(s),\mu_{\infty,s})\), with
\(\mathcal Z_{-1}=Z\).  By
\cref{cor:fixed-point-two-models}, \(Z\) is globally smooth.  Choose a
connected compact smooth domain \(E\Subset Z\) with
\(\mu_{\infty,-1}(E)>0\).  Smooth convergence of the distinguished
densities gives \(\delta>0\) such that
\[
 \mu_{i,-1}(\Psi_i(E))\ge\delta.
\]
Since \(p_i\) is an actual centre,
\[
 \int_Xd_{g_i(-1)}(p_i,x)^2\,d\mu_{i,-1}(x)\le H_{2m+2}.
\]
Therefore some \(y_i\in\Psi_i(E)\) satisfies
\[
 d_{g_i(-1)}(p_i,y_i)\le
 R_0:=\sqrt{H_{2m+2}/\delta}.
\]
Choose a connected precompact smooth
\(\Omega\Subset Z\) with
\[
 d_{g_Z}(E,\partial\Omega)>R_0+4.
\]
On \(\overline\Omega\), the pulled-back metrics are uniformly
bilipschitz to \(g_Z\).  If \(p_i\notin\Psi_i(\Omega)\), a minimizing
curve from \(y_i\) to \(p_i\) has a first exit through
\(\Psi_i(\partial\Omega)\).  Pullback and the lower bilipschitz estimate
make that initial segment longer than \(R_0+1\), a contradiction.  The
same first-exit argument shows that the whole minimizing geodesic stays
in \(\Psi_i(\Omega)\).  Hence
\(\widetilde p_i:=\Psi_i^{-1}(p_i)\) lies in a fixed compact subset.
After a subsequence,
\begin{equation}\label{eq:packing-centre-alignment}
 \widetilde p_i\longrightarrow p_\infty\in Z,\qquad
 d_{g_i(-1)}(p_i,\Psi_i(p_\infty))\longrightarrow0.
\end{equation}

\emph{Step 2: move the centre to the bounded-drift fibre.}
\Cref{lem:general-horizontal-transport} gives
\(\bar p_i\in F_{q_i}\) with
\begin{equation}\label{eq:packing-horizontal-distance}
 d_{g_i(-1)}(p_i,\bar p_i)\le C_{h,A}.
\end{equation}

\emph{Step 3: choose the curvature radius before packing.}
By \cref{cor:fixed-point-two-models},
\[
 K_*:=\sup_{Z\times[-1-1/8,-1]}|\Rm(g_\infty(s))|<\infty.
\]
Put \(\widehat K_*=1+K_*\) and choose
\[
 0<\rho\le
 \min\left\{\frac14,\frac1{\sqrt8},
 (4\widehat K_*)^{-1/2}\right\}.
\]
This choice precedes and is independent of the packing number \(N\).
The buffered spacetime argument needed after the packing points are
chosen is supplied by \cref{lem:buffered-parabolic-capture}.

\emph{Step 4: pack the moving fibre.}
Suppose the diameters in \eqref{eq:general-packing-diameter} are
unbounded.  After passing to a subsequence, choose
$\widetilde y_i\in F_{q_i}$ with
\[
 d_{g_i(-1)}(\widetilde y_i,\bar p_i)\longrightarrow\infty.
\]
The fibre is connected, so the continuous function
$y\mapsto d_{g_i(-1)}(y,\bar p_i)$ has image containing the interval
from $0$ to this last distance.  Hence, for each fixed $N$ and all
sufficiently large $i$, there are
\(y_{i,1},\ldots,y_{i,N}\in F_{q_i}\) with
\[
 d_{g_i(-1)}(y_{i,k},\bar p_i)=k,\qquad
 d_{g_i(-1)}(y_{i,k},y_{i,\ell})\ge|k-\ell|.
\]
The radius-\(\rho\) balls are disjoint and, by
\eqref{eq:packing-horizontal-distance},
\[
 d_{g_i(-1)}(y_{i,k},p_i)\le N+C_{h,A},
 \qquad 1\le k\le N.
\]
For this fixed \(N\), apply
\cref{lem:buffered-parabolic-capture} with
\(R_N=N+C_{h,A}+1\).  Hence the worldline track of every required
fixed final-time radius-\(\rho\) ball satisfies
\eqref{eq:buffered-parabolic-curvature}.  The index threshold and the
further subsequence may depend on \(N\), but \(\rho\) does not.

We now verify the noncollapsing input in the unrescaled variables of the
\(T=1\) flow.  Put
\[
                       t_i=1-\tau_i,
                       \qquad r_i=\sqrt{\tau_i}\,\rho.
\]
The interval \([-1-\rho^2,-1]\) corresponds to
\[
 [t_i-r_i^2,t_i]
 =[1-\tau_i(1+\rho^2),1-\tau_i]\subset[0,1)
\]
for all large \(i\).  On the captured parabolic neighbourhood,
curvature scaling gives
\[
 |\Rm(g(1+\tau_i s))|_{g(1+\tau_i s)}
 =\tau_i^{-1}|\Rm(g_i(s))|_{g_i(s)}
 \le(\tau_i\rho^2)^{-1}=r_i^{-2}.
\]
Moreover,
\[
                      \frac{r_i^2}{t_i}
                      =\frac{\tau_i\rho^2}{1-\tau_i}
                      \longrightarrow0.
\]
Thus Perelman's no-local-collapsing theorem and its uniform
\(\kappa\)-noncollapsing formulation
\cite[Sections~4.1--4.2]{PerelmanEntropy} give
\begin{equation}\label{eq:general-packing-noncollapse}
 \Vol_{g_i(-1)}B_{g_i(-1)}(y_{i,k},\rho)
 \ge\kappa_{\rm nc}\rho^{2m+2}=:c_0>0.
\end{equation}
The constant \(c_0\) is independent of \(N,k,i\).

The Schwarz estimate makes
\[
 \pi:(X,g_i(-1))\longrightarrow(Y,\tau_i^{-1}g_Y)
\]
uniformly Lipschitz.  Since each \(y_{i,k}\) lies in
\(F_{q_i}\) and \eqref{eq:general-bounded-drift-fibres} holds, all the
balls lie in
\[
 \Omega_i(R_{A,\rho})
 =\pi^{-1}\!\left(B_\eta(q,R_{A,\rho}\sqrt{\tau_i})\right)
\]
for one \(R_{A,\rho}\) independent of \(N,k,i\).  By
\cref{cor:type-I-tube-volume},
\[
 \Vol_{g_i(-1)}\Omega_i(R_{A,\rho})\le C_{A,\rho}.
\]
Therefore
\[
 C_{A,\rho}\ge
 \sum_{k=1}^N\Vol_{g_i(-1)}B(y_{i,k},\rho)
 \ge Nc_0.
\]
Since \(N\) is arbitrary, this is impossible.  Hence the fibre diameters
are bounded.  Combining this bound with
\eqref{eq:packing-horizontal-distance} gives
\eqref{eq:general-packing-ball}.

Finally fix a finite \(B\ge0\).  Completeness and Hopf--Rinow give a
precompact tangent domain containing the closed
\((D_A+B+6)\)-ball about \(p_\infty\).  The first-exit argument and
\eqref{eq:packing-centre-alignment} place
\(B_{g_i(-1)}(p_i,D_A+B)\) inside its convergence image.  Smooth
curvature convergence gives
\eqref{eq:general-packing-enlarged-curvature}.
\end{proof}

\begin{remark}[One bounded anchor captures the whole fibre]
\label{rem:bounded-horizontal-anchor}
In the setting of \cref{prop:general-bounded-drift-packing}, it is
enough to find \(\bar p_i\in F_{q_i}\) such that
\[
 d_{g_i(-1)}(p_i,\bar p_i)\le A_0.
\]
Indeed, the Schwarz estimate and the centre localization of
\cref{lem:automatic-label} give
\[
 d_\eta(q_i,q)\le C(A_0+1)\sqrt{\tau_i}.
\]
Hence \cref{prop:general-bounded-drift-packing} yields uniform
whole-fibre diameter, pointed containment, and curvature bounds on
every fixed enlargement.  The horizontal lifts in
\cref{lem:general-horizontal-transport} provide precisely such anchors.
\end{remark}

\subsection{The exact-zero and fibre-area package}
\label{ti:sec:P1-area-capture}

\begin{remark}[Area normalization]
For every regular fibre \(F_y\), restriction of the evolving class,
the relative tangent sequence, and Whitney additivity give
\begin{equation}\label{ti:eq:P1-fibre-class}
 [\omega(t)]|_{F_y}=2\pi(1-t)c_1(\PP^1).
\end{equation}
Since \(\int_{\PP^1}c_1(\PP^1)=2\),
\begin{equation}\label{ti:eq:P1-area-kahler}
 \int_{F_y}\omega(t)=4\pi(1-t).
\end{equation}
Thus, in the real Ricci-flow convention \(g=2g_\omega\), for every
\(\tau\in(0,1]\),
\begin{equation}\label{ti:eq:P1-area-eight-pi}
 \Area_{\tau^{-1}g(1-\tau)}(F_y)
 =\tau^{-1}\,2\int_{F_y}\omega(1-\tau)=8\pi.
\end{equation}
This identity is a normalization check for fixed limiting points and is the
area input used in the secondary-scale contradiction in the proof of
\cref{ti:lem:moving-pole-fibre-capture}.
\end{remark}

Pointed smooth convergence alone does not say that the limiting zero set
is the limit of an \emph{exact} zero set.  We first record the local
bridge used below and in the secondary-scale argument of
\cref{ti:sec:general-base}.

The following is the standard parameter-dependent implicit-function
theorem in a tubular neighbourhood; for the compact isotopy statement,
see \cite[Theorem~20.2]{AbrahamRobbin}.
\begin{lemma}
\label{ti:lem:P1-zero-graphs}
Let \((U,h)\) be a smooth Riemannian manifold, and let
\(H_i,H:U\to\R^k\) be smooth maps with
\(H_i\to H\) in \(C^\infty_{\rm loc}\).  Suppose that \(0\) is a regular
value of \(H\), and set \(S=H^{-1}(0)\).
\begin{enumerate}[label=(\roman*),leftmargin=2.4em]
\item If \(K\Subset\operatorname{int}K^+\Subset S\) are compact
domains, then, for all large \(i\), \(H_i^{-1}(0)\) contains a small
normal graph \(K_i\) over \(K\), and \(K_i\to K\) smoothly.
\item If \(S\) is compact and is the entire zero set of \(H\) in a
precompact tubular neighbourhood \(U_0\Subset U\), then, for all large
\(i\), \(H_i^{-1}(0)\cap U_0\) contains a compact boundaryless normal
graph \(S_i\) over \(S\).  Moreover, \(S_i\to S\) smoothly and
\(S_i\cong S\).
\end{enumerate}
If Riemannian metrics \(h_i\) on \(U\) converge smoothly to \(h\), the
induced metrics and volume forms on these graphs converge smoothly.
\end{lemma}

\subsection{The direct complex-surface closure}
\label{surf:sec:removal-curvature}

The removal argument and fixed-fibre packing proof through
\cref{surf:TypeIonbundle} assume \(m=1\), hence \(n=2\).  By
\cref{prop:limiting-package}\textup{(H6)}, the tangent model is a
K\"ahler orbifold shrinker with isolated singularities.  The
following direct argument is independent of the singular product theorem
used in the final arbitrary-base closure.  We write
\(\pi_t:=\Pi_t:\mathcal R_t\to\C\).

We first extend $\pi_t$ to the whole time-slice $\mathcal{X}_t$, which is
modeled on an orbifold K\"ahler shrinker with isolated singularities.
\begin{proposition}\label{surf:extofPhi}
For every \(t<0\), the map \(\pi_t:\mathcal R_t\to\C\) extends
holomorphically to \(\pi_t:\mathcal X_t\to\C\).
\end{proposition}
\begin{proof}
Fix an isolated orbifold point
\(a\in\mathcal X_t\setminus\mathcal R_t\).  A neighbourhood of \(a\)
has the form \(B/\Gamma\), where \(B\subset\C^2\) is a small ball and
\(\Gamma\subset U(2)\) is finite.  The regular part is
\((B\setminus\{0\})/\Gamma\), so \(\pi_t\) lifts to a holomorphic
function on
\[
B\setminus\{0\}.
\]
Since \(\dim_{\C}B=2\), Hartogs' extension theorem extends the lifted
function holomorphically across \(0\).  Its \(\Gamma\)-invariance
extends by the identity theorem, so it descends to \(B/\Gamma\).
Repeating this at every isolated singular point gives a holomorphic map
\[
\pi_t:\mathcal X_t\longrightarrow\C.
\]
\end{proof}

Next, we use the compact-local horizontal estimate
\cref{prop:zariski-product-horizontal} to rule out the isolated orbifold
singularities of $\mathcal X_t$ (which is also $\hat{X}$ under our
notation).
\begin{proposition}\label{surf:ruleoutsing}
For every \(t<0\), the space \(\mathcal X_t\) is smooth, and the
extended map \(\pi_t:\mathcal X_t\to\C\) is holomorphic.
\end{proposition}
\begin{proof}
We continue to use the notation from the proof of
\cref{surf:conofPhi}.  Thus we fix $t\in(-\infty,0)$, use the adapted
center $p_i$ constructed there, set $s_i=1+(1-t_i)t$ and
$b_i=\pi(p_i)$, and have $g_{i,t}=(1-t_i)^{-1}g(s_i)$.  Passing
\eqref{surf:nablaPhiupb} to the smooth limit on compact subsets of the
regular part \(\mathcal R_t\subset \mathcal X_t\), we obtain the uniform
estimate:
\[
|d\pi_t|_{\hat g_t,g_{\rm E}}^2\leq C
\]
on \(\mathcal R_t\). We next prove a uniform lower bound for the same
squared norm.
Fix a relatively compact coordinate disk \(q\in V\Subset U_q\), a
regular point \(x_\infty\in\mathcal R_t\), and a compact neighbourhood
\(K\Subset\mathcal R_t\) of \(x_\infty\).  Since \(\pi_t(K)\) is
compact in \(\C\), the corresponding points in the rescaled target
\((Y,(1-t_i)^{-1}g_Y,b_i)\) lie in a fixed rescaled compact ball for
all large \(i\).  In \((Y,g_Y)\), they therefore lie in a shrinking
neighbourhood of \(q\).  Hence the image under \(\pi\) of the
corresponding embedded copy of \(K\) lies in \(V\) for all large \(i\).
Apply \cref{prop:zariski-product-horizontal} on \(V\).  Equivalently,
use its product-slice estimate \eqref{eq:fixed-fibre-product-slice} on
each slice \(V\times\{z\}\).  It lifts every base tangent vector with
length at most \(C_V\) times its \(g_Y\)-length.  Thus
\(S_{s_i}\le C_Vg_Y\), and hence
\(Q_{s_i}\ge C_V^{-1}g_Y^{-1}\).  Since the
target is a curve, this is precisely the lower differential bound
\begin{equation}\label{surf:nablaPhilowb}
 |d\pi|^2_{g_{i,t},\,(1-t_i)^{-1}g_Y}
 =\operatorname{tr}_{g(s_i)}\pi^*g_Y
 \ge C_V^{-1}.
\end{equation}
The constant depends on \(q\) and the fixed trivialization, not on the
compact set \(K\).
Passing this estimate to the limit gives
\begin{equation}\label{surf:nablaPhilowb2}
 |d\pi_t|_{\hat g_t,g_{\rm E}}^2\ge C_V^{-1}
 \qquad\text{on }K.
\end{equation}
Since \(x_\infty\in\mathcal R_t\) was arbitrary, there exists
\(C=C(q)>0\) such that
\begin{equation}\label{surf:nablaPhib}
C^{-1}\leq |d\pi_t|_{\hat g_t,g_{\rm E}}^2\leq C.
\end{equation}

Now suppose that \(a\in\mathcal X_t\) is an isolated orbifold
singularity, modeled on \(B/\Gamma\), where \(B\subset\C^2\) is a
small ball and \(\Gamma\subset U(2)\) acts freely on
\(B\setminus\{0\}\).  The map \(\pi_t\) lifts to a
\(\Gamma\)-invariant holomorphic function on \(B\setminus\{0\}\), which
extends to a holomorphic function
\(\widetilde\pi_t:B\to\C\) by \cref{surf:extofPhi}.  The lifted
orbifold metric extends smoothly and positively across \(0\).
Consequently \eqref{surf:nablaPhib} and continuity give
\[
|d\widetilde\pi_t(0)|\geq c>0.
\]

Since \(\widetilde\pi_t\) is \(\Gamma\)-invariant,
\(\widetilde\pi_t(\gamma z)=\widetilde\pi_t(z)\) for every
\(\gamma\in\Gamma\).  Differentiating at \(0\) gives
\[d\widetilde\pi_t(0)\circ\gamma= d\widetilde\pi_t(0).\]
Thus \(d\widetilde\pi_t(0)\) is a nonzero \(\Gamma\)-invariant complex
linear functional on \(\C^2\).  Choose unitary coordinates so that
it is a nonzero scalar multiple of $dz_1$, and rescale the target
coordinate to obtain \(d\widetilde\pi_t(0)=dz_1\).  Then
\(dz_1\circ\gamma=dz_1\) for every \(\gamma\in\Gamma\).  Since
\(\gamma\in U(2)\), this forces \(\gamma\) to preserve the
\(z_1\)-coordinate.  Equivalently,
\(\gamma\) has the form
\(\gamma(z_1,z_2)=(z_1,\lambda_\gamma z_2)\) for some
\(\lambda_\gamma\in U(1)\).  If \(\lambda_\gamma\neq1\), then
\(\gamma\) fixes every point on the complex line \(\{z_2=0\}\),
contrary to the free action on \(B\setminus\{0\}\).  Thus every
\(\gamma\) is the identity, contradicting that \(a\) is singular.

Therefore \(\mathcal X_t\) has no isolated orbifold singularities.
\end{proof}

\begin{lemma}
\label{surf:2dKRS}
Every complete smooth complex two-dimensional shrinking gradient
K\"ahler--Ricci soliton has bounded curvature.
This follows from \cite[Theorem~1.2]{LiWangSurface}; see also
\cite[Theorem~4.4]{ConlonHallgrenMa}.
\end{lemma}

We now use fixed-fibre packing to prove the joint diameter and curvature
estimate.
\begin{proposition}\label{surf:TypeIonbundle}
Assume $m=1$, so that $X$ is a complex surface.
There exists a constant $C=C(\omega_0,q)<\infty$ such that
\begin{equation}\label{surf:eq:diameter-curvature-conclusion}
 \operatorname{diam}_{g(t)}F_q\le C\sqrt{1-t},
 \qquad
 \sup_{x\in F_q}(1-t)|\Rm_{g(t)}|_{g(t)}(x)\le C
 \quad(0\le t<1),
\end{equation}
where the diameter is the extrinsic diameter in $(X,g(t))$.
\end{proposition}
\begin{proof}
Fix once and for all the anchored limiting point
\(\xi=(\mu_t)_{t<1}\) introduced above.  If either estimate fails,
smoothness on compact time intervals gives \(t_i\nearrow1\), with
\(\tau_i:=1-t_i\), for which one of the following holds:
\[
 \operatorname{diam}_{\tau_i^{-1}g(t_i)}F_q\longrightarrow\infty,
 \tag{D}
\]
or there are \(x_i\in F_q\) such that
\[
 |\Rm_{\tau_i^{-1}g(t_i)}|(x_i)
 =\tau_i|\Rm_{g(t_i)}|(x_i)\longrightarrow\infty.
 \tag{C}
\]
Set
\[
 g_i(s)=\tau_i^{-1}g(1+\tau_i s),
 \qquad
 \nu_s^i=\mu_{1+\tau_i s},
 \qquad -\tau_i^{-1}\le s<0.
\]

\medskip
\noindent\textbf{Step 1: the tangent model and adapted centres.}
After passing to a subsequence, \cref{prop:limiting-package} gives a
tangent space modeled on a complete normal klt shrinking
K\"ahler--Ricci soliton
\[
 (\widehat X,\hat g_s,J_{\widehat X},f_{\widehat X}).
\]
Since \(m=1\), \cref{surf:ruleoutsing} makes \(\widehat X\) a smooth
complex surface.  By \cref{surf:2dKRS} and self-similarity,
\[
 K_0:=
 1+\sup_{\widehat X\times[-1-1/8,-1]}
 |\Rm_{\hat g_s}|_{\hat g_s}<\infty.
\]
This number may depend on the realized tangent space, which is harmless.

As in the proof of \cref{surf:conofPhi}, choose a fixed compact regular
coordinate ball of positive limiting measure.  Smooth density
convergence and the variance bound then give points \(p_i\), converging
in that chart to \(p_\infty\in\widehat X\), such that
\begin{equation}\label{surf:eq:adapted-centre-short}
 \int_Xd_{g_i(-1)}(p_i,y)^2\,d\nu^i_{-1}(y)\le C_c.
\end{equation}
Thus \cref{surf:locsmoothcon} captures every fixed
\(g_i(-1)\)-ball about \(p_i\) in a smooth convergence chart.  Before
choosing any packing number, fix
\begin{equation}\label{surf:eq:packing-radius-short}
 0<\rho\le
 \min\left\{\frac14,\frac1{\sqrt8},(4K_0)^{-1/2}\right\}.
\end{equation}

\medskip
\noindent\textbf{Step 2: a horizontal anchor in \(F_q\).}
The anchored limiting point is based at \(q\).  Applying
\cref{surf:lem:base-localization}\textup{(iii)} to
\eqref{surf:eq:adapted-centre-short} gives
\[
 d_{g_Y}(\pi(p_i),q)\le C\sqrt{\tau_i}.
\]
For large \(i\), write \(p_i=(q_i,z_i)\) in
\(\pi^{-1}(U_q)\simeq U_q\times\PP^1\), and put
\[
 a_i:=(q,z_i)\in F_q.
\]
The \(g_Y\)-geodesic from \(q_i\) to \(q\) lies in a fixed relatively
compact coordinate disk.  Lifting it in the horizontal product slice
and using \eqref{eq:fixed-fibre-product-slice} yields
\begin{equation}\label{surf:eq:horizontal-anchor-short}
 d_{g_i(-1)}(p_i,a_i)
 \le C\tau_i^{-1/2}d_{g_Y}(q_i,q)
 \le C_h.
\end{equation}

\medskip
\noindent\textbf{Step 3: fixed-fibre packing.}
We claim that, along the chosen sequence,
\begin{equation}\label{surf:eq:diameter-short}
 \operatorname{diam}_{g_i(-1)}F_q\le D
\end{equation}
for some \(D<\infty\).  Otherwise, after passing to a subsequence,
choose \(\widetilde x_i\in F_q\) with
\(d_{g_i(-1)}(\widetilde x_i,a_i)\to\infty\).  Since \(F_q\) is
connected, for every fixed integer \(N\) and all large \(i\) there are
points \(y_{i,1},\ldots,y_{i,N}\in F_q\) satisfying
\[
 d_{g_i(-1)}(y_{i,k},a_i)=k,
 \qquad
 d_{g_i(-1)}(y_{i,k},y_{i,\ell})\ge |k-\ell|.
\]
Consequently, the balls \(B_{g_i(-1)}(y_{i,k},\rho)\) are pairwise
disjoint and, by \eqref{surf:eq:horizontal-anchor-short},
\[
 d_{g_i(-1)}(y_{i,k},p_i)\le N+C_h,
 \qquad 1\le k\le N.
\]

For this fixed \(N\), apply
\cref{lem:buffered-parabolic-capture} with \(R_N=N+C_h+1\).  After a
further subsequence, the worldline tracks of these balls over
\([-1-\rho^2,-1]\) lie in one smooth convergence domain and satisfy
\[
 |\Rm_{g_i(s)}|_{g_i(s)}\le\rho^{-2}.
\]
Returning to the unrescaled flow, put \(r_i=\sqrt{\tau_i}\rho\).  The
preceding interval becomes \([t_i-r_i^2,t_i]\), the curvature bound
becomes \(r_i^{-2}\), and
\[
 \frac{r_i^2}{t_i}
 =\frac{\tau_i\rho^2}{1-\tau_i}\longrightarrow0.
\]
Perelman's no-local-collapsing theorem
\cite[Sections~4.1--4.2]{PerelmanEntropy} therefore gives, after
rescaling volumes by \(\tau_i^{-2}\), a constant
\(\kappa_{\rm nc}>0\), independent of \(N,k,i\), such that
\begin{equation}\label{surf:eq:noncollapse-short}
 \Vol_{g_i(-1)}B_{g_i(-1)}(y_{i,k},\rho)
 \ge\kappa_{\rm nc}\rho^4=:c_0>0.
\end{equation}

The parabolic Schwarz estimate gives
\[
 \operatorname{tr}_{g_i(-1)}
 \pi^*(\tau_i^{-1}g_Y)\le C.
\]
Since \(\pi(y_{i,k})=q\), all these balls lie in a fixed Type-I tube
\[
 \Omega_i=
 \pi^{-1}\!\left(B_{g_Y}(q,C_\rho\sqrt{\tau_i})\right),
\]
where \(C_\rho\) depends only on the Schwarz constant and the fixed
radius \(\rho\).
For all large \(i\), the shrinking base ball is contained in a fixed
relatively compact subset of \(U_q\).
By \cref{cor:type-I-tube-volume},
\[
 \Vol_{g_i(-1)}(\Omega_i)\le C_\Omega.
\]
The disjointness and \eqref{surf:eq:noncollapse-short} imply
\[
 C_\Omega\ge
 \sum_{k=1}^{N}
 \Vol_{g_i(-1)}B_{g_i(-1)}(y_{i,k},\rho)
 \ge Nc_0.
\]
Choosing \(N>C_\Omega/c_0\) gives a contradiction and proves
\eqref{surf:eq:diameter-short}.  In particular, \textup{(D)} is
impossible.

\medskip
\noindent\textbf{Step 4: capture of a curvature blow-up point.}
It remains to exclude \textup{(C)}.  Since \(x_i,a_i\in F_q\),
\eqref{surf:eq:diameter-short} and
\eqref{surf:eq:horizontal-anchor-short} give
\[
 d_{g_i(-1)}(x_i,p_i)
 \le d_{g_i(-1)}(x_i,a_i)+d_{g_i(-1)}(a_i,p_i)
 \le D+C_h.
\]
Apply \cref{surf:locsmoothcon} with radius \(D+C_h+1\).  The resulting
precompact smooth convergence chart contains \(x_i\), and smooth
curvature convergence there gives
\[
 \sup_i|\Rm_{g_i(-1)}|_{g_i(-1)}(x_i)<\infty,
\]
contradicting \textup{(C)}.

Thus neither normalized diameter nor normalized curvature can be
unbounded.  This proves
\eqref{surf:eq:diameter-curvature-conclusion}.
\end{proof}
\begin{remark}[Comparison of the two singularity-removal routes]
For \(m=1\), \cref{surf:ruleoutsing} proves smoothness directly and
independently of splitting at the limiting time.  Alternatively,
\cref{cor:local-full-rank-package} reduces the time-\((-1)\) model to
\(\C^m\times\Sigma\), where the complete normal curve \(\Sigma\) is
smooth and is either the Gaussian plane or the round sphere.  Both
products have bounded curvature on compact negative-time slabs.  For
\(m>1\), the product decomposition itself removes the possible
singularities and is used again in
\cref{surf:sec:exact-level-topology}.
\end{remark}

\subsection{The fixed-fibre closure in arbitrary base dimension}
\label{ti:sec:arbitrary-base-fixed-fibre}
\label{surf:sec:exact-level-topology}

\begin{proposition}
\label{surf:prop:exact-level-topology}
\label{ti:prop:P1-capture}
Let \(\xi\) be a fixed limiting point based at
\(q=q_\pi(\xi)\in Y^{\rm prod}\).  After forgetting the distinguished
conjugate heat measures, the underlying K\"ahler metric flow of every
tangent space at \(\xi\) is the shrinking cylinder
\[
 \left(\C^m\times\PP^1,
 g_{\rm E}+(-s)g_{S^2}\right)_{s<0},
 \qquad \Rc(g_{S^2})=\tfrac12g_{S^2}.
\]
More precisely, let \(\tau_i\searrow0\) realize such a tangent space and
let \(p_i\) be \(H_{2m+2}\)-centres of
\(\mu_{i,-1}=\mu_{1-\tau_i}\) with respect to
\(g_i(-1)=\tau_i^{-1}g(1-\tau_i)\) at backward scale one.
After passing to a further subsequence that still realizes the same
tangent space, there is \(D<\infty\) such that, for all sufficiently
large \(i\),
\[
 F_q\subset
 B_{\tau_i^{-1}g(1-\tau_i)}(p_i,D),
\]
and, for every finite \(A\ge0\), there is \(C_A<\infty\) such that
\[
 \sup_{B_{\tau_i^{-1}g(1-\tau_i)}(p_i,D+A)}
 |\Rm_{\tau_i^{-1}g(1-\tau_i)}|_{\tau_i^{-1}g(1-\tau_i)}
 \le C_A.
\]
\end{proposition}

\begin{proof}
Fix an arbitrary tangent space at \(\xi\), together with a scale
sequence \(\tau_i\searrow0\) realizing it.  Any further subsequence
below still realizes the same tangent space.  Choose
\(H_{2m+2}\)-centres \(p_i\) of \(\mu_{i,-1}=\mu_{1-\tau_i}\) with
respect to \(g_i(-1)=\tau_i^{-1}g(1-\tau_i)\) at backward scale one.
Write \(Z\) for the time-\((-1)\) model.  A further subsequence used in
\cref{prop:general-packing-diameter} still realizes the same tangent
space.  By
\cref{cor:fixed-point-two-models}, \(Z\) is smooth.  When \(m=1\) and
\(\xi\) is the anchored limiting point fixed in
\cref{surf:sec:removal-curvature}, \cref{surf:ruleoutsing} gives an
independent, direct proof of the same conclusion.  Since
$F_q\cong\PP^1$, the residual shrinker classification together with
\cref{cor:local-full-rank-package,prop:local-coordinate-interface} gives
\begin{equation}\label{surf:eq:q-centred-splitting}
 Z\cong\C^m\times\Sigma,
 \qquad G=\operatorname{pr}_{\C^m},
 \qquad \Sigma\cong\C\ \hbox{or}\ \PP^1,
\end{equation}
where $\Sigma$ is complete.  Choose a coordinate neighbourhood
\(U_w\Subset U_q\) and injective holomorphic coordinates
\(w=(w^1,\ldots,w^m):U_w\to\C^m\) centred at \(q\).  For a fixed
\(L\in\mathrm{GL}(m,\C)\), depending on the chosen tangent-space
realization, define the \(q\)-centred maps
\begin{equation}\label{surf:eq:q-centred-maps}
 \mathcal G_i:=L\bigl(\tau_i^{-1/2}w\circ\pi\bigr):
 \pi^{-1}(U_w)\longrightarrow\C^m,
 \qquad \mathcal G_i\longrightarrow G
\end{equation}
converge smoothly on every fixed compact convergence chart, which lies
in \(\pi^{-1}(U_w)\) for all large \(i\).  Relative to this domain,
\begin{equation}\label{surf:eq:q-centred-exact-zero}
 \mathcal G_i^{-1}(0)=F_q,
 \qquad G^{-1}(0)=\{0\}\times\Sigma.
\end{equation}
These are not the $b_i$-centred maps of
\cref{surf:sec:setup-map}:
those maps are adapted to the pointed convergence and have zero set
$F_{b_i}$, whereas \eqref{surf:eq:q-centred-maps} preserves the exact
fixed fibre.

By \cref{prop:general-packing-diameter}, applied with \(q_i=q\), at the
corresponding original times $t_i=1-\tau_i$ there is a constant $D$
independent of $i$ such that
\begin{equation}\label{surf:eq:limiting-point-diameter}
 \operatorname{diam}_{\tau_i^{-1}g(t_i)}F_q\le D.
\end{equation}
Let $z,z_0\in S:=G^{-1}(0)$.  Choose compact connected domains
$K\Subset K^+\Subset S$ whose interiors contain both points.
Completeness of $Z$ and Hopf--Rinow make
$\overline{B_{g_Z}(z_0,D+5)}$ compact.  Choose a connected precompact
smooth convergence domain $\Omega\Subset Z$ whose interior contains this
ball and $K^+$ and whose boundary satisfies
\[
 d_{g_Z}(z_0,\partial\Omega)>D+4.
\]
Choose a smooth convergence embedding $\Psi_i$ on a neighbourhood of
$\overline\Omega$.  Put
$H_i=\mathcal G_i\circ\Psi_i$ and $H=G$.  Part~\textup{(i)} of
\cref{ti:lem:P1-zero-graphs}, with \(\C^m\) viewed as \(\R^{2m}\),
applies because \(G=\operatorname{pr}_{\C^m}\) makes \(0\) a regular
value.  It gives points $u_i,u_{0,i}$ on the exact-zero
graphs with
\[
 u_i\to z,
 \qquad u_{0,i}\to z_0,
 \qquad
 x_i:=\Psi_i(u_i),\ x_{0,i}:=\Psi_i(u_{0,i})\in F_q.
\]
Thus \eqref{surf:eq:limiting-point-diameter} gives
\begin{equation}\label{surf:eq:approximating-zero-distance}
 d_{\tau_i^{-1}g(t_i)}(x_i,x_{0,i})\le D.
\end{equation}

We justify the distance passage without presupposing compactness of $S$.
On $\overline\Omega$, the pulled-back metrics converge smoothly and are
bilipschitz with constants tending to one.  For large $i$, a minimizing
$\tau_i^{-1}g(t_i)$-geodesic from $x_{0,i}$ to $x_i$ cannot leave
$\Psi_i(\Omega)$: otherwise its first-exit segment would pull back to a
curve from a point near $z_0$ to $\partial\Omega$ of limiting length
greater than $D+1$, contradicting
\eqref{surf:eq:approximating-zero-distance}.  Pulling the whole geodesic
back and passing to the limit yields
\[
 d_Z(z,z_0)\le D.
\]
Since $z,z_0$ were arbitrary and the splitting is isometric,
$\operatorname{diam}_{g_\Sigma(-1)}\Sigma\le D$.  The complete surface
$\Sigma$ is therefore compact by Hopf--Rinow.  This excludes the Gaussian
alternative in \eqref{surf:eq:q-centred-splitting}, so
$\Sigma\cong\PP^1$.

Now $S=\{0\}\times\Sigma$ is compact.  Choose a precompact tubular
neighbourhood $U_0\Subset Z$ in which $S$ is the whole zero set of $G$.
Enlarge the convergence domain and, without changing notation, take
$\Psi_i$ to be defined on a neighbourhood of $\overline U_0$; reset
$H_i=\mathcal G_i\circ\Psi_i$ there.  Again \(0\) is a regular value of
\(H=G\), with target \(\C^m\simeq\R^{2m}\).
Part~\textup{(ii)} of \cref{ti:lem:P1-zero-graphs} gives a compact
boundaryless surface $S_i\subset U_0$, diffeomorphic to $S$, such that
\[
 \Psi_i(S_i)\subset F_q.
\]
The restriction $\Psi_i|_{S_i}$ is an embedding between real surfaces;
its image is open in $F_q$ by invariance of domain and closed because it
is compact.  Since $F_q\cong\PP^1$ is connected, the image is all of
$F_q$.

Finally, for the actual centres fixed before the realizing subsequence,
the alignment argument
\eqref{eq:packing-centre-alignment} and
\cref{prop:general-packing-package}, with \(q_i=q\), give a number
\(D<\infty\) for which the whole fibre lies in
\(B_{\tau_i^{-1}g(1-\tau_i)}(p_i,D)\).  For every finite \(A\ge0\),
\eqref{eq:general-packing-enlarged-curvature} gives a uniform curvature
bound on the radius-\(D+A\) ball.  This proves all assertions.
\end{proof}

We now combine the preceding ingredients to complete the fixed-fibre
theorem.

\begin{proposition}
\label{ti:prop:surface-base}
Let \(q\in Y^{\rm prod}\).
\begin{enumerate}[label=(\roman*),leftmargin=2.4em]
\item At every fixed limiting point based at \(q\), the underlying
K\"ahler metric flow of every tangent space is the smooth shrinking
cylinder
\[
 \bigl(\C^m\times\PP^1,
       g_{\rm E}+(-s)g_{S^2}\bigr)_{s<0},
 \qquad \Rc(g_{S^2})=\tfrac12g_{S^2}.
\]
\item The fixed fibre satisfies
\[
 \sup_{0\le t<1}\sup_{x\in F_q}
 (1-t)|\Rm(g(t))|_{g(t)}(x)<\infty.
\]
\item Fix \(x_0\in F_q\) and one anchored limiting point
\(\xi=(\mu_t)_{t<1}\) obtained from poles \((x_0,s_j)\),
\(s_j\nearrow1\).  For every \(\tau_i\searrow0\), choose
\(H_{2m+2}\)-centres \(p_i\) of
\(\mu_{i,-1}:=\mu_{1-\tau_i}\) with respect to
\(g_i(-1)=\tau_i^{-1}g(1-\tau_i)\) at backward scale one.  After a
realizing subsequence there is \(D<\infty\) such that
\[
 F_q\subset B_{\tau_i^{-1}g(1-\tau_i)}(p_i,D),
\]
and, for every finite \(A\ge0\), the curvature is uniformly bounded
on the radius-\(D+A\) ball.
\end{enumerate}
\end{proposition}

\begin{proof}
Fix first an arbitrary limiting point based at \(q\).  Its base point
and Type-I moment are automatic by \cref{ti:lem:label}.  For every
realizing scale sequence,
\cref{cor:fixed-point-two-models} gives the two smooth possibilities
\(\C^{m+1}\) and \(\C^m\times\PP^1\), with the \(q\)-centred projection
map.  The diameter and exact-level argument of
\cref{surf:prop:exact-level-topology} excludes the Gaussian possibility.
This proves (i).

For (ii), fix \(x_0\in F_q\) and an anchored limiting point
\(\xi=(\mu_t)_{t<1}\) as in \textup{(iii)}.  If the estimate failed,
there would be \(t_i\nearrow1\) and \(x_i\in F_q\) such that
\[
 (1-t_i)|\Rm(g(t_i))|(x_i)\longrightarrow\infty.
\]
Use the same scales \(\tau_i=1-t_i\), choose \(H_{2m+2}\)-centres of
\(\mu_{i,-1}=\mu_{1-\tau_i}\) with respect to \(g_i(-1)\), and pass to
a tangent space at this limiting point.  By
\cref{prop:general-packing-package}, with \(q_i=q\), the entire fibre is
contained in one fixed-radius pointed ball and the curvature is bounded
there.  Since
\[
 |\Rm(\tau_i^{-1}g(t_i))|(x_i)
 =\tau_i|\Rm(g(t_i))|(x_i),
\]
this contradicts the chosen sequence.  Smoothness on compact
subintervals supplies the estimate away from \(t=1\).

Finally apply \cref{prop:general-packing-package} to an arbitrary scale
sequence for this anchored limiting point, with \(q_i=q\).  It gives
the asserted whole-fibre containment and curvature bounds on every
fixed enlargement, proving \textup{(iii)}.
\end{proof}

\begin{remark}[The scale-separated branch in the moving-pole proof]
The strengthened packing proposition controls moving fibres whenever
their drift from the base label \(q_\pi(\xi)\) is
\(O(\sqrt{1-t_i})\).  It does not treat
\[
 (1-t_i)^{-1/2}d_\eta(q_i,q)\longrightarrow\infty.
\]
The two-scale argument in
\cref{ti:prop:moving-fibre} handles precisely the displayed branch by
passing first to a moving point-kernel limit and only then taking a
tangent space at the resulting limiting point.
\end{remark}

\paragraph{Return to the original time scale.}
The arguments in this section were written after the constant parabolic
normalization
\[
 \widetilde\omega(s)=T^{-1}\omega(Ts),\qquad
 \widetilde g(s)=T^{-1}g(Ts),\qquad
 \widetilde\eta=T^{-1}\eta.
\]
Undoing this normalization replaces $1-t$ by $T-t$ and leaves every
limiting-time rescaling unchanged:
\[
 \widetilde t_i:=\frac{t_i}{T},\qquad
 \widetilde\tau_i:=1-\widetilde t_i=\frac{T-t_i}{T},
 \qquad
 \widetilde\tau_i^{-1}
 \widetilde g(1+\widetilde\tau_i s)
 =(T-t_i)^{-1}g(T+(T-t_i)s).
\]
Moreover,
\[
 d_{\widetilde\eta}(q_i,q)\le A\sqrt{1-\widetilde t_i}
 \quad\Longleftrightarrow\quad
 d_\eta(q_i,q)\le A\sqrt{T-t_i}.
\]
Thus \cref{prop:general-bounded-drift-packing,ti:prop:surface-base}
hold in the original variables with
\[
 \tau_i=T-t_i,\qquad
 g_i(s)=\tau_i^{-1}g(T+\tau_i s),
\]
and with every occurrence of $1-t$ replaced by $T-t$.

\section{From fibrewise control to compact-local Type-I curvature}
\label{ti:sec:general-base}

Let
\[
                 \pi:X^{m+1}\longrightarrow Y^m
\]
be the limiting morphism of the semiample finite-time K\"ahler--Ricci
flow, and set $F_q=\pi^{-1}(q)$.  All local arguments take place in
relatively compact regular-base regions inside Zariski-open sets where
$\pi$ is biregularly a product with $\PP^1$ and $(\mathrm S)$ and
$(\mathrm Q)$ hold.  \Cref{prop:zariski-product-horizontal} supplies
both estimates on every such region; its upper horizontal bound is the
product-slice estimate \cite[Lemma~2.2]{FuZhang}.  The fibres are
connected.  We use the real normalization $g=2g_\omega$ and write
$\tau=T-t$.  \Cref{sec:surface-prototype} used the normalization $T=1$.
After the corresponding parabolic rescaling, its estimates hold here
with every occurrence of $1-t$ replaced by $T-t$.

The fixed-fibre and pointed-ball package
\cref{ti:prop:surface-base} proves, for every fixed $q$ in a
Zariski-open product region,
\begin{equation}\label{eq:short-fixed-fibre}
 \sup_{0\le t<T}\sup_{x\in F_q}
       (T-t)|\Rm(g(t))|_{g(t)}(x)<\infty.
\end{equation}
The constant in \eqref{eq:short-fixed-fibre} may a priori depend on $q$.
The issue in this section is the change in quantifiers from
\[
 \text{for every $q$ there is $C(q)$}
 \qquad\text{to}\qquad
 \text{there is $C_K$ for every $q\in K$}.
\]

\begin{proposition}
\label{ti:prop:moving-fibre}
Let \(K\Subset Y^{\rm rv}\) admit a finite cover by Zariski-open sets
\(U_\alpha\subset Y^{\rm rv}\) such that
\[
 \pi^{-1}(U_\alpha)\longrightarrow U_\alpha
 \quad\text{is biregularly isomorphic to}\quad
 U_\alpha\times\PP^1\longrightarrow U_\alpha.
\]
Then there is \(C_K<\infty\) such that
\begin{equation}\label{ti:eq:compact-local-TypeI}
 \sup_{0\le t<T}\sup_{x\in\pi^{-1}(K)}
 (T-t)|\Rm(g(t))|_{g(t)}(x)\le C_K.
\end{equation}
\end{proposition}

\subsection{Moving fibres and the drift scale}

Suppose that $q_i\to q$, $t_i\nearrow T$, and put
\begin{equation}\label{eq:drift-ratio}
 \tau_i=T-t_i,
 \qquad
 \Lambda_i=\frac{d_\eta(q_i,q)^2}{\tau_i}.
\end{equation}
The fixed-fibre pointed-ball package controls the moving fibres when
$\Lambda_i$ is bounded.

\begin{lemma}
\label{lem:bounded-drift-general}
Fix \(q\) in a Zariski-open product region.  Suppose that
\[
 q_i\longrightarrow q,
 \qquad t_i\nearrow T,
 \qquad d_\eta(q_i,q)^2\le A(T-t_i),
 \qquad 0\le A<\infty.
\]
Then every choice of points \(x_i\in F_{q_i}\) satisfies
\begin{equation}\label{eq:bounded-drift-bad-general}
 \limsup_{i\to\infty}
 (T-t_i)|\Rm(g(t_i))|_{g(t_i)}(x_i)<\infty.
\end{equation}
\end{lemma}

\begin{proof}
Suppose the conclusion fails.  After passing to a subsequence,
\[
 (T-t_i)|\Rm(g(t_i))|_{g(t_i)}(x_i)\longrightarrow\infty.
\]
Fix an anchored limiting point \(\xi=(\mu_t)_{t<T}\) based at $q$, and
put $\tau_i=T-t_i$.  Choose \(H_{2m+2}\)-centres $p_i$ of
\(\mu_{t_i}\) at backward scale \(\tau_i\) with respect to \(g(t_i)\),
and then pass to a subsequence realizing a tangent space at \(\xi\).
Since
\[
                      d_\eta(q_i,q)\le\sqrt A\sqrt{\tau_i},
\]
\Cref{lem:general-horizontal-transport} gives a point
$\bar p_i\in F_{q_i}$ at uniformly bounded horizontal distance from
$p_i$ for the metric $g_i(-1)=\tau_i^{-1}g(t_i)$.  By
\cref{rem:bounded-horizontal-anchor}, the entire fibre, not merely
$\bar p_i$, lies in a fixed ball about $p_i$, and curvature is uniformly
bounded on a fixed enlargement of that ball.  Hence
\[
 \tau_i|\Rm(g(t_i))|(x_i)
 =|\Rm(g_i(-1))|(x_i)\le C_A,
\]
which is a contradiction.
\end{proof}

If \(\Lambda_i\to\infty\), choose a strongly geodesically convex
neighbourhood \(V'\Subset Y^{\rm prod}\) of \(q\).  For curves whose
projections remain in $V'$, the lower quotient estimate $(\mathrm S)$
gives
\begin{equation}\label{eq:scale-separated-distance}
 \dist_{\tau_i^{-1}g(t_i)}(F_q,F_{q_i})
 \ge c\tau_i^{-1/2}d_\eta(q_i,q)
 =c\sqrt{\Lambda_i}\longrightarrow\infty.
\end{equation}
If the projected curve leaves $V'$, its initial segment from $q$ to
$\partial V'$ has a fixed positive $\eta$-length and gives an even
larger lower bound after rescaling.  Consequently,
\eqref{eq:scale-separated-distance} holds even when a minimizing curve
leaves the product chart.  The bad fibre therefore escapes every
bounded region in the rescalings defining a tangent space at a limiting
point based at $q$, and it must be studied using point kernels based at
the moving bad points.

\subsection{The moving-pole fibre-capture lemma}

The next lemma gives the required moving-pole compactness.  It applies
to any moving bad sequence; \cref{lem:bounded-drift-general} reduces the
application below to the scale-separated case.

\begin{lemma}
\label{ti:lem:moving-pole-fibre-capture}
Let \(U\subset Y^{\rm rv}\) be a Zariski-open product region.  Suppose
\[
 q_i\longrightarrow q_\infty\in U,
 \qquad x_i\in F_{q_i},
 \qquad t_i\nearrow T.
\]
Set \(\tau_i=T-t_i\) and assume that
\begin{equation}\label{eq:short-moving-bad}
             \tau_i|\Rm(g(t_i))|_{g(t_i)}(x_i)\longrightarrow\infty.
\end{equation}
Choose nested coordinate sets
\[
 q_\infty\in V_0\Subset V_a\Subset V_b\Subset V\Subset U
\]
and holomorphic coordinates \(w=(w^1,\ldots,w^m)\) on \(V\).  After
discarding finitely many terms, assume that \(q_i\in V_0\).

Then there are \(\sigma_j\searrow0\) and a subsequence
\(i(j)\nearrow\infty\) such that the pointed flows
\begin{equation}\label{eq:short-double-flow}
 \widehat g_j(t)
 =\frac{1}{\sigma_j\tau_{i(j)}}
   g\bigl(t_{i(j)}+\sigma_j\tau_{i(j)}t\bigr),
 \qquad
 -\frac{t_{i(j)}}{\sigma_j\tau_{i(j)}}\le t\le0,
\end{equation}
have the following properties.
\begin{enumerate}[label=\textup{(\roman*)},leftmargin=2.4em]
\item Equipped with the point conjugate heat kernels based at
$(x_{i(j)},0)$, the flows converge in $\mathbb F$, and smoothly on the
full regular spacetime, to a connected noncollapsed K\"ahler shrinking
soliton \(\mathcal Z\).  Write \((Z,g_Z)=\mathcal Z_{-1}\) and
\[
 d\nu_{Z,-1}=(4\pi)^{-(m+1)}e^{-f_Z}\,dV_{g_Z}
\]
for its distinguished shrinker probability measure.  Its normalized
Nash entropy satisfies
\[
 \mathcal N(Z):=
 \int_{\Reg Z}f_Z\,d\nu_{Z,-1}-(m+1)<0.
\]

\item On the localized convergence charts, the maps
\begin{equation}\label{eq:short-double-map}
 \widehat H_j
 =\frac{(w-w(q_{i(j)}))\circ\pi}
        {\sqrt{\sigma_j\tau_{i(j)}}}
\end{equation}
converge smoothly on compact subsets of the regular spacetime to a map
$H_Z$.  Its time-$-1$ restriction extends uniquely to a globally
Lipschitz holomorphic map $H_Z:Z\to\C^m$ satisfying on $\Reg Z$
\begin{equation}\label{eq:short-parallel-map}
 \nabla^{g_Z}dH_Z=0,
 \qquad
 c|a|^2\le |d(a\cdot H_Z)|_{g_Z}^2\le C|a|^2
 \quad(a\in\C^m).
\end{equation}

\item The time-$-1$ space is smooth.  More precisely, there exists
\(L\in\mathrm{GL}(m,\C)\), with no target translation, such that
\begin{equation}\label{eq:short-cylinder}
 Z\cong\C^m\times\PP^1,
 \qquad LH_Z=\operatorname{pr}_{\C^m},
\end{equation}
where $\PP^1$ carries the round metric normalized by
$\Rc(g_{S^2})=\frac12g_{S^2}$.

\item The convergence is realized on a regular exhaustion by embeddings
$\Psi_j$ for which
\begin{equation}\label{eq:short-map-convergence}
 \Psi_j^*\widehat g_j(-1)\longrightarrow g_Z,
 \qquad
 \widehat H_j\circ\Psi_j\longrightarrow H_Z
\end{equation}
smoothly on compact subsets of \(Z=\Reg Z\).  For every \(U'\Subset Z\)
and all sufficiently large \(j\),
\begin{equation}\label{eq:short-exact-zero}
 \Psi_j(U')\cap\widehat H_j^{-1}(0)
 =\Psi_j(U')\cap F_{q_{i(j)}}.
\end{equation}
\end{enumerate}
\end{lemma}

\begin{proof}
We divide the proof into four steps.

\smallskip
\noindent\emph{Step 1: the first moving-pole limit and its negative
entropy.}
Define
\begin{equation}\label{eq:short-first-rescaling}
 \widetilde g_i(s)=\tau_i^{-1}g(t_i+\tau_i s),
 \qquad -t_i/\tau_i\le s\le0,
\end{equation}
and equip these flows with the point conjugate heat kernels $\nu_s^i$
based at $(x_i,0)$.  The comparison with Perelman's $\mu$-functional
and its monotonicity give a uniform scale-one pointed Nash-entropy lower
bound \cite[Proposition~4.35 and equation~(4.34)]{BamlerStructure}.
Bamler's variance estimate gives uniform $H_{2m+2}$-concentration
\cite[Corollary~3.8, equation~(3.9)]{BamlerHeat}; the corresponding
metric-flow pairs belong to Bamler's compactness class by
\cite[Lemma~7.3]{BamlerCompactness}.  Since
$t_i/\tau_i\to\infty$, the backward intervals exhaust
$(-\infty,0]$.  Metric-flow-pair compactness
\cite[Theorems~7.4 and~7.6]{BamlerCompactness} gives
\begin{equation}\label{eq:short-first-limit}
 (X,\widetilde g_i(s),\nu_s^i,x_i)
 \longrightarrow
 (\mathcal X,g_\infty(s),\nu_{x_\infty;s},x_\infty).
\end{equation}
Write $\mathcal R^{\mathcal X}$ for the regular spacetime of
$\mathcal X$ and $\mathcal R_s^{\mathcal X}$ for its time-$s$ slice.
The convergence is smooth on the regular spacetime by
\cite[Theorems~2.5 and~2.14]{BamlerStructure}.  The parallel complex
structures converge there to a K\"ahler structure by the
parallel-tensor argument of
\cite[Theorem~2.5 and its proof]{HallgrenJian}.

The distinguished limiting kernel has a smooth density
\[
 d\nu_{x_\infty;s}=v_{\infty,s}\,dV_{g_\infty(s)}
\]
on the full regular spacetime, and
\begin{equation}\label{eq:short-density-package}
 v_{\infty,s}>0,
 \qquad
 \nu_{x_\infty;s}(\mathcal X_s\setminus\mathcal R_s^{\mathcal X})=0,
 \qquad
 \int_{\mathcal R_s^{\mathcal X}}v_{\infty,s}\,dV=1.
\end{equation}
Bamler's structure theorem identifies the smooth-convergence locus
with the regular spacetime \cite[Theorem~2.5]{BamlerStructure}, while
the distinguished-flow conclusion gives smooth density convergence
there \cite[Theorem~9.31(a)]{BamlerCompactness}.  The heat-kernel
description gives zero singular mass; full support and the strong
minimum principle give positivity; and the limiting conjugate flow has
total mass one.  See
\cite[the paragraph preceding Theorem~2.10 and the paragraph containing
equation~(2.1)]{BamlerStructure}.

Entropy convergence \cite[Theorem~2.10]{BamlerStructure} and
monotonicity give a constant $C_N<\infty$ such that
\begin{equation}\label{eq:short-entropy-window}
 -C_N\le\mathcal N_{x_\infty,0}(\vartheta)\le0,
 \qquad 0<\vartheta\le1,
\end{equation}
and hence a finite limit
\[
 \mathcal N_{x_\infty}(0)
 :=\lim_{\vartheta\searrow0}
     \mathcal N_{x_\infty,0}(\vartheta)>-\infty.
\]
Moreover,
\begin{equation}\label{eq:short-negative-entropy}
                       \mathcal N_{x_\infty}(0)<0.
\end{equation}
Let \(\varepsilon_{2m+2}>0\) be the dimensional constant in Bamler's
entropy \(\varepsilon\)-regularity theorem
\cite[Theorem~10.2]{BamlerHeat}.  If
\eqref{eq:short-negative-entropy} failed, choose a fixed small
\(\vartheta>0\) so that
$\mathcal N_{x_\infty,0}(\vartheta)>-\varepsilon_{2m+2}/2$.
Entropy convergence
then gives
$\mathcal N^{\widetilde g_i}_{x_i,0}(\vartheta)>-\varepsilon_{2m+2}$
for all large \(i\).  The entropy \(\varepsilon\)-regularity theorem
bounds the curvature scale at
$(x_i,0)$ from below by a positive multiple of $\sqrt\vartheta$.  This
contradicts
\[
 |\Rm(\widetilde g_i(0))|(x_i)
 =\tau_i|\Rm(g(t_i))|(x_i)\longrightarrow\infty.
\]

\smallskip
\noindent\emph{Step 2: the centred base map and Hessian dissipation.}
Choose $\chi\in C_c^\infty(V)$ with $\chi=1$ on $V_b$ and set
\begin{equation}\label{eq:short-first-map}
 H_i=\tau_i^{-1/2}
       \bigl[\chi\,(w-w(q_i))\bigr]\circ\pi.
\end{equation}
Here the bracket is defined on \(V\), extended by zero off \(V\), and
then pulled back to \(X\); its pullback is smooth because
\(\supp\chi\Subset V\).
Choose smooth base barriers $\rho_i$, constant outside $V_b$ and
vanishing only at $q_i$, such that, uniformly in $i$,
\begin{equation}\label{eq:short-barrier}
 c\,d_\eta(\,\cdot\,,q_i)^2\le\rho_i
 \le C\,d_\eta(\,\cdot\,,q_i)^2,
 \qquad
 |\Delta_{g(t)}(\rho_i\circ\pi)|\le C.
\end{equation}
The construction in \cref{def:squared-distance-barrier}, uniformly for
\(q_i\in V_0\), and \cref{ti:cor:basefunction} give these barriers and
the Laplacian bound.  Put
\[
 u_i:=\tau_i^{-1}\rho_i\circ\pi.
\]
Scaling gives
\[
 |\Delta_{\widetilde g_i(s)}u_i|
 =|\Delta_{g(t_i+\tau_i s)}(\rho_i\circ\pi)|\le C.
\]
Since \(\nu_s^i\rightharpoonup\delta_{x_i}\) as \(s\uparrow0\) and
\(u_i(x_i)=0\), conjugate-heat duality gives
\[
 \int_Xu_i\,d\nu_s^i
 =\int_s^0\!\int_X\Delta_{\widetilde g_i(r)}u_i\,d\nu_r^i\,dr
 \le C(-s).
\]
Thus, for every fixed \(A>0\) and all sufficiently large \(i\),
\begin{equation}\label{eq:short-barrier-moment}
 \int_X\tau_i^{-1}\rho_i\circ\pi\,d\nu_s^i\le C(-s),
 \qquad -A\le s<0.
\end{equation}

Let $p_{i,s}$ be an $H_{2m+2}$-centre for $\nu_s^i$ at backward scale
$-s$.  Apply Markov's inequality to
\eqref{eq:short-barrier-moment} and to the variance estimate at
$p_{i,s}$.  Choose the two implicit radii so that each exceptional set
has $\nu_s^i$-mass less than $1/4$.  Their complements intersect at a
point $y_{i,s}$ satisfying
\[
 d_{\widetilde g_i(s)}(y_{i,s},p_{i,s})\le C\sqrt{-s},
 \qquad
 d_\eta(\pi(y_{i,s}),q_i)\le C\sqrt{\tau_i(-s)}.
\]
For $s$ in a fixed compact negative-time interval, the second estimate
puts $\pi(y_{i,s})$ in $V_a$ for large $i$.  If a minimizing geodesic
from $y_{i,s}$ to $p_{i,s}$ first left $\pi^{-1}(V_b)$, the portion
whose projection crosses from $V_a$ to $\partial V_b$ would have,
by $(\mathrm S)$, $\widetilde g_i(s)$-length at least
\[
 c\tau_i^{-1/2}
 \dist_\eta(\overline V_a,Y\setminus V_b),
\]
contradicting the preceding $O(\sqrt{-s})$ upper bound.  The geodesic
therefore stays over $V_b$, and applying $(\mathrm S)$ once more gives
\begin{equation}\label{eq:short-centre-label}
 d_\eta(\pi(p_{i,s}),q_i)^2\le C\tau_i(-s).
\end{equation}
The same first-exit argument shows that every fixed pointed ball lies
over $V_b$ for large $i$.  To localize a compact regular spacetime chart
in \eqref{eq:short-first-limit}, cover it by finitely many regular
subballs with uniformly positive limiting kernel mass.  Smooth density
convergence and the centre second-moment bound place a point of each
approximating subball within $O(\sqrt{-s})$ of $p_{i,s}$, while smooth
metric convergence controls its diameter.  The first-exit estimate
localizes each subball, hence the whole chart.  Thus $\chi=1$ on every
chart under consideration.  On these localized charts,
conditions $(\mathrm S)$ and $(\mathrm Q)$ give the two-sided Gram estimate
\begin{equation}\label{eq:short-Gram}
 c|a|^2\le|d(a\cdot H_i)|_{\widetilde g_i(s)}^2
 \le C|a|^2,
 \qquad a\in\C^m.
\end{equation}
The same centre argument gives $|H_i(p_{i,s})|\le C\sqrt{-s}$.
On the localized charts \(\chi=1\), so each component of \(H_i\) is
harmonic.  The centre bound and gradient estimate give local \(C^0\)
bounds, and interior elliptic estimates on the smoothly converging
regular charts give \(C^\infty_{\rm loc}\)-compactness.  Consequently,
the maps converge smoothly on negative-time regular
charts to a stationary holomorphic harmonic map
\[
                  H:\mathcal R^{\mathcal X}\longrightarrow\C^m.
\]
The pointwise comparison $|H_i|^2\le C\tau_i^{-1}\rho_i\circ\pi$
and \eqref{eq:short-density-package} give
\begin{equation}\label{eq:short-limit-moment}
 \int_{\mathcal R_s^{\mathcal X}}|H|^2\,d\nu_{x_\infty;s}\le C(-s).
\end{equation}

Let $\varphi_i$ be a real or imaginary component of $H_i$.  Scaling and
the base-function estimate \cref{ti:cor:basefunction} give
\begin{equation}\label{eq:short-map-errors}
 |\nabla\varphi_i|\le C,
 \qquad
 |\Delta\varphi_i|\le C\sqrt{\tau_i}.
\end{equation}
Writing $d\nu_s^i=v_{i,s}\,dV_{\widetilde g_i(s)}$, the point-kernel
potential is defined by
\[
 v_{i,s}=(4\pi(-s))^{-(m+1)}e^{-f_{i,s}}.
\]
The point-kernel estimate
\cite[Proposition~5.13, equation~(5.14)]{BamlerHeat} gives
\[
 (-s)\int_X\bigl(|\nabla f_{i,s}|^2
                 +R_{\widetilde g_i(s)}\bigr)\,d\nu_s^i
 \le m+1.
\]
The scalar-curvature maximum principle on the original compact flow
gives $R_{g(t)}\ge-C$, and hence
$R_{\widetilde g_i(s)}\ge-C\tau_i$.  Since
$\nabla\log v_{i,s}=-\nabla f_{i,s}$, the last two estimates give
\begin{equation}\label{eq:short-Fisher}
 \int|\nabla\log v_{i,s}|^2\,d\nu_s^i
 \le\frac{C}{-s}+C\tau_i,
 \qquad -1\le s<0.
\end{equation}
For $E_i(s)=\int_X|\nabla\varphi_i|^2\,d\nu_s^i$, we now record the
error term in the conjugate-heat Bochner identity explicitly.  The
function $\varphi_i$ is independent of the rescaled time $s$, while
$\partial_s\widetilde g_i=-2\Rc(\widetilde g_i)$.  Hence
\begin{equation}\label{eq:short-pointwise-Bochner}
 (\partial_s-\Delta)|\nabla\varphi_i|^2
 =-2|\nabla^2\varphi_i|^2
  -2\langle\nabla\varphi_i,\nabla\Delta\varphi_i\rangle.
\end{equation}
For every smooth function $u$, conjugate-heat duality gives
\begin{equation}\label{eq:short-conjugate-duality}
 \frac{d}{ds}\int_Xu\,d\nu_s^i
 =\int_X(\partial_s-\Delta)u\,d\nu_s^i.
\end{equation}
For completeness, this identity follows from
\[
 \partial_sv_{i,s}=-\Delta v_{i,s}+R v_{i,s},
 \qquad
 \partial_s dV_{\widetilde g_i(s)}=-R\,dV_{\widetilde g_i(s)},
\]
which imply
$\partial_s(d\nu_s^i)=-(\Delta v_{i,s})dV_{\widetilde g_i(s)}$.
Likewise, because $\varphi_i$ is time independent,
$\partial_s|\nabla\varphi_i|^2
=2\Rc(\nabla\varphi_i,\nabla\varphi_i)$.  Therefore
\begin{align*}
 E_i'(s)
 &=2\int_X\Rc(\nabla\varphi_i,\nabla\varphi_i)\,d\nu_s^i
   -\int_X|\nabla\varphi_i|^2\Delta v_{i,s}\,dV\\
 &=2\int_X\Rc(\nabla\varphi_i,\nabla\varphi_i)\,d\nu_s^i
   -\int_X\Delta|\nabla\varphi_i|^2\,d\nu_s^i.
\end{align*}
The Ricci term in the scalar Bochner formula cancels the first term in
the last display, leaving exactly
\eqref{eq:short-pointwise-Bochner} integrated against $d\nu_s^i$.
Applying \eqref{eq:short-conjugate-duality} to
$u=|\nabla\varphi_i|^2$ and integrating by parts on the compact
manifold $X$, we obtain
\begin{align}
 E_i'(s)
 &=-2\int|\nabla^2\varphi_i|^2\,d\nu_s^i+\mathcal E_i(s),
 \label{eq:short-Bochner}\\
 \mathcal E_i(s)
 &:=2\int(\Delta\varphi_i)^2\,d\nu_s^i \notag\\
 &\quad
 +2\int\Delta\varphi_i
   \langle\nabla\varphi_i,\nabla\log v_{i,s}\rangle\,d\nu_s^i.
 \label{eq:short-Bochner-error-definition}
\end{align}
Indeed, the only term not already displayed in
\eqref{eq:short-pointwise-Bochner} is
\begin{align*}
 -2\int_X\langle\nabla\varphi_i,
       \nabla\Delta\varphi_i\rangle v_{i,s}\,dV
 &=2\int_X(\Delta\varphi_i)^2\,d\nu_s^i\\
 &\quad
 +2\int_X\Delta\varphi_i
   \langle\nabla\varphi_i,\nabla\log v_{i,s}\rangle\,d\nu_s^i,
\end{align*}
which is precisely \eqref{eq:short-Bochner-error-definition}.
Since $\nu_s^i$ is a probability measure, \eqref{eq:short-map-errors}
and Cauchy--Schwarz, followed by \eqref{eq:short-Fisher}, yield
\begin{align}
 |\mathcal E_i(s)|
 &\le C\tau_i+C\sqrt{\tau_i}(-s)^{-1/2}.
 \label{eq:short-Bochner-error}
\end{align}
Here the first term comes from
$2\int(\Delta\varphi_i)^2\,d\nu_s^i$, and the second follows from
\[
 \left|\int\Delta\varphi_i
  \langle\nabla\varphi_i,\nabla\log v_{i,s}\rangle\,d\nu_s^i\right|
 \le C\sqrt{\tau_i}
 \left(\frac{C}{-s}+C\tau_i\right)^{1/2}.
\]
Thus the error is integrable at $s=0$.  Integrating
\eqref{eq:short-Bochner} over $[-1,-\varepsilon]$, using
$E_i(-\varepsilon)\ge0$ and $E_i(-1)\le C$, and then letting
$\varepsilon\searrow0$, gives
\begin{equation}\label{eq:short-uniform-dissipation}
 \int_{-1}^{0}\int_X|\nabla^2\varphi_i|^2\,d\nu_s^i\,ds\le C.
\end{equation}
Smooth metric, map, and density convergence on compact regular
spacetime charts, followed by lower semicontinuity and exhaustion,
therefore gives, after summing over the real and imaginary components,
\begin{equation}\label{eq:short-dissipation}
 \int_{-1}^{0}\int_{\mathcal R_s^{\mathcal X}}
 |\nabla^2H|^2\,d\nu_{x_\infty;s}\,ds<\infty.
\end{equation}

\smallskip
\noindent\emph{Step 3: the secondary tangent, the quantitative
diagonal, and the round cylinder.}
Choose a tangent flow $\mathcal Z$ of $\mathcal X$ at $x_\infty$ at
scales $\sigma_j\searrow0$.  Bamler's tangent theorem makes it a metric
shrinking soliton
\cite[Theorem~2.6 and Addendum~2.7]{BamlerStructure}.  Let
 $\mathcal R^{\mathcal Z}$ denote its full negative-time regular
spacetime.  Let
$\mathcal X^{(j)}$ be the rescaling of the first limit by
$\sigma_j^{-1}$.  The entropy window \eqref{eq:short-entropy-window}
gives
\[
 \mathcal N^{\mathcal X^{(j)}}_{x_\infty,0}(1)
 =\mathcal N^{\mathcal X}_{x_\infty,0}(\sigma_j)\ge-C_N.
\]
At the fixed scale $\sigma_j$, entropy convergence transfers this lower
bound to a defining smooth first-level sequence.  Choose its index so
that the scale-one entropy is at least \(-C_N-1\).  The resulting
diagonal is uniformly \((1,C_N+1)\)-noncollapsed in the sense of
Chan--Ma--Zhang.  On each compact negative-time interval, their theorem
\cite[Theorem~1.2]{ChanMaZhangSmooth} gives smooth metric convergence
on the full regular spacetime of the prescribed tangent space.
Bamler's distinguished-flow theorem gives smooth
convergence of the densities there
\cite[Theorem~9.31(a)]{BamlerCompactness}, and parallel-tensor
compactness supplies the limiting K\"ahler structure.

Choose a regular spacetime exhaustion
$\mathcal K_j\Subset\mathcal R^{\mathcal Z}$ and secondary embeddings
\begin{equation}\label{eq:short-secondary-embeddings}
 \Phi_j:\mathcal K_j\longrightarrow\mathcal R^{\mathcal X},
 \qquad s\circ\Phi_j=\sigma_jt.
\end{equation}
On the secondary rescalings set
\begin{equation}\label{eq:short-secondary-density}
 g^{(j)}(t)=\sigma_j^{-1}g_\infty(\sigma_jt),
 \qquad \nu_t^{(j)}=\nu_{x_\infty;\sigma_jt},
 \qquad v_t^{(j)}=\sigma_j^{m+1}v_{\infty,\sigma_jt}.
\end{equation}
Then $d\nu_t^{(j)}=v_t^{(j)}dV_{g^{(j)}(t)}$.  The limiting secondary
density is smooth and positive on $\mathcal R^{\mathcal Z}$, has zero
singular mass, and has total mass one, by the same distinguished-kernel,
full-support, and strong-minimum-principle argument used in
\eqref{eq:short-density-package}.

The order of the double diagonal is important.  First fix
$\sigma_j$, $\mathcal K_j$, $\Phi_j$, and a required post-rescaling
accuracy $\varepsilon_j\searrow0$.  Let $\Theta_{i,j}$ be first-level
regular convergence embeddings on a neighbourhood of
$\Phi_j(\mathcal K_j)$.  Only then choose $i(j)$ so large that, after
composition with $\Phi_j$ and secondary rescaling, the metric,
time-vector, density, complex-structure, and $\mathbb F$ errors are at
most $\varepsilon_j$.  Put
\[
                       \Psi_j=\Theta_{i(j),j}\circ\Phi_j.
\]
On each fixed finite regular spacetime coordinate atlas,
$C_{\rm par}^k$ denotes the integer-order parabolic norm
\begin{equation}\label{eq:short-parabolic-norm}
 \|u\|_{C_{\rm par}^k(\mathcal K)}
 :=\sum_{2a+|\beta|\le k}
   \|\partial_t^a\nabla^\beta u\|_{C^0(\mathcal K)},
\end{equation}
with the analogous definition for tensors.  Thus one time derivative
counts as two spatial derivatives.  This notation records quantitative
smooth convergence on compact regular spacetime charts; it does not
mean merely topological local compactness.
On the fixed atlas of $\mathcal K_j$, impose at the same time
\begin{equation}\label{eq:short-map-error}
 \bigl\|H_{i(j)}\circ\Psi_j-H\circ\Phi_j\bigr\|
 _{C_{\rm par}^{j+1}(\mathcal K_j)}
 \le\varepsilon_j\sqrt{\sigma_j}
\end{equation}
before multiplying the maps by $\sigma_j^{-1/2}$.  We also choose
$i(j)$ so that
\begin{equation}\label{eq:short-entropy-tolerance}
 \left|
 \mathcal N^{\widetilde g_{i(j)}}_{x_{i(j)},0}(\sigma_j)
 -\mathcal N^{g_\infty}_{x_\infty,0}(\sigma_j)
 \right|\le\varepsilon_j.
\end{equation}
These choices produce the doubly rescaled flows
\eqref{eq:short-double-flow}.  To make the metric-flow-pair statement
explicit, define the rescaled point kernels and their densities by
\begin{equation}\label{eq:short-double-kernels}
 \widehat\nu_t^j:=\nu_{\sigma_jt}^{i(j)},
 \qquad
 d\widehat\nu_t^j=\widehat v_j(t)\,dV_{\widehat g_j(t)},
 \qquad
 \widehat v_j(t):=\sigma_j^{m+1}v_{i(j),\sigma_jt}.
\end{equation}
Thus, on the full regular spacetime, the diagonal converges smoothly in
its metrics, time vectors, complex structures, distinguished densities,
and normalized base maps; the pointed metric-flow pairs also converge
in $\mathbb F$.  Parabolic
invariance gives
\begin{equation}\label{eq:short-entropy-diagonal}
 \mathcal N^{\widehat g_j}_{x_{i(j)},0}(1)
 =\mathcal N^{\widetilde g_{i(j)}}_{x_{i(j)},0}(\sigma_j)
 \longrightarrow\mathcal N_{x_\infty}(0)<0.
\end{equation}
The diagonal is $H_{2m+2}$-concentrated by scale invariance, is uniformly
noncollapsed by \eqref{eq:short-entropy-diagonal}, and its backward
intervals exhaust $(-\infty,0]$ because
\[
              \frac{t_{i(j)}}{\sigma_j\tau_{i(j)}}\longrightarrow\infty.
\]
Bamler's intrinsic-completion theorem
\cite[Theorem~2.4(c)]{BamlerStructure} says that the regular-slice
distance is the finite intrinsic length distance and that its completion
is \(Z\).  Hence the regular slice is path connected and \(Z\) is
connected.  The quantitative diagonal above verifies the hypotheses of
\HZanalyticref, which supplies only the normal klt structure and
metric--analytic compatibility used below.  This describes the tangent
at this stage; its smooth product structure is proved below.

Let $d\nu_{Z,-1}$ be the distinguished probability measure on the
time-$-1$ slice and write its density as
\begin{equation}\label{eq:short-shrinker-density}
 d\nu_{Z,-1}=(4\pi)^{-(m+1)}e^{-f_Z}\,dV_{g_Z}.
\end{equation}
The normalized Nash entropy of this distinguished shrinker measure is
\begin{equation}\label{eq:short-shrinker-entropy-definition}
 \mathcal N(Z)
 :=\int_{\Reg Z}f_Z\,d\nu_{Z,-1}-(m+1).
\end{equation}
This notation is deliberately different from Perelman's
$\nu$-functional and from the conjugate measures $\nu_t$.  Bamler's
tangent-entropy theorem \cite[Theorem~2.11]{BamlerStructure}, together
with \eqref{eq:short-entropy-diagonal}, gives
\begin{equation}\label{eq:short-tangent-entropy}
 \mathcal N(Z)
 =\lim_{\vartheta\searrow0}
   \mathcal N^{g_\infty}_{x_\infty,0}(\vartheta)
 =\mathcal N_{x_\infty}(0)<0.
\end{equation}

Put $H^{(j)}=\sigma_j^{-1/2}H$ on the secondary rescalings.  The moment
bound becomes
\[
                  \int|H^{(j)}|^2\,d\nu_t^{(j)}\le C(-t).
\]
The limiting distinguished density is smooth and strictly positive on
compact regular charts, so this is a local unweighted $L^2$ bound.
Interior estimates give smooth local convergence to a holomorphic map
$H_Z$.  For $0<A<B<\infty$, exact Hessian scaling gives
\begin{align}
 &\int_{-B}^{-A}\int
 |\nabla^2H^{(j)}|^2\,d\nu_t^{(j)}\,dt \notag\\
 &\qquad=
 \int_{-B\sigma_j}^{-A\sigma_j}\int
 |\nabla^2H|^2\,d\nu_{x_\infty;s}\,ds
 \longrightarrow0.
 \label{eq:short-Hessian-vanishing}
\end{align}
Here \eqref{eq:short-dissipation} is used in the last step: the
integral of its $L^1$ integrand over the shrinking interval
$[-B\sigma_j,-A\sigma_j]$ tends to zero.

We spell out what \eqref{eq:short-Hessian-vanishing} says about the
limit.  If $\nabla dH_Z$ were nonzero at a regular spacetime point,
continuity would give a positive lower bound on a smaller precompact
parabolic cylinder.  The smooth positive limiting density also has a
positive lower bound there.  Smooth convergence of the metrics, maps,
and densities would then give a fixed positive lower bound for the
weighted spacetime $L^2$ norm of $\nabla dH^{(j)}$, contrary to
\eqref{eq:short-Hessian-vanishing}.  Hence
\begin{equation}\label{eq:short-Hessian-zero}
                         \nabla dH_Z=0
\end{equation}
on the full regular spacetime.  Componentwise, this means exactly that
the real Hessians of every real and imaginary component of $H_Z$
vanish; there is no distinction between this statement and ``the
Hessian of $H_Z$ is zero.''

The Gram estimate \eqref{eq:short-Gram} passes to the limit.  Its upper
bound makes $H_Z$ Lipschitz for the intrinsic length distance on
$\Reg Z$.  Bamler's intrinsic-completion theorem
\cite[Theorem~2.4(c)]{BamlerStructure} identifies this distance with
the restriction of the metric-space distance and identifies $Z$ with
the metric completion of $\Reg Z$.  Since $\C^m$ is complete,
$H_Z$ therefore has a unique Lipschitz extension to all of $Z$.  Each
component is locally bounded and holomorphic on $\Reg Z$.  The weak
Riemann extension theorem for the normal complex space supplied by
\HZanalyticref[\textup{(a)}]\ then extends it
uniquely as a holomorphic function across $\Sing Z$; see
\cite[Chapter~7, Section~4]{GrauertRemmertSheaves}.  The metric and
holomorphic extensions agree
by continuity and density.  This proves \eqref{eq:short-parallel-map}.

We now verify all hypotheses needed for the residual-model theorem.
Bamler's structure theory gives completeness, the intrinsic
regular-slice distance and metric completion, and the soliton equation
\cite[Theorems~2.4(c), 2.6 and Addendum~2.7]{BamlerStructure}.
The sequence-level theorem \HZanalyticref\ gives normality, klt
singularities, equality of the analytic and metric regular loci, the
bounded-potential K\"ahler current, and the local ambient lower bound.
Finally, \eqref{eq:short-shrinker-density} is the normalized finite-mass
distinguished soliton measure.  In particular, its potential satisfies
\begin{equation}\label{eq:short-soliton-normalization}
 \Rc(g_Z)+\nabla^2f_Z=\frac12g_Z
 \quad\hbox{on }\Reg Z,
 \qquad
 \int_{\Reg Z}e^{-f_Z}\,dV_{g_Z}=(4\pi)^{m+1}.
\end{equation}
The singular product theorem and residual-curve classification
\cref{thm:singular-product,ti:prop:P1-models} therefore give, after one
fixed complex-linear normalization $L\in\mathrm{GL}(m,\C)$ and without
target translation,
\[
 Z\cong\C^m\times\Sigma,
 \qquad LH_Z=\operatorname{pr}_{\C^m},
 \qquad
 \Sigma\cong\C\ \hbox{or}\ \PP^1.
\]
The first alternative is Gaussian and is excluded by the tangent
entropy identity \eqref{eq:short-tangent-entropy}.  Thus
\eqref{eq:short-cylinder} holds.  In particular, the intermediate
normal klt space supplied by the sequence-level package is now proved
to be the smooth K\"ahler product $\C^m\times\PP^1$.

\smallskip
\noindent\emph{Step 4: exact labels and exact zero levels.}
After the preceding extraction, refine the diagonal once more, without
changing the prescribed tangent space, so that
$\sigma_j^{-1/2}H\circ\Phi_j\to H_Z$ on $\mathcal K_j$ to the required
$C_{\rm par}^j$ accuracy.  Multiplication of
\eqref{eq:short-map-error} by $\sigma_j^{-1/2}$ leaves an error at most
$\varepsilon_j$.  Together with the metric diagonal, this proves
\eqref{eq:short-map-convergence}.  The localization argument puts
the composed charts inside $V_b\subset\{\chi=1\}$.  On these charts the
approximating maps are therefore exactly
\eqref{eq:short-double-map}, not merely cutoff approximations.  Since
$w$ is a coordinate system,
\[
 (w-w(q_{i(j)}))\circ\pi=0
 \quad\Longleftrightarrow\quad
 \pi=q_{i(j)}.
\]
This proves \eqref{eq:short-exact-zero}.  No target translation has
entered, and applying $L\in\mathrm{GL}(m,\C)$ does not change the zero set.
\end{proof}

\subsection{The area contradiction and applications}

\begin{proof}[Proof of \cref{ti:prop:moving-fibre}]
Suppose that \eqref{ti:eq:compact-local-TypeI} fails.  Smoothness on compact
subintervals gives $x_i\in\pi^{-1}(K)$ and $t_i\nearrow T$ such that,
with
\[
 q_i=\pi(x_i),
 \qquad \tau_i=T-t_i,
\]
one has
\begin{equation}\label{eq:short-global-bad}
                  \tau_i|\Rm(g(t_i))|(x_i)\longrightarrow\infty.
\end{equation}
After passing to a subsequence, $q_i\to q_\infty\in K$.  Pass to a
further subsequence on which
$d_\eta(q_i,q_\infty)^2/\tau_i$ is either bounded or tends to infinity.
The bounded case contradicts \cref{lem:bounded-drift-general}; hence we
are in the second case.
Choose a product chart containing $q_\infty$ and apply
\cref{ti:lem:moving-pole-fibre-capture}.

In the time-$-1$ slice of the secondary tangent, set
\[
                       S=\{0\}\times\PP^1.
\]
This is the compact zero set of $LH_Z$.  Choose precompact tubular
neighbourhoods
\[
                  S\subset U_0\Subset U_1\Subset\Reg Z
\]
so that $S$ is the entire zero set in $\overline U_1$.  The lower
differential bound in \eqref{eq:short-parallel-map} makes zero a regular
value.  On $U_1$ set
\[
 G_j=L\widehat H_j\circ\Psi_j,
 \qquad
 G=LH_Z=\operatorname{pr}_{\C^m}.
\]
Since \(G_j\to G\) smoothly, \(0\) is a regular value of \(G\), and
\(S\) is the entire zero set in \(\overline U_1\),
\cref{ti:lem:P1-zero-graphs}\textup{(ii)} gives compact embedded spheres
\[
 S_j\subset(L\widehat H_j\circ\Psi_j)^{-1}(0)\cap U_0
\]
that are normal graphs over \(S\) and converge smoothly to it.
Since \(L\) is invertible, \eqref{eq:short-exact-zero}, applied with
\(U'=U_0\), gives
\[
                       \Psi_j(S_j)\subset F_{q_{i(j)}}.
\]
The fibre $F_{q_{i(j)}}$ is a connected smooth $\PP^1$.  The image
$\Psi_j(S_j)$ is a compact embedded real two-manifold in that fibre.
Its inclusion is locally a diffeomorphism, so it is open in the fibre
by invariance of domain; it is closed by compactness.  Connectedness of
the fibre therefore gives
\begin{equation}\label{eq:short-whole-fibre}
                       \Psi_j(S_j)=F_{q_{i(j)}}.
\end{equation}
Let $\gamma_j:S\to S_j$ be the normal-graph parametrizations.  Smooth
metric convergence and smooth convergence of the graph sections give
\[
 \gamma_j^*\Psi_j^*\widehat g_j(-1)
 \longrightarrow g_Z|_S
 \quad\hbox{in }C^\infty(S).
\]
Together with \eqref{eq:short-whole-fibre}, this gives
\begin{align}
 \Area_{\widehat g_j(-1)}(F_{q_{i(j)}})
 &=\int_S dA_{\gamma_j^*\Psi_j^*\widehat g_j(-1)} \notag\\
 &\longrightarrow\int_SdA_{g_Z|_S}
 =\Area_{g_{S^2}}(\PP^1)=8\pi.
 \label{eq:short-area-finite}
\end{align}
Indeed, $\Rc(g_{S^2})=\frac12g_{S^2}$ means that its Gauss curvature
is $1/2$.  Thus Gauss--Bonnet gives
\[
 \frac12\Area_{g_{S^2}}(\PP^1)
 =2\pi\chi(\PP^1)=4\pi.
\]

On the other hand, restriction of the limiting class to every regular
fibre gives
$[\omega(t)]|_{F_q}=2\pi(T-t)c_1(\PP^1)$ and hence the exact identity
\begin{equation}\label{eq:short-exact-area}
                         \Area_{g(t)}(F_q)=8\pi(T-t).
\end{equation}
This is the general-$T$ form of
\cref{ti:eq:P1-fibre-class,ti:eq:P1-area-kahler}, with the additional
factor two from $g=2g_\omega$.
At time $t=-1$ in \eqref{eq:short-double-flow}, the original time is
$t_{i(j)}-\sigma_j\tau_{i(j)}$, and hence
\[
 T-\bigl(t_{i(j)}-\sigma_j\tau_{i(j)}\bigr)
 =\tau_{i(j)}(1+\sigma_j).
\]
Since the area of a real surface is multiplied by $\lambda$ when the
ambient metric is multiplied by $\lambda$, equations
\eqref{eq:short-double-flow} and \eqref{eq:short-exact-area} give
\begin{align}
 \Area_{\widehat g_j(-1)}(F_{q_{i(j)}})
 &=\frac{1}{\sigma_j\tau_{i(j)}}
   8\pi\tau_{i(j)}(1+\sigma_j) \notag\\
 &=8\pi\frac{1+\sigma_j}{\sigma_j}
 \longrightarrow\infty.
 \label{eq:short-area-diverges}
\end{align}
This contradicts \eqref{eq:short-area-finite} and proves the
proposition.
\end{proof}

\subsection{Sharp two-sided fibre diameter from Type-I regularity}
\label{ti:sec:diameter-from-typeI}

We now derive the sharp fibre-diameter estimate from the compact-local
Type-I curvature bound.  This step is independent of the limiting-entropy
uniformization argument.  At the Type-I scale, curvature control and
noncollapsing give uniform ambient bounded geometry.  The normalized
base maps then have uniformly regular two-dimensional levels, and the
exact fibre area converts their uniform local disks into a global
intrinsic-diameter bound.

We first isolate the quantitative level-set statement used below.

\begin{lemma}
\label{ti:lem:uniform-level-disks}
Let \(d\ge3\), \(1\le k<d\), \(r_b>0\), and \(\Lambda\ge1\).  There
exist constants
\[
 a_0=a_0(d,k,r_b,\Lambda)>0,
 \qquad
 R_0=R_0(d,k,r_b,\Lambda)<\infty
\]
with the following property.  Let \((M^d,h)\) be a Riemannian manifold
and \(p\in M\).  Suppose that \(\exp_p\) is a diffeomorphism from
\(B_{T_pM}(0,2r_b)\) onto its image and that, in these exponential
coordinates,
\begin{equation}\label{ti:eq:level-chart-bounds}
 \Lambda^{-1}g_{\rm E}
 \le \exp_p^*h\le\Lambda g_{\rm E},
 \qquad
 \|\exp_p^*h\|_{C^2(B(0,2r_b))}\le\Lambda.
\end{equation}
Let
\[
 H:\exp_p\bigl(B_{T_pM}(0,2r_b)\bigr)\longrightarrow\R^k
\]
be smooth, with \(H(p)=0\), and suppose that
\begin{equation}\label{ti:eq:level-map-C2}
 |dH|_h+|\nabla dH|_h\le\Lambda
\end{equation}
and
\begin{equation}\label{ti:eq:level-map-rank}
 \Lambda^{-1}|a|^2
 \le
 \left|\sum_{\alpha=1}^k a_\alpha dH^\alpha\right|_h^2
 \le\Lambda |a|^2
 \qquad(a\in\R^k).
\end{equation}
Then the component of \(H^{-1}(0)\) through \(p\) contains an embedded
closed \((d-k)\)-disk \(\mathcal D_p\) such that
\begin{align}
 \mathcal D_p
 &\subset
 B_{H^{-1}(0),\,h|_{H^{-1}(0)}}(p,R_0),
 \label{ti:eq:level-disk-intrinsic-ball}\\
 \Vol_{h|_{H^{-1}(0)}}(\mathcal D_p)
 &\ge a_0.
 \label{ti:eq:level-disk-volume}
\end{align}
The ball in \eqref{ti:eq:level-disk-intrinsic-ball} is intrinsic to the
regular level.
\end{lemma}

\begin{proof}
Set
\[
 \widetilde H=H\circ\exp_p:
 B_{T_pM}(0,2r_b)\longrightarrow\R^k,
 \qquad A=d\widetilde H_0.
\]
By \eqref{ti:eq:level-map-rank}, \(A\) is surjective.  With respect to
the orthogonal decomposition
\[
 T_pM=E\oplus N,
 \qquad E=\ker A,
 \qquad N=E^\perp,
\]
the restriction \(B:=A|_N:N\to\R^k\) is an isomorphism and
\(\|B^{-1}\|\) is bounded in terms of \(\Lambda\).

The coordinate bounds
\cref{ti:eq:level-chart-bounds,ti:eq:level-map-C2} give
\[
 \|D^2\widetilde H\|_{C^0(B(0,r_b))}\le C(d,k,\Lambda).
\]
Choose \(0<\rho_1<r_b/10\), depending only on the displayed data, so
small that
\begin{equation}\label{ti:eq:level-contraction-choice}
 \|B^{-1}\|
 \sup_{|\zeta|\le4\rho_1}|D\widetilde H_\zeta-A|
 \le\frac14.
\end{equation}
For \(u\in E\), \(|u|\le2\rho_1\), consider the map on \(N\)
\[
 \mathcal T_u(v):=v-B^{-1}\widetilde H(u+v).
\]
After decreasing \(\rho_1\) once more, \(A(u)=0\), Taylor's theorem,
and \eqref{ti:eq:level-contraction-choice} show that \(\mathcal T_u\)
preserves \(B_N(0,\rho_1)\) and is \(1/2\)-Lipschitz.  It has a unique
fixed point \(v=\varphi(u)\), and the
fixed-point equation is precisely
\[
                         \widetilde H(u+\varphi(u))=0.
\]
The parameter-dependent contraction theorem gives a smooth map
\[
 \varphi:B_E(0,2\rho_1)\longrightarrow N,
 \qquad \varphi(0)=0,
 \qquad \|D\varphi\|_{C^0}\le C(d,k,\Lambda).
\]
Consequently
\[
 \mathcal D_p:=
 \left\{
  \exp_p\bigl(u+\varphi(u)\bigr):
  u\in E,\ |u|\le\rho_1
 \right\}
\]
is an embedded closed \((d-k)\)-disk in the zero level through \(p\).
The metric comparison in \eqref{ti:eq:level-chart-bounds} and the
bound for \(D\varphi\) show that radial graph paths have uniformly
bounded length, which gives
\eqref{ti:eq:level-disk-intrinsic-ball}.  The lower metric comparison
gives a uniform positive lower bound for the induced volume of the
graph over \(B_E(0,\rho_1)\), proving
\eqref{ti:eq:level-disk-volume}.
\end{proof}

\begin{proposition}
\label{ti:prop:compact-local-fibre-diameter}
Let \(K\Subset Y^{\rm rv}\) admit a finite cover by Zariski-open sets
\(U_\alpha\subset Y^{\rm rv}\) such that
\[
 \pi^{-1}(U_\alpha)\longrightarrow U_\alpha
 \quad\text{is biregularly isomorphic to}\quad
 U_\alpha\times\PP^1\longrightarrow U_\alpha.
\]
Then there are constants \(0<c_K\le C_K<\infty\) such that, for every
\(q\in K\) and every \(0\le t<T\),
\begin{equation}\label{ti:eq:compact-local-two-sided-diameter}
 c_K\sqrt{T-t}
 \le
 \operatorname{diam}_{(X,g(t))}F_q
 \le
 \operatorname{diam}_{(F_q,g(t)|_{F_q})}F_q
 \le C_K\sqrt{T-t}.
\end{equation}
The first diameter is measured by the ambient distance on \(X\), while
the last is intrinsic to the induced fibre metric.
\end{proposition}

\begin{proof}
Choose a compact set \(K^+\Subset Y^{\rm rv}\), still covered by
finitely many product regions, such that
\[
                         K\subset\operatorname{int}K^+.
\]
Refine this cover by finitely many holomorphic coordinate
neighbourhoods
\[
 V_\alpha\Subset W_\alpha\Subset O_\alpha
 \Subset U_{\beta(\alpha)},\qquad
 K\subset\bigcup_\alpha V_\alpha,
\]
and choose injective charts
\(w_\alpha:O_\alpha\to\C^m\).  All constants below are uniform over
this finite refinement.

\smallskip
\noindent\emph{Step 1: ambient bounded geometry at the Type-I scale.}
Put
\[
                         \tau=T-t,
 \qquad h_t=\tau^{-1}g(t).
\]
Applying \cref{ti:prop:moving-fibre} on \(K^+\) gives
\begin{equation}\label{ti:eq:diameter-TypeI-input}
 \sup_{\pi^{-1}(K^+)}|\Rm(h_t)|_{h_t}
 =\sup_{\pi^{-1}(K^+)}\tau|\Rm(g(t))|_{g(t)}
 \le C_{K^+}.
\end{equation}
For \(t\) sufficiently close to \(T\), consider
\[
 h_t(s):=\tau^{-1}g(t+\tau s),
 \qquad -\theta\le s\le0,
\]
where \(0<\theta<1\) is fixed.  Since
\(T-(t+\tau s)=\tau(1-s)\), the same Type-I estimate gives
\begin{equation}\label{ti:eq:diameter-parabolic-curvature}
 |\Rm(h_t(s))|\le \frac{C_{K^+}}{1-s}\le C_{K^+}
 \qquad(-\theta\le s\le0)
\end{equation}
wherever the projection remains in \(K^+\).

The parabolic Schwarz estimate gives
\[
                  \pi^*(\tau^{-1}\eta)\le C h_t(s).
\]
Consequently a fixed-radius \(h_t(s)\)-ball centered over \(K\)
projects into an \(\eta\)-ball of radius \(C\sqrt\tau\).  Since \(K\)
has positive distance from \(Y\setminus K^+\), all fixed-radius
parabolic balls used below lie over \(K^+\) once \(\tau\) is small.

Choose \(r_*>0\) so that \(4r_*^2<\theta\) and
\(C_{K^+}\le r_*^{-2}\).  For \(t\) sufficiently close to \(T\), we
also have \(t-4r_*^2\tau\ge0\).  After shrinking \(r_*\) once more, for every
\(x\in\pi^{-1}(K)\) the doubled backward parabolic neighbourhood
\[
 \bigcup_{-4r_*^2\le s\le0}
 B_{h_t(s)}(x,2r_*)\times\{s\}
\]
lies over \(K^+\) and has uniformly bounded curvature by
\eqref{ti:eq:diameter-parabolic-curvature}.  Perelman's
no-local-collapsing theorem
\cite[Sections~4.1--4.2]{PerelmanEntropy} therefore applies.  At the
original radius \(r_*\sqrt\tau\), it gives
\begin{equation}\label{ti:eq:diameter-noncollapse}
 \Vol_{h_t}B_{h_t}(x,r_*)
 \ge\kappa r_*^{2m+2}
\end{equation}
for \(x\in\pi^{-1}(K)\) and all sufficiently small \(\tau\).  The
curvature and volume bounds, together with the
Cheeger--Gromov--Taylor injectivity-radius estimate
\cite[Theorem~4.7]{CheegerGromovTaylor}, give
\begin{equation}\label{ti:eq:diameter-injectivity}
                       \inj_{h_t}(x)\ge i_0>0.
\end{equation}
Choose
\[
 0<r_b<\min\{r_*/8,i_0/4\}.
\]
Local Shi derivative estimates \cite{ShiDerivatives}, applied to the
preceding rescaled parabolic neighbourhoods, give uniform curvature
derivative bounds on the smaller \(r_b\)-balls.  Together with
\eqref{ti:eq:diameter-injectivity}, these bounds provide one
\(\Lambda\ge1\), independent of \(x\) and \(t\), for which the
exponential coordinates have uniform \(C^2\)-bounds on \(B(0,4r_b)\);
in particular, \eqref{ti:eq:level-chart-bounds} holds on
\(B(0,2r_b)\).

\smallskip
\noindent\emph{Step 2: uniform local fibre disks.}
Fix \(q\in V_\alpha\) and \(p\in F_q\).  In the corresponding product
region, the fixed-radius ball needed below lies in
\(\pi^{-1}(W_\alpha)\) for all sufficiently small \(\tau\).  This
follows from the Schwarz estimate and the uniform positive distance
from \(\overline V_\alpha\) to \(Y\setminus W_\alpha\).  On
\(\pi^{-1}(O_\alpha)\) define
\begin{equation}\label{ti:eq:normalized-local-base-map}
 \mathcal H_{q,t}
 :=\tau^{-1/2}
 \bigl(w_\alpha-w_\alpha(q)\bigr)\circ\pi.
\end{equation}
Regard its real and imaginary components as a map to \(\R^{2m}\).
Since \(h_t^{-1}=\tau g(t)^{-1}\), the scale factors cancel in the Gram
matrix of \(d\mathcal H_{q,t}\).  The two quotient estimates
\(c\eta\le S_t\le C\eta\) and compactness of the finite coordinate
cover therefore give
\begin{equation}\label{ti:eq:normalized-map-uniform-rank}
 c_0|a|^2
 \le
 \left|d(a\mathbin{\cdot}\mathcal H_{q,t})\right|_{h_t}^2
 \le C_0|a|^2
 \qquad(a\in\R^{2m}).
\end{equation}

Every component of \(\mathcal H_{q,t}\) is harmonic because the map is
holomorphic.  Since \(\mathcal H_{q,t}(p)=0\), the upper differential
bound in \eqref{ti:eq:normalized-map-uniform-rank} gives a uniform
\(C^0\)-bound on \(B_{h_t}(p,4r_b)\).  Interior elliptic estimates in
the bounded-geometry charts from Step~1 then give
\begin{equation}\label{ti:eq:normalized-map-Hessian}
              |\nabla d\mathcal H_{q,t}|_{h_t}\le C_1
\end{equation}
on \(B_{h_t}(p,2r_b)\).  On that ball, injectivity of \(w_\alpha\)
gives
\begin{equation}\label{ti:eq:local-exact-fibre}
        \mathcal H_{q,t}^{-1}(0)=F_q
\end{equation}
there.  Applying \cref{ti:lem:uniform-level-disks} with
\(d=2m+2\) and \(k=2m\) produces, for every \(p\in F_q\), an embedded
disk \(\mathcal D_{p,t}\subset F_q\) and constants \(a_0,R_0>0\),
independent of \(q,p,t\), such that
\begin{align}
 \mathcal D_{p,t}
 &\subset B_{(F_q,h_t|_{F_q})}(p,R_0),
 \label{ti:eq:fibre-local-disk-ball}\\
 \Area_{h_t}(\mathcal D_{p,t})&\ge a_0.
 \label{ti:eq:fibre-local-disk-area}
\end{align}

\smallskip
\noindent\emph{Step 3: the ambient lower diameter bound.}
The K\"ahler form associated with the real metric \(h_t\) is
\(\Omega_t=2\tau^{-1}\omega(t)\).  Restriction of the evolving class
to \(F_q\cong\PP^1\) gives
\begin{equation}\label{ti:eq:normalized-fibre-area-lower-step}
              \int_{F_q}\Omega_t=\Area_{h_t}(F_q)=8\pi.
\end{equation}
Choose \(0<r_0<i_0/2\).  If
\(\operatorname{diam}_{(X,h_t)}F_q<r_0\), then, after fixing
\(p\in F_q\), the entire fibre lies in the normal ball
\(B_{h_t}(p,r_0)\).  On this contractible ball the closed form
\(\Omega_t\) is exact.  Stokes' theorem would give
\[
                       8\pi=\int_{F_q}\Omega_t=0,
\]
which is impossible.  Hence
\begin{equation}\label{ti:eq:normalized-ambient-lower}
             \operatorname{diam}_{(X,h_t)}F_q\ge r_0.
\end{equation}

\smallskip
\noindent\emph{Step 4: the intrinsic upper diameter bound.}
The same class calculation gives the exact normalized area
\begin{equation}\label{ti:eq:normalized-fibre-area}
                         \Area_{h_t}(F_q)=8\pi.
\end{equation}
Let \(x,y\in F_q\), and let
\(\gamma:[0,L]\to F_q\) be a minimizing unit-speed geodesic for the
induced metric, joining \(x\) to \(y\).  Put
\[
 N=\left\lfloor\frac{L}{3R_0}\right\rfloor,
 \qquad p_j=\gamma(3jR_0),
 \qquad j=0,\ldots,N.
\]
The intrinsic balls
\(B_{(F_q,h_t|_{F_q})}(p_j,R_0)\) are pairwise disjoint: since
\(\gamma\) is minimizing,
\[
 d_{(F_q,h_t|_{F_q})}(p_j,p_\ell)
 =3|j-\ell|R_0.
\]
Hence the
disks \(\mathcal D_{p_j,t}\) from Step~2 are pairwise disjoint, and
\cref{ti:eq:fibre-local-disk-area,ti:eq:normalized-fibre-area} give
\[
                         (N+1)a_0\le8\pi.
\]
Since \(L<3R_0(N+1)\), it follows that
\[
 L\le\frac{24\pi R_0}{a_0}.
\]
The points \(x,y\) were arbitrary, and therefore
\begin{equation}\label{ti:eq:normalized-intrinsic-upper}
 \operatorname{diam}_{(F_q,h_t|_{F_q})}F_q
 \le D_0:=\frac{24\pi R_0}{a_0}.
\end{equation}
Combining \eqref{ti:eq:normalized-ambient-lower} and
\eqref{ti:eq:normalized-intrinsic-upper} with the elementary inequality
between ambient and intrinsic distance gives
\begin{equation}\label{ti:eq:normalized-two-sided-diameter}
 r_0
 \le\operatorname{diam}_{(X,h_t)}F_q
 \le\operatorname{diam}_{(F_q,h_t|_{F_q})}F_q
 \le D_0
\end{equation}
for \(q\in K\) and all sufficiently small \(\tau\).

\smallskip
\noindent\emph{Step 5: scaling and the initial time interval.}
Because \(h_t=\tau^{-1}g(t)\), distances satisfy
\(d_{h_t}=\tau^{-1/2}d_{g(t)}\).  Thus
\eqref{ti:eq:normalized-two-sided-diameter} proves the desired estimate
on the near-\(T\) interval \([T-\tau_0,T)\).

On the compact time interval \([0,T-\tau_0]\), the metrics \(g(t)\)
form a smooth family.  In the finite product refinement, smoothness on
\(\overline{V_\alpha}\times\PP^1\times[0,T-\tau_0]\) uniformly bounds
the induced fibre metrics above by fixed product metrics and gives a
uniform intrinsic upper bound.  In every product region choose two
distinct constant fibre
sections.  Their ambient \(g(t)\)-distance is a positive continuous
function on \(\overline{V_\alpha}\times[0,T-\tau_0]\), and hence has a
positive minimum.  Taking the minimum over the finite cover yields
constants \(e_K,D_K>0\) such that
\[
 e_K
 \le\operatorname{diam}_{(X,g(t))}F_q
 \le\operatorname{diam}_{(F_q,g(t)|_{F_q})}F_q
 \le D_K
\]
for \(q\in K\) and \(0\le t\le T-\tau_0\).  Since
\(\sqrt{\tau_0}\le\sqrt{T-t}\le\sqrt T\) on this interval, choose
\(c_K\le e_K/\sqrt T\) and \(C_K\ge D_K/\sqrt{\tau_0}\), enlarging the
near-\(T\) constants if necessary.  The compact-time and near-\(T\)
estimates then prove
\eqref{ti:eq:compact-local-two-sided-diameter}.
\end{proof}

\subsection{Proofs of the main theorems and cylindrical corollary}
\label{ti:sec:main-proofs}

\begin{proof}[Proof of \cref{thm:intro-type-I}]
By GAGA, \(E\) is algebraic, and \(\PP(E)\to Y\) is biregularly
isomorphic to a product
over a Zariski-open neighbourhood of every point of $Y$.  Compactness
gives a finite product cover, and the bundle projection has
$Y^{\rm rv}=Y$.  Apply the moving-fibre estimate
\cref{ti:prop:moving-fibre} with $K=Y$.  Since $\pi^{-1}(Y)=X$, the
resulting estimate is exactly
\eqref{ti:eq:P1-regular-TypeI-front}.
Apply
\cref{ti:prop:compact-local-fibre-diameter} with \(K=Y\).  Its
conclusion is exactly
\eqref{ti:eq:P1-fibre-diameter-front}.
\end{proof}

\begin{proof}[Proof of \cref{ti:cor:projective-P1}]
Let $\xi$ be a fixed limiting point.  By
\cref{lem:automatic-label}, it is based at a unique point
$q=q_\pi(\xi)\in Y$, with no further hypothesis.  After forgetting the
distinguished conjugate heat measures, assertion \textup{(i)} of
\cref{ti:prop:surface-base} identifies the underlying K\"ahler metric
flow of every tangent space at $\xi$ with the round shrinking cylinder
\eqref{ti:eq:intro-cylinder}.
\end{proof}
{\footnotesize
\begin{sloppypar}

\end{sloppypar}
}

\end{document}